\documentclass[a4paper,10pt,reqno]{amsart}

\usepackage{amsmath}
\usepackage{amsfonts}
\usepackage{amssymb}
\usepackage{amsthm}
\usepackage{mathrsfs}
\usepackage[shortlabels]{enumitem}
\usepackage{graphicx}
\usepackage{color}
\usepackage[latin1]{inputenc}
\usepackage{MnSymbol}
\usepackage{enumitem}
\usepackage{extpfeil}
\usepackage{mathtools}
\usepackage{url}
\usepackage[margin=1.1in]{geometry}
\usepackage{fancyhdr}
\usepackage[all,cmtip]{xy}
\usepackage{array,multirow}
\usepackage[x11names]{xcolor}
\usepackage{tikz}
\usepackage{pgf}
\usepackage{appendix}
\usepackage{float}
\RequirePackage{hyperref}

\usetikzlibrary{trees}
\usetikzlibrary{shapes}
\usetikzlibrary{arrows}

\definecolor{marronfonce}{cmyk}{0,0.06,0.20,0.6}
\colorlet{orangeclair}{DarkOrange3!85}
\colorlet{orangefonce}{DarkOrange3}
\colorlet{bleuclair}{DeepSkyBlue4!50}
\colorlet{bleufonce}{DeepSkyBlue4}

\tikzstyle{rdc}=[rectangle,draw,thick, rounded corners=4pt,inner sep=2mm] % Forme des boîtes aux angles arrondis
\tikzstyle{rdc}=[rectangle,draw,thick, rounded corners=4pt,inner sep=2mm] % Forme des boîtes aux angles arrond

\hypersetup{pdfpagemode={UseOutlines},
bookmarksopen=true,
bookmarksopenlevel=0,
hypertexnames=false,
colorlinks=true, % colorie les liens
citecolor=teal, % couleur des citations
linkcolor=blue, % couleur des references internes
urlcolor=magenta, % couleur des URL
pdfstartview={FitV},
unicode,
breaklinks=true,}

\newcommand{\CC}{{\mathbf C}}

\newcommand{\RR}{{\mathbf R}}

\newcommand{\ZZ}{{\mathbf Z}}
\newcommand{\HH}{{\mathbb H}}
\newcommand{\del}{{\partial}}

\def\ker{{\rm{ker}}}

\def\dim{{\rm{dim}}}
\def\sgn{{\rm{sgn}}}

\def\Z{{\mathbf{Z}}}

\def\nnn{{\underline{n}}}
\def\Area{{\mathrm{Area}}}

\newcommand{\mmm}{{\underline{m}}}

\newcommand{\ggreater}{>}

\newtheorem{thm}{Theorem} % Theorem in the introduction
\newtheorem{cor}[thm]{Corollary} % Corollary in the introduction
\theoremstyle{plain}
\newtheorem{theorem}{Theorem}[section]

\newtheorem{claim}[theorem]{Claim}
\newtheorem{conjecture}[theorem]{Conjecture}
\newtheorem{corollary}[theorem]{Corollary}

\newtheorem{lemma}[theorem]{Lemma}
\newtheorem{proposition}[theorem]{Proposition}
\newtheorem{question}[theorem]{Question}
\newtheorem*{question*}{Question}
\newtheorem{addendum}[theorem]{Addendum}

\theoremstyle{definition}
\newtheorem{remark}[theorem]{Remark}

\newtheorem*{acknowledgements*}{Acknowledgements}
\newtheorem{example}[theorem]{Example}
\newtheorem{definition}[theorem]{Definition}
\newtheorem*{notation*}{Notation}
\newtheorem*{convention*}{Convention}
\newtheorem*{initial*}{Initial step of the induction}

\numberwithin{equation}{section}

\title{Cone-equivalent nilpotent groups with different Dehn functions}

\author{Claudio Llosa Isenrich} 
\address{Faculty of Mathematics, Karlsruhe Institute of Technology, Englerstra\ss e 2, 76131 Karlsruhe, Germany}
\email{claudio.llosa@kit.edu}

\author{Gabriel Pallier} 
\address{Sorbonne Universit\'e, IMJ-PRG, 75252 Paris Cedex 05, France.}
\email{gabriel@pallier.org}

\author{Romain Tessera}
\address{Institut de Math\'ematiques de Jussieu-PRG, Universit\'e Paris-Diderot, CNRS, Case 7012, 75205 Paris Cedex 13, France}
\email{romatessera@gmail.com}

\thanks{ C.L.I. was supported by a public grant as part of the FMJH, by the Max Planck Institute for Mathematics and by the Lise Meitner fellowship M2811-N of the Austrian Science Fund (FWF).\\
\indent G.P. was supported by the European Research Council (ERC Starting Grant 713998 GeoMeG `\emph{Geometry of Metric Groups}').}


\begin{document}

\begin{abstract}
For every $k\geqslant 3$, we exhibit a simply connected $k$-nilpotent Lie group $N_k$ whose Dehn function behaves like  $n^k$, while the Dehn function of its associated Carnot graded group $\mathsf{gr}(N_k)$ behaves like $n^{k+1}$. This property and its consequences allow us to reveal three new phenomena. First, since those groups have uniform lattices, this provides the first examples of pairs of finitely presented groups with bilipschitz asymptotic cones but with different Dehn functions. 
The second surprising feature of these groups is that for every even integer $k \geqslant 4$ the centralized Dehn function of $N_k$ behaves like $n^{k-1}$ and so has a different exponent than the Dehn function. 
This answers a question of Young. Finally, we turn our attention to sublinear bilipschitz equivalences (SBE). Introduced by Cornulier, these are maps between metric spaces inducing bi-Lipschitz homeomorphisms between their asymptotic cones. These can be seen as weakenings of quasiisometries where the additive error is replaced by a sublinearly growing function $v$. We show that a $v$-SBE between  $N_k$  and $\mathsf{gr}(N_k)$ must satisfy $v(n)\succcurlyeq  n^{1/(2k + 2)}$, strengthening the fact that those two groups are not quasiisometric. This is the first instance where an explicit lower bound is provided for a pair of SBE groups. 
\end{abstract}

\maketitle

\section{Introduction}
The goal of this work is to improve our understanding of the large scale geometry of simply connected nilpotent Lie groups and, more specifically, of an asymptotic invariant called the Dehn function, which encodes fundamental geometric and algebraic information on the group.
Given a simply connected Lie group $G$ equipped with a left-invariant Riemannian metric, the Dehn function $\delta_G(r)$ is the smallest real number such that every rectifiable loop $\gamma$ of length $\leqslant r$ in $G$ admits a filling by a Lipschitz disc of area $\leqslant \delta_G(r)$. An important feature of the Dehn function is the invariance of its asymptotics under quasi-isometry  (see \S \ref{sec:dehnvarious}). The study of filling functions of nilpotent Lie groups is a very difficult subject that has been deeply explored by Gromov, who initiated it \cite{AsInv,GroNilp}, and other authors (e.g.\  \cite{Allcock,GerstenRileyHolt, PittetIsopNil, young2006scaled,YoungFillingNil,Wenger}). 
The main result of this paper should be seen as a contribution to this important subject. However, one of its key applications and the choice of groups studied can be better appreciated in the wider context of the study of the large scale geometry of simply connected nilpotent groups. We will thus start by recalling known facts and central open problems in this area. A reader who directly wants to proceed to our results can go straight to \S \ref{subsec:Intro-central-products}.

\subsection{Background on the large scale geometry of nilpotent groups}

A motivation for focussing on simply connected Lie groups rather than discrete groups is that every finitely generated nilpotent group maps with finite kernel onto a lattice in a unique simply connected nilpotent Lie group (called its real Malcev completion) \cite{MalcevNilvarietes}. 
It follows that the quasi-isometry classification of finitely generated nilpotent groups reduces to that of simply connected nilpotent Lie groups, which is conjectured to have the following very neat formulation (see \cite[Conjecture 19.114]{CornulierOber}). 
\begin{conjecture}\label{conj:Nilpotent} 
Two simply connected nilpotent Lie groups are quasi-isometric if and only if they are isomorphic.  
\end{conjecture}
 
Conjecture \ref{conj:Nilpotent} is more commonly stated in the discrete case: two finitely generated torsion-free nilpotent groups are quasi-isometric if and only if they have isomorphic real Malcev completions (this is mentioned as an open question in \cite{FM}). It is tempting to ask whether a quasi-isometry between two such groups implies that they are commensurable. This turns out to be false. Indeed, for any ring $R$, let $\HH_d(R)$ denote the $d$-dimensional Heisenberg group over that ring. Then $\HH_d(\mathbf Z[\sqrt{2}])$ and $\HH_d(\mathbf Z)^2$ are both (uniform) lattices in $\HH_d(\mathbf R)^2$, therefore they are quasi-isometric. However, their rational Malcev completions are not isomorphic, which is equivalent to saying that the groups are not commensurable \cite{MalcevNilvarietes}.

The lowest-dimensional example of a pair of simply connected nilpotent Lie groups for which Conjecture \ref{conj:Nilpotent} is still open occurs in dimension $5$ (for a complete overview of the state of the art in dimension $\leqslant 6$ we refer to \cite{CornulierOber}). This shows that we are still far from having a complete proof even in low dimensions. On the other hand there is ample evidence pointing towards the veracity of Conjecture \ref{conj:Nilpotent}, with one of the first striking results being Pansu's Theorem. In order to state it precisely we need to recall the notions of a Carnot graded Lie algebra (resp.\ a Carnot graded nilpotent Lie group).

 We denote $\gamma_1 \mathfrak{g} = \mathfrak{g}$, $\gamma_{i+1} \mathfrak{g}=[\mathfrak{g},\gamma_{i}\mathfrak{g}]$ the lower central series of the Lie algebra $\mathfrak{g}$ (resp. $\gamma_i G$ the lower central series of the group $G$). A Lie algebra $\mathfrak{g}$ (resp. group $G$) has step\footnote{Various terminologies exist in the literature: $s$-step nilpotent, $s$-nilpotent, or nilpotent of class $s$.} $s$ if $s$ is the smallest integer such that  $\gamma_{s+1} \mathfrak{g}=\{0\}$ (resp. $\gamma_{s+1} G=\left\{1\right\}$). 
 The lower central series gives rise to a filtration of $\mathfrak{g}$ in the sense that $[\gamma_{i} \mathfrak{g},\gamma_{j} \mathfrak{g}]\subset \gamma_{i+j} \mathfrak{g}$. 
 
 A Lie algebra is called \emph{Carnot gradable} if this filtration comes from a grading, i.e.\ a decomposition $\mathfrak{g}=\bigoplus_i m_i$ satisfying $\gamma_{j} \mathfrak{g}=\bigoplus_{i\geqslant j} m_i$ and $[m_i,m_j]\subset m_{i+j}$; such a grading is called a Carnot grading.
It is always possible to associate a Carnot graded Lie algebra $\mathsf{gr}(\mathfrak{g})$ to any nilpotent Lie algebra $\mathfrak{g}$ by letting 
$\mathsf{gr}(\mathfrak{g})=\bigoplus_{i\geqslant 1} m_i$ for $m_i=\gamma_{i} \mathfrak{g}/\gamma_{i+1} \mathfrak{g}$ and defining the Lie bracket in the obvious way to make $m_i$ a grading (see \S \ref{subsec:Carnot-gradings} for more details).
We denote $\mathsf{gr}(G)$ the simply connected nilpotent Lie group whose Lie algebra is $\mathsf{gr}(\mathfrak{g})$. The pair $(\mathsf{gr}(G), m_1)$ is then called a Carnot-graded group (some authors say stratified group). 
Observe that $\mathsf{gr}(\mathfrak{g})$ has the same dimension and step as $\mathfrak{g}$.

We say for convenience that two groups are cone equivalent if their asymptotic cones with respect to any given non-principal ultrafilter are bilipschitz\footnote{Note that our notion of cone equivalence differs from Cornulier's notion of cone equivalence between maps in \cite{CornulierCones11}.}. It is easy to see that two groups that are quasi-isometric are cone equivalent. 
Pansu's fundamental Theorem provides a complete classification of simply connected nilpotent groups up to cone equivalence.

\begin{theorem}[{\cite{PanCBN,Breuillard}} and {\cite{PansuCCqi}}] 
Let $G$ be a simply connected nilpotent Lie group, equipped with a left-invariant word metric $d$ associated to some compact generating subset. 
Then $(G,d/n)$ converges in the Gromov-Hausdorff topology to $\mathsf{gr}(G)$ equipped with a left-invariant sub-Finsler metric $d_c$ as $n\to \infty$.
Moreover, if two simply connected nilpotent Lie groups $G$ and $G'$ have bilipschitz asymptotic cones (e.g.\ if they are quasi-isometric), then $\mathsf{gr}(G)$ and $\mathsf{gr}(G')$ are isomorphic.
\label{thm:Pansu}
\end{theorem}

In particular, Theorem \ref{thm:Pansu} shows that two simply connected nilpotent Lie groups $G$ and $G'$ are cone equivalent if and only if $\mathsf{gr}(G)$ and $\mathsf{gr}(G')$ are isomorphic. Another beautiful piece of work on this subject is due to Shalom. He shows that Betti numbers are invariant under quasi-isometry among finitely generated nilpotent groups  \cite[Theorem 1.2]{ShalomHarmonic}. This enabled Shalom to produce the first examples of cone equivalent nilpotent groups that are not quasi-isometric. To close this quick survey we mention that Sauer \cite{SauerHom} strengthened Shalom's Theorem by proving the quasi-isometry invariance of the real cohomology algebra of such groups, thereby extending the class of cone equivalent pairs that can be distinguished up to quasi-isometry. 

These results show that for nilpotent groups, being cone equivalent is indeed weaker than being quasi-isometric, thereby giving credit to Conjecture \ref{conj:Nilpotent}. Recently, Cornulier introduced the following generalization of quasiisometries, which provides a quantitative version of cone equivalence for nilpotent groups \cite{cornulier2017sublinear}.

\begin{definition}[Cornulier]
A map between two metric spaces $F:(X,d_X)\to (Y,d_Y)$ is called a sublinear bilipschitz equivalence (SBE) if there exists a non-decreasing map $v:\RR_+\to \RR_+$ that is sublinear (i.e.\ $\lim_{t\to \infty} v(t)/t=0$) and $x_0\in X$, $y_0\in Y$, and $M\geqslant 1$, such that for all $r\geqslant 0$ and $x,x'\in B(x_0,r)$
\[M^{-1}d_X(x,x')-v(r)\leqslant d_Y(F(x),F(x')) \leqslant Md_X(x,x')+v(r),\]
and for all $y\in B(y_0,r)$ there exists $x\in X$ such that 
$d(F(x),y)\leqslant v(r).$
\label{def:SBE}
\end{definition}

SBEs are designed to induce bilipschitz homeomorphisms between asymptotic cones \cite[Proposition 2.13.]{CornulierCones11}\footnote{Note that they were called ``cone bilipschitz equivalences'' in \cite{CornulierCones11}.}. 
In \cite{CornulierCones11} Cornulier observes that Pansu's Theorem can be reformulated in terms of the existence of a SBE between $G$ and $\mathsf{gr}(G)$ (see Corollary \ref{cor:weak-Pansu}). 
On the other hand quasi-isometries correspond to the special case of $v$ being bounded. Hence the study of simply connected nilpotent Lie groups up to sublinear bilipschitz equivalence is a way to interpolate between the conjectural quasi-isometric classification and Pansu's Theorem.

In this paper we shall focus on a certain family of pairs of cone equivalent nilpotent groups. We know by Shalom that these groups are not quasi-isometric. But proving that they have different Dehn functions allows us to derive a much stronger statement: we obtain an explicit  asymptotic lower bound on the possible functions $v$ such that these groups are $v$-SBE (see \S \ref{secIntro:SBE} for precise statements).

We now proceed to a detailed description of our results.

\subsection{Central products and non-Carnot gradable nilpotent groups}
\label{subsec:Intro-central-products}

Most examples of simply connected nilpotent Lie groups that one might readily think of are Carnot graded. In particular this is the case for all groups of dimension at most $5$, with two exceptions, and for all $2$-nilpotent groups.  
However, this observation is rather misleading, as the predominance of Carnot gradable groups turns out to be a low-dimensional phenomenon. Indeed, in high dimensions being Carnot gradable is a rather rare phenomenon and it even seems reasonable to go as far as to say that a generic nilpotent Lie group will not be Carnot gradable. This emphasizes the importance of understanding nilpotent Lie groups that are not Carnot gradable, even if the tools at hand are much more limited.

One way of obtaining interesting examples of nilpotent Lie groups that are not Carnot gradable is a general construction called a {\em central product}. 
Given two Lie algebras $ \mathfrak{k}$ and $ \mathfrak{l}$, central subspaces $ \mathfrak{z}\subset  \mathfrak{k}$ and $ \mathfrak{z'}\subset  \mathfrak{l}$, and an isomorphism $\theta:  \mathfrak{z}\to  \mathfrak{z}'$, we define the central product $\mathfrak{g}= \mathfrak{k}\times_\theta  \mathfrak{l}$ to be the quotient of the direct product $\mathfrak{k}\times \mathfrak{l}$ by the central ideal $\{(z,-\theta(z)); z\in  \mathfrak{z}\}$.

Let $k$ (resp.\ $l$) be the maximal integer such that $ \mathfrak{z}$ (resp.\ $ \mathfrak{z}'$) is contained in the $k$-th (resp. $l$-th) term of the lower central series of $ \mathfrak{k}$ (resp.\ $ \mathfrak{l}$). If $k>l\geqslant 2$ and $\mathfrak{k}$ and $\mathfrak{l}$ are Carnot graded with 1-dimensional centres, it is easy to check that the Lie algebra $ \mathfrak{g}$ is not Carnot gradable and that $\mathsf{gr}(\mathfrak{g})$ is isomorphic to the direct product $\mathfrak{k}\times (\mathfrak{l}/ \mathfrak{z}')$. 

To introduce the explicit family of groups that will form our main object of study, we start by recalling a classical class of Carnot graded Lie algebras.

\begin{definition}
The standard filiform $p$-nilpotent Lie algebra $\mathfrak l_p$ is the step $(p-1)$ nilpotent Lie algebra of dimension $p$ with basis 
$\left\{ x_1,x_2,\ldots, x_p\right\}$ satisfying $[x_1,x_i]=x_{i+1}$ for $2\leqslant i\leqslant p-1$ and $[x_i,x_j]=0$ for $1<i\leqslant j\leqslant p$ or if $(i,j)=(1,p)$. 
\end{definition}

We denote by $L_p$ the corresponding simply connected Lie group. The semi-direct product $\Lambda_p\cong \ZZ^{p-1} \rtimes_{\phi} \ZZ$, with $\phi(x_1)(x_i)= x_{i+1}$, $2\leqslant i \leqslant p-1$, and $\phi(x_1)(x_p)=0$, defines a lattice in $L_p$, where we denote by $x_1$ the generator of $\ZZ$ and by $x_2,\dots, x_p$ the generators of $\ZZ^{p-1}$. This provides us with a natural presentation $\mathcal{P}(\Lambda_p)$ of $\Lambda_p$ which we will use later.\footnote{Note that using the same notation for the generators of the lattice $\Lambda_p$ and the generators of the Lie algebra $\mathfrak{l}_p$ will not cause any confusion, as it will always be clear from context which one of the two we are working with.}

If $p\geqslant 3$, the center of $ \mathfrak{l}_p$ is the one dimensional subalgebra spanned by $z:=x_p$. 
For $3\leqslant q\leqslant p$ we define the Lie algebra $ \mathfrak{g}_{p,q}$ to be the central product (defined unambiguously) of $ \mathfrak{l}_p$ and $ \mathfrak{l}_q$.
 We let $G_{p,q}$ be the corresponding simply connected Lie group.
$G_{p,q}$ admits a uniform lattice $\Gamma_{p,q}$ which is simply the central product of $\Lambda_p$ and $\Lambda_q.$ As a concrete example, observe that $G_{3,3}$ and $\Gamma_{3,3}$ are the $5$-dimensional Heisenberg group $\HH_5(\mathbf R)$ and its integer lattice $\HH_5(\mathbf Z)$ respectively.  

The groups $G_{p,q}$ and their corresponding Lie algebras $\mathfrak{g}_{p,q}$ will form our main object of study in this paper; in particular the cases when $q=p-1$ or $q=p$. A key motivation for this is that the Lie algebras $\mathfrak{g}_{p,q}$ for $q,p\geqslant 3$ are Carnot gradable if and only if $q=p$ and thus that the corresponding groups $G_{p,q}$ are not isomorphic to their asymptotic cones if $q\neq p$. Indeed, for $2< q <p$, the associated Carnot-graded Lie algebra $\mathsf{gr}(\mathfrak{g}_{p,q})$ is isomorphic to the direct product $ \mathfrak{l}_p\times  \mathfrak{l}_{q-1}$ (note that $ \mathfrak{l}_2=\mathbf R^2$) and thus $\mathsf{gr}(G_{p,q})\cong L_p \times L_{q-1}$. Moreover, we observe that $G_{p,q}$ and thus $\mathsf{gr}(G_{p,q})$ are $\max(p-1,q-1)$-step nilpotent. We will now proceed to exploit the difference between $G_{p,q}$ and $L_p\times L_{q-1}$ to reveal the first phenomenon from the abstract.

\subsection{A family of pairs of cone equivalent groups with different Dehn functions}
We shall use the following notation\footnote{We emphasize that in contrast to a common convention in the setting of Dehn functions we do not allow for a linear term in the definition of $\preccurlyeq$. This has two reasons: {\emph{(i)}} we do not consider any Dehn functions of hyperbolic groups, and {\emph{(ii)}} we require this stronger form of equivalence in the context of sublinear bilipschitz equivalence below.}: if $f,g$ are functions defined on $\mathbf{Z}_{\geqslant 0}$ we write $f(n) \preccurlyeq g(n)$ if $ \vert f(n) \vert \leqslant A \vert g(An + A) \vert +A$ for some $A \geqslant 0$, and $f(n) \asymp g(n)$ if $f(n) \preccurlyeq g(n) \preccurlyeq f(n)$.
Finally, $f(n) \prec g(n)$ means that $f(n) \preccurlyeq g(n)$ holds but $f(n) \asymp g(n)$ does not.

\vspace{.2cm}

Our main result is a computation of the Dehn functions of the groups $G_{p,p}$ and $G_{p,p-1}$:

\begin{thm}
\label{thmIntro:Main}
For all $p\geqslant 4$, $\delta_{G_{p,p}}(n) \asymp n^{p-1}$ and $\delta_{G_{p,p-1}}(n) \asymp n^{p-1}$.
\end{thm}
 
On the other hand it follows from classical arguments that $\mathsf{gr}(G_{p,p-1})\cong L_p\times L_{p-2}$ has Dehn function $\asymp n^{p}$. Hence we deduce from Pansu's Theorem the following corollary.

\begin{cor}
\label{cor:cone-equivalent-groups-with-different-Dehn}
For every $r\geqslant 3$ there is a pair of finitely generated (or simply connected Lie) $r$-nilpotent groups with bilipschitz asymptotic cones but whose Dehn functions have different growth types.
\end{cor}

Note that for $p=3$ Theorem \ref{thmIntro:Main} does not hold in the $p-1$ case: while the Dehn function of $G_{3,3}  = \HH_5(\mathbf{R})$ is known to be quadratic \cite{Allcock, OlsSapCombDehn}, the Dehn function of $G_{3,2}\cong \HH_3(\mathbf{R})\times \RR$ is cubic. We also emphasize that the fact that $G_{3,3}$ does have quadratic Dehn function will later form the basis for our induction argument in the proof of the upper bounds in Theorem \ref{thmIntro:Main}.

It has been known since Gromov \cite{AsInv} that topological properties of asymptotic cones impose restrictions on Dehn functions (e.g. if the asymptotic cone is a real tree, resp. simply connected, then the Dehn function is linear, resp. polynomially bounded). A consequence of a theorem of Papasoglu is that if $\delta_{\mathsf{gr}(G)} \lesssim n^d$, then $\delta_G(n) \lesssim n^{d+\varepsilon}$ for every $\varepsilon>0$ (\cite[2.7]{DrutuRemplissage}, \cite{Papasoglu}). Corollary \ref{cor:cone-equivalent-groups-with-different-Dehn} shows that there is no converse to this theorem, proving that the fine behaviour of the Dehn function is not always captured by the asymptotic cone.
The fact that central products can have a lower Dehn function than their factors has been noticed by Olshanskii and Sapir \cite{OlsSapCombDehn}, and  by Young \cite{YoungFillingNil} for a large class of examples. However, in all situations studied by these authors, the groups in question are actually step $2$ nilpotent and therefore Carnot gradable. 

The lowest-dimensional occurence of the phenomenon described by Corollary \ref{cor:cone-equivalent-groups-with-different-Dehn} is in dimension $6$. 
Indeed, the group $G_{4,3}$ shares its asymptotic cone with two other $6$-dimensional step $3$ groups and the Dehn function of $G_{4,3}$ is cubic whereas for the two others it is quartic. We refer to \S \ref{sec:low-dimension} for a detailed discussion of all $6$-dimensional nilpotent Lie algebras and the Dehn functions of their associated simply connected nilpotent Lie groups.

\subsection{Sublinear bilipschitz equivalence}\label{secIntro:SBE}
Considering the Dehn function of the group $G_{4,3}$ was suggested by Cornulier and triggered our work \cite[Question 6.20]{cornulier2017sublinear}. Cornulier's motivation is coming from the study of sublinear bilipschitz equivalences between nilpotent groups. 
He proves that for every pair $(G,\mathsf{gr}(G))$ where $G$ has step $c$ one can choose $v$ of the form $v(t)\asymp t^e$ with $e=1-1/c$ \cite[Theorem 1.21]{cornulier2017sublinear}.

The Dehn function is well-known to be invariant under quasiisometry and Cornulier observed a weaker stability result for the Dehn function under SBE: having Dehn functions with different exponents implies an asymptotic lower bound $v\succcurlyeq t^e$ for some $e>0$ on the possible functions $v$ such that there can exist an $O(v)$-SBE.  With this in mind he suggested the pair $(G_{4,3},L_4\times \Z^2)$ as a possible example satisfying such a lower bound \cite[Example 6.19]{cornulier2017sublinear}.  
We confirm Cornulier's intuition and, more generally, prove the following result.

\begin{thm}\label{thmIntro:SublinearBIlip}
Let $p \geqslant 4$. If $0 \leqslant e \leqslant \frac{1}{2p}$ then there is no sublinear bilipschitz equivalence between $G_{p,p-1}$ and $L_p\times L_{p-2}$ with $v(t) = O(t^e)$.
\end{thm}
Actually the precise exponent in Theorem \ref{thmIntro:SublinearBIlip} is deduced from a slightly stronger version of Theorem \ref{thmIntro:Main}, saying that the filling occurs in a ball of radius comparable to the length of the loop\footnote{Only using the Dehn function would provide us with the much weaker lower bound of $\frac{1}{p^2}$ on the exponent.}.
This lower bound should be compared with the asymptotic upper bound of $t^{\frac{p-2}{p-1}}$ following from Cornulier's general estimates. It would be interesting to understand the precise asymptotics of the exponent as a function of $p$ as $p  \to + \infty$: in particular, does it tend to zero?  

\begin{remark}\label{rem:Shalom}
As already mentioned, the fact that the groups considered in Theorem \ref{thmIntro:SublinearBIlip} (or their lattices) are not quasiisometric is also a consequence of  \cite[Theorem 1.2]{ShalomHarmonic}. 
Indeed we shall see that $b_2(\Lambda_{p}\times  \Lambda_{p-2}) = b_2( \Gamma_{p,p-1}) +\epsilon_p$, where $\epsilon_p$ is $1$ or $2$ according to whether $p$ is even or odd (see Lemma \ref{lem:H^2g_p,q}). 
\end{remark}

\subsection{Centralized and regular Dehn functions  differ for nilpotent groups}
We now recall the algebraic definition of the Dehn function. Given a presentation (not necessarily finite) $\langle S \mid R\rangle$ of a group $G$, one can define its Dehn function as follows: we call an element $w$ of the free group $F_S$ generated by $S$ {\emph{null-homotopic} (in $G$)} if it represents the trivial element in $G$. For every null-homotopic word $w\in F_S$ we define $\Area(w)$ to be the minimal integer $k$ such that
$$w=\prod_{i=1}^k u_i^{-1}r_iu_i,$$
where $u_i\in F_S$ and $r_i\in R^{\pm 1}$.
The Dehn function $\delta_{G,S,R}(n)$ is the (possibly infinite) infimum of $\Area(w)$ over all null-homotopic words $w\in F_S$ of length at most $n$. 
If the group is finitely presented, then the Dehn function takes finite values and its asymptotic behavior  does not depend on the choice of finite presentation. A similar statement holds for compactly presented groups (see \S \ref{sec:eqn-defs-Dehn-fct}).

In \cite{BaumslagMillerShort} Baumslag, Miller and Short introduce the closely related notion of centralized Dehn function of a presentation $\langle S \mid R\rangle$ of a group $G$, which they define as follows:
\begin{definition}
Denote $\mathfrak R$ the normal subgroup of $F_S $ generated by $R$. Given a null-homotopic word $w\in F_S $, we define its central area $\Area^{\mathrm{cent}}(w)$   to be the minimal integer $k$ such that $$w=\prod_{i=1}^k r_i$$ in $\mathfrak R /[F_S ,\mathfrak R]$, with $r_i\in R^{\pm 1}$. 
The centralized Dehn function $\delta^{\mathrm{cent}}_{G,S,R}(n)$ is the (possibly infinite) infimum of $\Area^{\mathrm{cent}}(w)$ over all null-homotopic words $w$ of length at most $n$. 
\end{definition}
As for the Dehn function, one can show that the asymptotic behavior of the centralized Dehn function of a finitely presented group does not depend on a specific choice of finite presentation, so we simply denote it by $\delta_G^{\mathrm{cent}}$.
Note that we have $\delta_G^{\mathrm{cent}}\leqslant \delta_G$ by definition. It turns out that $\delta^{\mathrm{cent}}_G$ is in general easier to estimate as it is closely related to the second cohomology group of $G$, or, equivalently, to the central extensions of $G$. 
 In particular we have the following useful characterization of the centralized Dehn function for torsion-free nilpotent groups.

\begin{proposition}[see Proposition \ref{prop:Centralrcentralextensions}]
Let $\Gamma$ be a torsion-free nilpotent group and let $\mathfrak g$ be the Lie algebra of its Malcev completion. Then $\delta^{\mathrm{cent}}_\Gamma(n)\asymp n^r$, where $r$ is the largest integer such that $\mathfrak g$ admits a central extension $\RR\to \tilde{\mathfrak g}\to \mathfrak g$ whose kernel belongs to $\gamma_{r}  \tilde{\mathfrak g}$.
\end{proposition}
Such a central extension will be called $r$-central in the sequel. 
The centralized Dehn function was used in \cite{BaumslagMillerShort}  to obtain sharp lower bounds on the Dehn functions of certain nilpotent groups. In \cite{young2006scaled} Young mentions that for nilpotent groups it is unknown whether $\delta^{\mathrm{cent}}_\Gamma(n) \asymp \delta_\Gamma(n)$. Later Wenger exhibited a 2 step nilpotent group whose Dehn function strictly lies between quadratic and $n^2\log n$ \cite{Wenger}, therefore answering Young's question negatively. Here we show that even the growth exponents of the two functions can be different. 

\begin{thm}\label{propIntro:centralDehn}
Let $k$ be an integer $\geqslant 2$. We have $$\delta^{\mathrm{cent}}_{\Gamma_{2k,2k-1}}(n)\asymp \delta^{\mathrm{cent}}_{\Gamma_{2k+1,2k}}(n) \asymp n^{2k-1}.$$
Hence the Dehn function and the centralized Dehn function have different exponents for $\Gamma_{2k+1,2k}$.
\end{thm}

\subsection{Structure of the paper} In \S \ref{sec:overview} we give an overview of the proof of our main results. In \S \ref{sec:eqn-defs-Dehn-fct} we introduce basic notions and results regarding compact presentations, Dehn functions and filling diameters. In \S \ref{sec:warmup} we prove the upper bound in Theorem \ref{thmIntro:Main} for $p=4$ as a warm-up for the general case. In \S \ref{sec:preliminaries} we set the stage for the proof of the upper bound in Theorem \ref{thmIntro:Main} for general $p$, by deriving an explicit compact presentation for $G_{p,k}$ and then proving several preliminary results satisfied by words in its generators. \S \ref{sec:upper-bound-Dehn} contains the proof of the upper bound in Theorem \ref{thmIntro:Main}. In \S \ref{sec:lower-bound-dehn} we explore the existence of central extensions of central products. In \S \ref{sec:lower-bounds-forms} we derive the lower bounds in Theorem \ref{thmIntro:Main} for all $p$, showing that for odd $p$ the lower bounds on the Dehn function of $G_{p,p-1}$ provided by the centralized Dehn function are not optimal, thus also completing the proof of Theorem \ref{propIntro:centralDehn}. \S \ref{sec:SBE} is concerned with applying our results in the theory of SBE's, leading to a proof of Theorem \ref{thmIntro:SublinearBIlip}. In \S \ref{sec:low-dimension} we give an overview of the Dehn functions of nilpotent groups of dimension less or equal to six. Finally we list some open questions and speculations arising from our work in \S \ref{sec:questions}.

\subsection{Conventions and notations}

\subsubsection*{Groups and Lie algebras}
When working with words $w(X)$ in the generators of a group $G$ with presentation $\mathcal{P}=\left\langle X\mid R\right\rangle$ we will be careful to distinguish equalities of words and equalities of their corresponding elements in the group. To do so, for words $w_1(X)$ and $w_2(X)$ we will write $w_1(X)=w_2(X)$ if they are equal as words and $w_1(X)\equiv w_2(X)$ (with respect to $\mathcal{P}$ or $G$) if they represent the same element of the group. Whenever this is not clear from context we will make sure to mention the presentation (or group) explicitly when using $\equiv$. 
	
We will write $[w]$ for the group element represented by a word $w$ if we want to explicitly distinguish it from the word. We will denote by $\ell(w)$ the word length of a word $w(X)$ and for a group element $g\in G$ by $|g|_{X}:=\mathrm{Cay}_{G,X}(1,g)$ the distance of $g$ from the origin in the Cayley graph.
	
We call a word $w(X)$ central if $[w]$ is a central element of the group $G$.

\subsubsection*{Asymptotic comparisons}
We shall use the notation $A\lesssim_a B$ to mean that there exists some $C<\infty$ only depending on $a$ such that $A\leqslant CB$. Similarly we denote $A\simeq_a B$ if $A\lesssim_a B$ and $B\lesssim_a A$. Sometimes we will also say $A$ is in $O_a(B)$ if $A\lesssim_a B$ and $A=O_a(B)$ if $A\simeq_a B$.

\begin{acknowledgements*}
We are grateful to Yves Cornulier and Christophe Pittet for helpful comments on a previous version of this paper. We also thank Francesca Tripaldi for helpful discussions.
We thank the anonymous referee whose comments led to significant improvements in the exposition of our work and results.  
\end{acknowledgements*}

\newpage

\tableofcontents

\section{Overview of the proof}\label{sec:overview}
\label{sec:Intro-moral-proof}
To provide the reader with an intuition for the proofs in this paper we now briefly explain the moral ideas behind why the groups $G_{p,p-1}$ satisfy the conclusions of Theorem \ref{thmIntro:Main} and Theorem \ref{propIntro:centralDehn}. 

The proof of Theorem \ref{thmIntro:Main} and Theorem \ref{propIntro:centralDehn} has three fundamental parts, which make up most of this paper: 
\begin{enumerate}
 \item the proof of the upper bound of $n^{p-1}$ on the Dehn functions of $G_{p,p-1}$ and $G_{p,p}$. This will make up by far the biggest part of this work and will be contained in \S \ref{sec:warmup} -- \S \ref{sec:upper-bound-Dehn}; 
 \item the proof that $G_{p,p-1}$ admits no $(p-1)$-central extension when $p$ is odd, which will be contained in \S \ref{sec:lower-bound-dehn};
 \item the proof that the Dehn function of $G_{p,p-1}$ is nevertheless bounded below by $n^{p-1}$, irrespectively of the parity of $p$, which will be contained in \S \ref{sec:lower-bounds-forms}.\vspace{.3cm}
\end{enumerate}

Parts (2) and (3) turn out to be easier to explain in the setting of Lie algebras, while we postpone most of the explanation of Part (1) to \S \ref{sec:Intro-sketch-MainThm}. So we will adopt the Lie algebra point of view here. 

We recall the notation $x_1, \dots, x_p=z$ for the standard generators of the Lie algebra $\mathfrak{l}_p$ of $L_p$ and we will denote by $x_1, \dots,x_{p-1}, x_p=z, y_1, \dots ,y_{q-1}, y_q=z$ the standard generators of the Lie algebra $\mathfrak{g}_{p,q}$\footnote{To avoid confusion, let us mention that when we work with the Lie group $G_{p,q}$ we will denote the generators of the $L_q$-factor by $y_1,y_{p-q+2}, \dots ,y_{p-1},y_p=z$, as this turns out to be more convenient, while for the Lie algebra setting the indices chosen here turn out to be easier to work with. The Lie algebra approach and thus this choice of indices will only appear in \S \ref{sec:lower-bound-dehn}.} of $G_{p,q}$ for $p\geqslant q \geqslant 3$. We denote its dual basis by $\xi_1,\dots,\xi_{p-1},\xi_p=\zeta,\eta_1,\dots,\eta_{q-1},\eta_q=\zeta$. We will restrict to the case $q=p-1$ for simplicity, even though parts of our subsequent arguments extend directly to general $q\in \left\{3,\dots, p-1\right\}$.\vspace{.3cm}

\subsection{The fundamental reason for why everything works}
At the base of all three parts is the existence of the central element $z$ which connects the two factors of the central product via the identification $z=\theta(x_p)=y_q$ in $\mathfrak{g}_{p,q}=\mathfrak{l}_p\times_\theta \mathfrak{l}_q$, respectively its group theoretic analogue.

From the Lie algebra point of view this comes into play as follows: the differential of $\zeta$ is $d\zeta = -\xi_1\wedge \xi_{p-1}$ in $\mathfrak{l}_{p}$, and thus in $\mathfrak{l}_p \times \mathfrak{l}_{p-2}=\mathsf{gr}(\mathfrak{g}_{p,p-1})$, but $d\zeta= -\xi_1 \wedge \xi_{p-1} - \eta_1\wedge \eta_{q-1}$ in $\mathfrak{g}_{p,q}$.

The computational consequence is that it will be more difficult for a form $\omega \in \bigwedge ^2 \mathfrak{g}_{p,q}^{\ast}$ to have vanishing exterior derivative if it has terms with a non-trivial $\zeta$ contribution than is the case in $\bigwedge ^2 \mathfrak{l}_{p}^{\ast}$. Indeed, $d\zeta$ being a linear combination of two basis elements means that its differential ``interacts non-trivially'' with more other basis elements than if it only had one summand.

From the group theory point of view the relation $z=x_p=y_q$ will enable us to move central words in the $x_i$ between factors, allowing us to commute them more easily with other words in the $x_i$. \vspace{.3cm}

We briefly expand on how these observations come into play in Parts (1)--(3), thereby providing the moral idea of why and how our proof works. \vspace{.3cm}

{\noindent \bf Part (1):} 
Let us just mention at this point that the argument is by induction on $p$ and ultimately boils down to the idea that we can commute central words of length $n$ in the $x_i$ with other words of length $n$ in the $x_i$ at cost $n^{p-1}$ rather than $n^p$ (as one might naively expect). We achieve this by passing through the second factor of the central product via the subgroup $G_{p-1,p-1}\leqslant G_{p,p-1}$, which has Dehn function $n^{p-2}$ by induction. 
Actually proving this for general $p$ will require a chain of combinatorial results. However, a good intuition for the general ideas should be attainable from the case $p=4$, which we will sketch in \S \ref{sec:Intro-sketch-MainThm} and prove in detail in \S \ref{sec:warmup}.\vspace{.3cm}

{\noindent \bf Part (2):} 
In the $\mathfrak{l}_p$-factor of the Carnot Lie algebra $\mathfrak{l}_p \times \mathfrak{l}_{p-1}=\mathsf{gr}(\mathfrak{g}_{p,p-1})$ associated to $\mathfrak{g}_{p,p-1}$ there is a 2-form $\nu_{2p'}$ with $p'=\lceil p/2 \rceil$, which defines a $(2p'-1)$-central extension of $\mathfrak{l}_p$ and thus of $\mathfrak{l}_p \times \mathfrak{l}_{p-2}$. A precise definition of this form will be given in \S \ref{subsec:central-extensions-central-products}. Note that if $p$ is even $(2p'-1) = p-1$, whereas if $p$ is odd $2p'-1 = p$.
Interestingly, the form $\nu_{2p'}$ only defines a cocycle in $Z^2(\mathfrak{g}_{p,p-1},\RR)$ if $p$ is even, that is, its exterior derivative does not vanish when $p$ is odd. 
In terms of linear algebra the non-vanishing of its exterior derivative precisely boils down to the fact that $d\zeta$ has one summand more in $\mathfrak{g}_{p,q}$ than in $\mathfrak{l}_p \times \mathfrak{l}_{p-2}$ due to the central product structure.

Irrespectively of the parity of $p$ there are no other forms defining $r$-central extensions for $r\geqslant p-1$ in $Z^2(\mathfrak{g}_{p,p-1},\RR)$ and we deduce that $\mathfrak{g}_{p,p-1}$ admits a $(p-1)$-central extension if and only if $p$ is even. In combination with Theorem \ref{thmIntro:Main} this proves Theorem \ref{propIntro:centralDehn}.\vspace{.3cm}
 
{\noindent \bf Part (3):}
On first sight there is one more candidate for a cocycle defining a $(p-1)$-central extension of $\mathfrak{g}_{p,p-1}$, namely $\xi_1 \wedge \xi_{p-1}$. But of course it is a non-candidate, because it is the 2-form defining the ``obvious'' $(p-2)$-central extension $\mathfrak{l}_p \times \mathfrak{l}_{p-1} \to \mathfrak{g}_{p,p-1}$.

However, this ``false candidate'' for a $(p-1)$-central extension is precisely the reason why the Dehn function of $G_{p,p-1}$ for odd $p$ is bigger than one might expect from the centralized Dehn function.
 
Indeed, as we already mentioned, $\xi_1 \wedge \xi_{p-1}$ defines a cocycle in  $Z^2(\mathfrak{g}_{p,p-1},\RR)$ and in a sense the only problem is that the commutator $\left[x_1,x_{p-1}\right]\in \gamma_{p-1}\mathfrak{g}_{p,p-1}$ on which it is non-trivial is equal to $z$ and in particular does not vanish in $\mathfrak{g}_{p,p-1}$. 
 
We found a solution to overcome this issue and confirm our intuition that the Dehn function of $G_{p,p-1}$ is bounded below by $n^{p-1}$ also when $p$ is odd. The idea is to exploit a ``perturbation''  of the 2-form $\xi_1 \wedge \xi_{p-1}$ in order to show that the null-homotopic loops $\left[x_1^n,\left[x_1^n,\dots,\left[x_1^n,x_2^n\right]\dots\right]\right]$ have area bounded below by $n^{p-1}$. We will explain the technique used for this approach in the first half of \S \ref{sec:Intro-sketch-MainThm}. \vspace{.5cm}

\subsection{Sketch of proof of Theorem \ref{thmIntro:Main}}
\label{sec:Intro-sketch-MainThm}
The proof of Theorem \ref{thmIntro:Main} will cover the largest part of this paper. It splits into two independent parts: the proof of the lower bound, and the proof of the upper bound. The former will be contained in \S \ref{sec:lower-bounds-forms}, while the latter will span \S \ref{sec:warmup} -- \S \ref{sec:upper-bound-Dehn}. To make it more accessible we will provide a brief summary of the main ideas involved.

We start by discussing the proof of the {\bf lower bound}. When $p$ is even, then the lower bound is simply given by Theorem \ref{propIntro:centralDehn}. 
The case when $p$ is odd is much more involved and requires new ideas. Our method is inspired by Thurston's proof of the exponential lower bound on the Dehn function of the real 3-dimensional SOL group \cite{ECHLPT92}. 
Thurston proceeds as follows: he exhibits a $1$-form $\alpha$ on $G$ such that $d\alpha$ is left-invariant, and a sequence of loops $\gamma_n$ of length $n$ such that the integral of $\alpha$ along $\gamma_n$ is $\geqslant \lambda^n$ for some $\lambda>1$. 
A direct application of Stokes' theorem then implies that the area of any smooth embedded surface bounded by $\gamma_n$ must be bounded below by $c\lambda^n$ for some constant $c>0$ only depending on $\alpha$ and on a choice of left-invariant Riemannian metric on $G$.

The first step in our argument consists of the observation that
Thurston's assumption that $d\alpha$ is invariant can be relaxed to the
weaker assumption that it is ``bounded''. To that purpose, we define the
space of {\emph{bounded $k$-forms}} on $G$ to be the space of forms
$\alpha$ such that $\sup_{g\in G} \|(g_\ast \alpha)_{1_G}\|<\infty$,
where $\|\cdot\|$ is a norm on $\bigwedge^k \mathfrak{g}^*$ (note that
the boundedness condition does not depend on a choice of such a norm).
It is quite immediate to see that Thurston's approach works verbatim
replacing the condition that $d\alpha$ is invariant by the condition
that $d\alpha$ is bounded. We note that a related approach was developed by Gersten, who explains  how $\ell^{\infty}$-cocycles can be used to obtain lower bounds on the Dehn function of a finitely presented group $G$ \cite[2.7]{GerstenCohomLowerBounds}.

The second step and main innovation in our argument is the construction
of a suitable bounded 2-form by ``deforming'' a well-chosen invariant
form. For this we exploit the central product structure of our groups.
We start by observing that $G_{p,p-1}$ maps surjectively to $L_{p-1}$.
We shall consider a $2$-cocycle of $L_{p-1}$ associated to its central
extension $L_{p}$ and consider an invariant $2$-form $\beta$
representing it in de Rham cohomology. We will  then consider a relation
$r$ of length $n$ in $L_{p-1}$ and a primitive $\alpha$ of $\beta$ whose
integral along (a continuous path associated to) $r$ has size $\asymp
n^{p-1}$.
Although the word corresponding to $r$ won't define a relation in
$G_{p,p-1}$, its commutator $[y,r]$ for a suitable word $y$ will.  The
problem at this point is that the integral of $\alpha$ along $[y,r]$
will be zero.
So we shall perform a suitable  ``local perturbation" of $\alpha$,
obtaining a $1$-form $\alpha'$ whose integral along $[y,r]$ is $\asymp
n^{p-1}$, and such that $d\alpha'$ while not being invariant anymore
will remain bounded. This will show that the area of $[y,r]$ in
$L_{p-1}$ (and a fortiori in $G_{p,p-1}$) is at least $n^{p-1}$.

Actually when trying to implement the previous argument, we run into a regularity problem: we have to deal with forms that are not smooth, preventing us from using Stokes' theorem. A solution would be to smoothen our forms so that the previous argument could be applied directly. However, this would make our computations more cumbersome. We chose instead to privilege an alternative approach, which better suits the study of Dehn functions associated to compact presentations. 
The idea is to replace the condition that $d\alpha$ is bounded by the fact that the integral of $\alpha$ along any loop of bounded length is bounded. This condition is easy to work with and has the nice advantage of making sense for continuous $1$-forms. Moreover it satisfies a discrete version of Stokes' Theorem, inspired by \cite[Section 12.A]{CoTesDehn}. 

We now turn to the proof of the {\bf upper bound} in Theorem \ref{thmIntro:Main} that occupies the largest part of the paper and is our main contribution to the subject. In \S \ref{sec:warmup} we start by proving the upper bound $\delta_{G_{4,3}}(n) \lesssim n^{3}$.
Indeed, while containing the main idea, this bound turns out to be considerably easier to obtain than the more general bound $\delta_{G_{p,p-1}}(n) \lesssim n^{p-1}$. At the end of \S \ref{sec:warmup}, we shall explain the difficulties arising in the general case, and our strategy to overcome them.  For now, we shall focus on the special case $p=4$ and further restrict to the discrete group $\Gamma_{4,3}$.

The key idea in the proof is to exploit the fact that there is a canonical embedding $\Gamma_{3,3}\cong \HH_5(\ZZ)\hookrightarrow \Gamma_{4,3}$ of the 5-dimensional Heisenberg group which, as we mentioned before, has Dehn function $n^2$.  We will explain the main steps of the proof and, in particular, where we use the embedding of $\HH_5(\ZZ)$:

In a first step we reduce to considering null-homotopic words $w=w(x_1,x_2)$ in the generators of the first factor $\Lambda_4\leqslant \Gamma_{4,3}$ of the central product. The core of the argument, which we will explain now, consists of transforming $w(x_1,x_2)$ into a word that closely resembles the normal form $x_3^{a_3} x_1^{a_1}  x_2^{a_2} x_4^{a_4}$. Since for a null-homotopic word we must have $a_3=a_1=a_2=a_4=0$ we can then conclude from there.

Given a word $w(x_1,x_2)$ of length $\ell(w)=n$ the idea is to push all $x_1$'s to the left one-by-one, starting with the leftmost one. Modulo $\gamma_2(\Lambda_4)$ this will eventually yield the word $x_1^{a_1} x_2^{a_2}$. However, whenever we commute a $x_1$ with a $x_2^{O(n)}$ we produce an error term $x_3^{O(n)}$ which we then need to move out of the way. We do this by pushing it to the very left of the word, at the cost of producing a central word of the form $\left[x_1^{O(n)},x_3^{O(n)}\right]$. All steps up to this point require $O(n^2)$ relations and repeating this $O(n)$ times, once for each instance of $x_1$, would provide us with the desired area bound of $O(n^3)$. 

However, the problem is that this is only true modulo $\gamma_3(\Lambda_4)$. Instead we also need to move the word of the form $\left[x_1^{O(n)},x_3^{O(n)}\right]$ which we produced out of the way in every step. We want to do this by moving it to the very right of the word. This involves commuting it with words of the form $x_1^{O(n)}$, which in the 3-Heisenberg group $\HH_3(\ZZ) \cong \Lambda_3=\langle x_1,x_3\rangle$ requires $O(n^3)$ relations. After $O(n)$ repetitions we would thus end up with an upper area bound of $O(n^4)$ rather than $O(n^3)$. {\emph{This is the point at which we make fundamental use of the fact that the group $\langle x_1,x_3\rangle$ is the left factor of an embedded 5-dimensional Heisenberg group}} obtained by taking the central product of $\Lambda_3$ with itself. Indeed, this allows us to replace the central word $\left[x_1^{O(n)},x_3^{O(n)}\right]$  in the left factor by a central word $v$ of the same length in the right factor of the central product $\Gamma_{3,3}$ using $O(n^2)$ relations. We can then commute $v$ with $x_1^{O(n)}$ using only $O(n^2)$ relations. After $O(n)$ repetitions of the total process, each of which has a total cost of $O(n^2)$ relations, we thus reach a word that closely resembles the normal form $x_3^{a_3} x_1^{a_1}  x_2^{a_2} x_4^{a_4}$. For this we required only $O(n^3)$ relations, rather than the expected $O(n^4)$ relations, and we can conclude from there. Note that in fact in this last step we use that $\HH_5(\ZZ)$ has Dehn function $n^2$ once more to simplify a product of $O(n)$ copies of central words of the  form $\left[x_1^{O(n)},x_3^{O(n)}\right]$ into the trivial word. 

We will use various analogues of both of the kinds of above transformations coming from the embedded copy of $\HH_5(\ZZ)\leqslant \Gamma_{4,3}$ for general $p$, by exploiting the embedded subgroup $G_{p-1,p-1}\leqslant G_{p,p-1}$. They will appear at many points of the proof and ultimately lead to two key technical results: the Main commuting Lemma (Lemma \ref{lem:MainLemma}) and the Cancelling Lemma (Lemma \ref{lem:Strong-k-Lemma}), which in essence can be seen as our most general versions of the first and second application of $\HH_5(\ZZ)$ above. There will be various challenges to overcome for general $p$ in comparison to $p=4$. The most obvious one is that the central series of $\Lambda_p$ has more than three non-trivial terms. This means that there is not enough space to mimic the trick we used for $p=4$, where we conveniently left terms in $\gamma_1(\Lambda_4)$ in the middle, moved terms in $\gamma_2(\Lambda_4)$ to the left and finally moved terms in $\gamma_3(\Lambda_4)$ to the right, which provided us with a suitable normal form.

When computing the upper bounds for the Dehn functions of the $G_{p,p-1}$ we will use Dehn functions of compact presentations rather than either geometric methods or Dehn functions of discrete groups. Indeed, while we do use a more geometric approach in our proof of the lower bounds, we were not able to find an obvious geometric model for our groups that allows for the ``easy'' computation of upper bounds on Dehn functions. On the other hand they are too complicated to pursue a discrete combinatorial approach. It is thus really the hybrid approach between the two points of view provided by compact presentations of Lie groups that allows us to prove our results. Indeed, it provides us with the ``geometric'' flexibility of writing our words in a relatively simple and thus manageable form on the combinatorial side, while at the same time allowing us to use all of the classical tools and manipulations from discrete combinatorial group theory, thereby not requiring the use of an intricate geometric model. We thus believe that this kind of approach really merits attention, as it might also be instrumental in other problems in this area. We emphasize that this has also been suggested in \cite{CornulierTesseraDehnBaums}.

\section{Dehn functions, filling diameters and filling pairs}
\label{sec:eqn-defs-Dehn-fct}

In this section we will introduce basic notions on Dehn functions, filling diameters and filling pairs and collect some important well-known results on them.

\subsection{Dehn functions of compactly presented groups}\label{sec:dehnvarious}
Let $G$ be a compactly generated locally compact group. For any compact generating set $S$ let $K(G,S)$ be the kernel of the epimorphism $F_S \twoheadrightarrow G$ where $F_S$ denotes the free group over $S$.
Recall that $G$ is compactly presented with compact presentation $\mathcal P = \langle S \mid R \rangle$ if $K(G,S)$ is the normal closure of $ R \subset K(G,S)$ such that $ R$ is bounded with respect to the word metric on $F_S$. 
Simply connected Lie groups are known to be compactly presented (see for instance  \cite[Th 2.6]{TesseraMetricLC}). For simply connected nilpotent Lie groups such presentations can theoretically be obtained over an arbitrary compact generating set from the knowledge of a Lie algebra presentation using the Baker-Campbell-Hausdorff series (of which only finitely many terms actually appear). These presentations are however unpractical to work with and in \S \ref{subsec:compact-presentations} we shall thus provide explicit constructions of compact presentations for the groups $L_p$ and $G_{p,q}$.

Let $\mathcal P = \langle S \mid R \rangle$ be a compact presentation of a locally compact group $G$. Recall that a freely reduced word $w$ over $S$ represents the identity in $G$ if and only if it belongs to the normal closure of $R$. Further recall that we call such a word null-homotopic, that we define ${\rm Area}(w)$ as the minimal number of conjugates of relations $r \in R^{\pm 1}$ whose product is freely equal to $w$ and that the Dehn function $\delta_\mathcal P$ of a compact presentation $\mathcal{P}$ is defined by
\[
\delta_{\mathcal P}(n) = \sup \left\{ \operatorname{Area}(w): w \text{ null-homotopic and freely reduced of length } \leqslant n \right\}.
\]

\begin{remark}
Two remarks are in order here. First, it is easy to check that provided that it is finite, the asymptotic behaviour of $\delta_{\mathcal P}(n) $ does not depend on a choice of compact presentation. Second, by definition any compactly presented locally compact group admits a  presentation of the form $\langle S \mid R \rangle$ where $R=R_k$ consists of all null homotopic words in $S$ of length at most $k$ and for any such presentation $\delta_{\mathcal P}$ is finite (\cite[Proposition 11.3]{CornulierDehn}). 
\end{remark}

It turns out that the Riemannian definition of the Dehn function that we gave in the introduction and the combinatorial definition have the same asymptotic behaviour. More generally, given a Riemannian manifold $M$ define $F(r)$ to be the supremum of areas needed to fill loops of length at most $r$ in $M$.
The following result is due to Bridson when $G$ is discrete \cite[Section 5]{Bri02}. 

\begin{proposition}[{\cite[Proposition 2.C.1]{CoTesDehn}}] Let $G$ be a locally compact group with a proper cocompact isometric action on a simply connected Riemannian manifold $X$. Then $G$ is compactly presented and the Dehn function of $G$ satisfies
$$\delta(r)\asymp \max \left\{F(r),r\right\}.$$
\end{proposition}

To complete the picture we mention that the asymptotic behavior of the Dehn function is invariant under quasi-isometry; this was proved for finitely presented groups in \cite{Alo} and the proof adapts without changes to compactly presented groups.

\subsection{Fillings in balls of controlled radius}
We will be interested in constructing fillings where we simultaneously control the number of relations and the diameter of the image of the corresponding van Kampen diagram. Geometrically this amounts to filling a word in a ball of controlled radius. As in the previous section let $\mathcal P = \langle S \mid R \rangle$ be a compact presentation of a locally compact group $G$.
We will say that a word $w=w(S)$ has \emph{(word) diameter} $\leqslant d$ in $G$ if the associated path in the Cayley graph of $G$ stays at distance $\leqslant d$ from the identity $1\in G$. Equivalently $w$ has diameter $\leqslant d$ if for any decomposition $w=w_1\cdot w_2$ into two subwords we have ${\rm dist}_{{\rm Cay} (G,S)} \left(1,\left[w_1\right]\right)\leqslant d$.

\begin{definition}
 Given a null-homotopic word $w(S)$, we say that a filling 
 \[
  w(S) = \prod_{i=1}^k u_i^{-1} r_i u_i
 \]
 of area $k$ has \emph{ (filling) diameter $d$} if $u_i$ has word diameter $\leqslant d$ for $1\leqslant i \leqslant k$.
\end{definition}
We will often drop the specification ``word'' and ``filling'' diameter when it is clear from the context which one we mean.

%\begin{remark}The constants $O_{\mathcal P}(1)$ can differ between different implications. When passing from (2) to (1) a van Kampen diagram of diameter $\leqslant d$ implies that the word has a filling of diameter $\leqslant d$. Conversely filling diameter $\leqslant d$ only implies that the van Kampen diagram has diameter $\leqslant d+M$ though. The same considerations hold for the equivalence of (2) and (3). \end{remark}

We will say that two words $w(S)$ and $w'(S)$ are equivalent with area (or at cost) $k$ and diameter $d$ if $w'\cdot w^{-1}$ is null-homotopic and admits a filling with area $k$ and diameter $d$. In this case we will also say that the identity $w\equiv w'$ holds with area $k$ and diameter $d$ in $G$. 
\begin{remark}\label{rem:diameter}
 We emphasize that the definition of the diameter of the equivalence $w\equiv w'$ involved a choice: we chose to estimate the diameter of a filling of $w'\cdot w^{-1}$ rather than $w'^{-1}\cdot w$. While both words have the same filling areas they differ by a conjugation by $w'$ and thus their filling diameters can differ by $\ell(w')$. We shall stick to this choice throughout the paper. 
\end{remark}

We will frequently use the following simple observation:
\begin{lemma}\label{lem:diameter}
  Let $w=w(S)$ be a word that decomposes as $w(S)=w_1(S)\cdot w_2(S)\cdot w_3(S)$ and let $w_2'=w_2'(S)$ be equivalent to $w_2$ mod $\langle\langle R\rangle\rangle$ via a transformation with area $k$ and diameter $d$. 
 
 Then the identity $w\equiv w'\mbox{ mod } \langle \langle R\rangle \rangle$ for $w'=w_1w_2'w_3$ holds with area $k$ and diameter $d'\leqslant d+  r$ in $G$, where $r$ is the word diameter of $w_1$. In particular, if $d\leqslant n$ and $r\leqslant n$ then $d'\leqslant 2n$.
\end{lemma}
\begin{proof}
 This follows easily from the definitions.
\end{proof}
We call a word $w_1$ (resp. $w_3$) as in Lemma \ref{lem:diameter} a \emph{prefix} (resp. \emph{suffix}) word for the transformation of $w$ into $w'$. 
\begin{remark}
 The fact that only the prefix word $w_1$ plays a role in the estimate in Lemma \ref{lem:diameter} comes from the choice we discussed in Remark \ref{rem:diameter}.
\end{remark}

\subsection{Filling pairs}

\begin{definition}
Given two increasing unbounded functions $f,g:\mathbf R_+\to \mathbf R_+$, we say that a compactly presented group admits a $(f,g)$-filling pair if every null-homotopic word $w=w(S)$ of length $n$ has a filling of area in $O(f(n))$ and filling diameter in $O(g(n))$.
\end{definition}

Filling pairs are quasi-isometry invariants of compactly presented groups up to equivalence $\asymp$ (where for hyperbolic groups we allow for a linear term in the first entry). The proof is the same as for Dehn functions and we refer to Lemma \ref{lem:filling-transfer-SBE} for details, where we prove a more general result for SBEs.

If $G$ is a topological group, recall that $H < G$ is a retract of $G$ if it is a closed subgroup and there is a surjective homomorphism $\rho: G \to H$ which restricts to the identity on $H$. The following are well-known in the context of Dehn functions of finitely presented groups (see \cite[Lemma 1]{BaumslagMillerShort}, resp. \cite[Proposition 2.1]{Brick}) and their proofs adapt easily to filling pairs of compactly presented groups.

\begin{lemma}
\label{lem:Dehn-functions-of-retracts}
Let $G$ be a compactly presented locally compact group. If $H$ is a retract of $G$, then $H$ is compactly presented and any filling pair for $G$ is a filling pair for $H$.  
\end{lemma}

\begin{lemma}
\label{lem:Dehn-functions-of-direct-products}
Let $H_1$ and $H_2$ be noncompact compactly presented locally compact groups. Let $H = H_1 \times H_2$ and let $(f_1,g_1)$ (resp. $(f_2,g_2)$) be filling pairs for $H_1$ (resp. $H_2$).
Then $$\left(n^2+f_1(n)+f_2(n) ,n + g_1(n) + g_2(n)   \right)$$ is a filling pair for $H$.
\end{lemma}

\section{Warm up -- an upper bound for the Dehn function of \texorpdfstring{$G_{4,3}$}{G43}}
\label{sec:warmup}

As a warm up for the general proof of the upper bound of $n^{p-1}$ on the Dehn function of $G_{p,p}$ and $G_{p,p-1}$ we will discuss the special case when $p=4$. This case will serve as base case for our induction argument in \S \ref{sec:upper-bound-Dehn}. The case of general $p$ is very subtle, requiring a careful chain of technical lemmas. In contrast the case $p=4$ captures much of the essence of how our general proof works, while avoiding almost all of the technical difficulties. In particular, we can work hands on with the finitely presented lattice $\Gamma_{4,3}$. We will conclude this section by explaining the difficulties we will face when dealing with general values of $p$ and how we will resolve them.

\subsection{Deriving a cubical upper bound for \texorpdfstring{$\Gamma_{4,3}$}{Gamma43}}
\label{subsec:Proof-upper-bound-p-4}
As recalled in the previous section, the Dehn functions of $\Gamma_{4,3}$ and of $G_{4,3}$ are equivalent, and it will be easier here to deal with $\Gamma_{4,3}$. 
Some of the techniques and notation we will use in this section are inspired by Olshanskii and Sapir's combinatorial proof that the Dehn function of the $5$-dimensional Heisenberg group is quadratic \cite{OlsSapCombDehn}. However, our line of argument is rather different from theirs. Indeed we will start by assuming that $\delta_{H_5}(n)\asymp n^2$, which is the main result of their work, and deduce from it that $\delta_{\Gamma_{4,3}}(n)\asymp n^3$.

We recall that we work with the presentation
\[
 \mathcal{P}(\Gamma_{4,3}) =  \left\langle \begin{array}{cccc} x_1,& x_2,& x_3,& x_4,\\ y_1,& & y_3,& y_4,\\ &&& z \end{array} \left\mid  \begin{array}{l} \left[x_1,x_i\right] = x_{i+1}, 2\leqslant i \leqslant 3 ,\\ \left[y_1,y_3\right]=y_4, \left[x_i,y_j\right]=1, \\ x_4=y_4=z \mbox{ is central } \end{array}  \right. \right\rangle
\]
for $\Gamma_{4,3}$. Observe that it naturally contains the presentation $\mathcal{P}(\Gamma_{3,3})$ of the $5$-dimensional Heisenberg group $\HH_5(\ZZ)=\Gamma_{3,3}$ given by
\[
   \mathcal{P}(\Gamma_{3,3})  =  \left\langle \begin{array}{ccc} x_1,& x_3,& x_4,\\ y_1,& y_3,& y_4,\\ && z \end{array} \left\mid  \begin{array}{l} \left[x_1,x_3\right] = x_4 ,\\ \left[y_1,y_3\right]=y_4, \left[x_i,y_j\right]=1, \\ x_4=y_4=z \mbox{ is central } \end{array}  \right. \right\rangle.
\]
We state the following result:
\begin{theorem}[{\cite{Allcock, OlsSapCombDehn}}]\label{thm:5-dim-Heisenberg}
$\Gamma_{3,3}$ admits $(n^2,n)$ as a filling pair.
\end{theorem}
The linear bound on the diameter is not stated in these references. However, it is easy to deduce it from Allcock's proof. Since he works with the Riemannian version of the Dehn function in the real Heisenberg group, we postpone the presentation of his argument to \S \ref{subsec:PfMainThm}.

The key observation that makes our proof work is that the natural embedding of $\HH_5(\ZZ)$ in $\Gamma_{4,3}$ combined with Theorem \ref{thm:5-dim-Heisenberg} allows us to manipulate words of length $n$ in the letters $\left\{x_1,x_3,y_1,y_3\right\}$ at cost $\lesssim n^2$ and in a ball of diameter $\lesssim n$. The following is a particularly important immediate consequence, as it enables us to ``change between factors'' and thus exploit the central product structure of $\Gamma_{4,3}$.
\begin{lemma}\label{lem:ChangingFactors-p-4}
 There is a constant $C_0>0$ such that every word $w(x_1,x_3)$ of length $n$ representing an element of $\gamma_3(\Gamma_{4,3})$ is equivalent to the word $w(y_1,y_3)$ with area $\leqslant C_0 n^2$ and diameter $\leqslant C_0 n$ in $\Gamma_{4,3}$.
\end{lemma}

The most important class of central words $w(x_1,x_3)\in \gamma_3(\Gamma_{4,3})$ will be words of the form 
\[
T=T(m,n,l):= \left[x_1^m,x_3^n\right] \left[x_1^l,x_3\right],
\]
where $m$ and $n$ are integers and $l$ is an integer satisfying $0\leqslant |l| < |m|$.
In a sense they are the discrete prototype for the words $\Omega_k^j$ that we will introduce in \S \ref{sec:Omega} and then use throughout the remainder of the paper. The following observation is straight-forward
\begin{lemma}\label{lem:Obs-central-words-p-4}
 The equality $T(m,n,l)\equiv z^{mn+l}$ holds in $\Gamma_{3,3}$. Conversely, for every integer $k$ there are integers $m$, $n$, $l$ satisfying $T(m,n,l)\equiv z^k$, $|n|\leqslant |m|\leqslant 3|n|$, $0\leqslant |l| < |m|$ and ${\rm sgn}(mn)= {\rm sgn}(l)$.
\end{lemma}

We record the following simple consequence of Lemmas \ref{lem:ChangingFactors-p-4} and \ref{lem:Obs-central-words-p-4}:
\begin{lemma}\label{lem:Obs-central-words-p-4-conseq}
 There is a constant $C_1>0$ such that for every two words $T_1=T(m_1,n_1,l_1)$ and $T_2=T(m_2,n_2,l_2)$, their product $T_1 \cdot T_2$ can be transformed into a word $T_3=T(m_3,n_3,l_3)$ with
 \begin{enumerate}
  \item $m_3\cdot n_3 + l_3 = m_1 \cdot n_1 + l_1 + m_2\cdot n_2 + l_2$;
  \item $|m_3|, |n_3| \leqslant 3 \sqrt{|m_3\cdot n_3 +l_3|}$; and
  \item the identity $T_1 \cdot T_2 \equiv T_3$ holds with area $\leqslant C_1 \left(|m_1|+|n_1|+|m_2|+|n_2|\right)^2$\\ and diameter $\leqslant C_1 \left(|m_1|+|n_1|+|m_2|+|n_2|\right)$ in $\Gamma_{3,3}$ (and thus in $\Gamma_{4,3}$).
 \end{enumerate}
\end{lemma}

From this innocuous observation we deduce the subsequent lemma, which is the second key tool for our proof. We will use it in the case when $I= N$, in which it shows that a central null-homotopic word $w$ of the form $\prod_{i=1}^I T_i$ has area bounded by $C_2 N^3$. In particular, up to constants, its area is bounded by the function $n\mapsto n^{\frac{3}{2}}$ in $n=\ell (w)$, rather than by $n\mapsto n^2$, as one might a priori expect.
\begin{lemma}\label{lem:merging-central-words-p-4}
 Let $N,I>0$ and let $T_i=T(m_i,n_i,l_i)$, $1\leqslant i \leqslant I$ be words with $|m_i\cdot n_i + l_i| \leqslant N^2$ and $|m_i|,|n_i|\leqslant 3 N$. Assume that $\prod_{i=1}^I T_i$ is null-homotopic. There is a constant $C_2>0$ such that the identity
 \[
\prod_{i=1}^I T_i \equiv 1 
 \]
holds in $\Gamma_{4,3}$ with area $\leqslant C_2 \cdot I \cdot N^2$ and diameter $\leqslant C_2 \left( \cdot (I N^2)^{\frac{1}{3}}+N\right)$.
\end{lemma}

\begin{proof}
 The proof is by induction on $I$, with the result for $I=1$ being trivial. Assume that the result holds for $I\geqslant 1$ and let $\prod_{i=1}^{I+1} T_i$ be null-homotopic. Since $T_i\equiv z^{m_in_i +l_i}$ for $1\leqslant i \leqslant I+1$ is in the center of $\Gamma_{3,3}$ it follows that $\sum _{i=1}^{I+1} m_in_i+l_i=0$. In particular, there is some $i_0$ such that $T_{i_0}\cdot T_{i_0+1}\equiv z^k$ with
 \[
 |k|\leqslant {\rm{max}}\left\{|m_{i_0}\cdot n_{i_0}+l_{i_0}|,~ |m_{i_0+1}\cdot n_{i_0+1}+l_{i_0+1}|\right\} \leqslant N^2.
 \]

 By Lemma \ref{lem:Obs-central-words-p-4-conseq} there is a word $T'_{i_0}=T(m'_{i_0},n'_{i_0},l'_{i_0})$ which satisfies the identity $T'_{i_0}\equiv T_{i_0}\cdot T_{i_0+1}$ with area $\leqslant C_1 \cdot 12^2 N^2$, diameter $\leqslant C_1 \cdot 12\cdot N$ and such that, moreover, the word
 \begin{equation}\label{eqn:prefix-estimate-p-4}
  T_1\cdot \dots \cdot T_{i_0-1} \cdot T'_{i_0}\cdot T_{i_0+2}\cdot \dots \cdot T_{I}
 \end{equation}
satisfies the induction hypothesis for $I$. Choosing $C_2 \geqslant 12^2\cdot  C_1$ thus completes the assertion on the area. 

By Lemma \ref{lem:diameter} it suffices to show that the word diameter of the prefix word $T_1\cdot \dots \cdot T_{i_0-1}$ is $\lesssim \left( \left( I\cdot N^2\right)^{\frac{1}{3}}+N\right)$  in $\Gamma_{4,3}$ to obtain the desired diameter bound. However, this follows by observing that by assumption $\prod_{i=1}^{i_0-1}T_i \equiv z^t$ with $|t|\leqslant (i_0-1) N^2\leqslant I \cdot N^2$ and that the subgroup $\langle z \rangle \leqslant \Gamma_{4,3}$ is $n^{\frac{1}{3}}$-distorted \cite{Osindistort} (also see Lemma \ref{lem:sec-efficient-tech-1} below).
\end{proof}

We will now explain how to use Lemmas \ref{lem:ChangingFactors-p-4} and \ref{lem:merging-central-words-p-4} to show
\begin{theorem}\label{thmDehnL4ZH3}
$\Gamma_{4,3}$ admits $(n^3,n)$ as a filling pair. 
\end{theorem}

\begin{claim}
It suffices to prove that there is a constant $C>0$ such that all null-homotopic words $w=w(x_1,x_2)$ of length $\ell(w)\leqslant n$ admit a filling of area $\leqslant C n^3$ and diameter $\leqslant C n$ in $\Gamma_{4,3}$. 
\end{claim}

\begin{proof}
The subgroup generated by the $x_i$ intersects the subgroup generated by the $y_i$ in the central subgroup $\langle z \rangle$. Thus, given a null-homotopic word $u$ of length at most $n$ in the generators $x_i$ and $y_i$ of $\Gamma_{4,3}$, we can use the commutation relations $\left[x_i,y_j\right]=1$ and Lemma \ref{lem:ChangingFactors-p-4} to replace it by a word $v$ in the $x_i$ of the same length at cost $\leqslant K_1\cdot n^2$ and in a ball of diameter $\leqslant K_1 n$ for a suitable constant $K_1>0$. Using $O(n)$ relations of the form $\left[x_1,x_i\right]=x_{i+1}$ we can now replace $v$ by a null-homotopic word $w(x_1,x_2)$ of length bounded by $K_2 n$ for a suitable constant $K_2>0$. 
\end{proof}

\begin{claim}\label{claim:creating-x3s-p-4}
 There is a constant $C>0$ such that for all $n\in \ZZ$, the null-homotopic word $\left[x_2^n,x_1\right]x_3^n$ admits a filling of area $\leqslant C n^2$ and diameter $\leqslant C n$ in $\Gamma_{4,3}$.
 \label{lemAreaCommwnx1}
\end{claim}
\begin{proof}
The proof is straight-forward: consider $x_2^nx_1$ and move $x_1$ to the left, by commuting it with the $x_2$'s one by one, using the relation $\left[x_1,x_2\right]=x_3$. Then move all $x_3$'s produced in the process to the right using the relation $\left[x_2,x_3\right]=1$ (see also Proposition \ref{prop:binom-Lie} below).
\end{proof}

So let $w(x_1,x_2)$ be a null-homotopic word of length $\ell(w)\leqslant n$.

To obtain an upper bound on the area of $w(x_1,x_2)$ we will iteratively move all instances of $x_1$ in $w$ to the left, starting with the left-most. After moving an $x_1$ to the left we move all $x_3$'s created in the process to the left. As a consequence we will obtain a word of the form $T_i=T(m_i,n_i,0)$ with $|m_i|, |n_i| \leqslant n$, which we move to the right. 

After the $i$-th iteration of this process we may assume that we have a word of the form
\[
x_3^{k_1}x_1^{k_2}x_2^{k_3} x_1^{\pm 1} v(x_1,x_2) \prod _{j=0}^{i-1} T_{i-j},
\]
where $|k_2|+ \vert k_3 \vert +1+\ell(v(x_1,x_2))\leqslant n$ and $|k_1|\leqslant i \cdot n$.

Since the exponent sum of the $x_1$'s and $x_2$'s is zero, repeating this process $I\leqslant n$ times will yield a null-homotopic word
\[
x_3^{a}\prod _{j=0}^{I-1} T_{I-j}.
\]
Since $\prod _{j=0}^{I-1} T_{I-j}$ is in the center of $\Gamma_{4,3}$ it follows that it is null-homotopic and thus $a=0$. We now apply Lemma \ref{lem:merging-central-words-p-4} with $N:=n$ to conclude that $\prod_{j=0}^{I-1} T_{I-j}$ admits a filling of area $\leqslant C_2 I \cdot n^2 \leqslant C_2 n^3$ and diameter $\leqslant 2\cdot C_2 \cdot n$.

It remains to explain the $i+1$-th iteration of our procedure and to check that it has quadratically bounded area and linearly bounded diameter. It is here where we will make fundamental use of Lemma \ref{lem:ChangingFactors-p-4}. We will discuss the case $x_1^{+1}$, the case $x_1^{-1}$ being similar. The following identities hold in $\Gamma_{4,3}$:

\begingroup
\allowdisplaybreaks
\begin{align}
 & x_3^{k_1}x_1^{k_2}x_2^{k_3} x_1 v(x_1,x_2)\prod _{j=0}^{i-1} T_{i-j}\\
 \equiv &x_3^{k_1}x_1^{k_2} x_1 x_2^{k_3} x_3^{-k_3} v(x_1,x_2)\prod _{j=0}^{i-1} T_{i-j}\label{eqar1}\\
 \equiv&x_3^{k_1}x_1^{k_2+1} x_3^{-k_3} x_2^{k_3} v(x_1,x_2)\prod _{j=0}^{i-1} T_{i-j}\label{eqar2}\\
 \equiv&x_3^{k_1}x_3^{-k_3} x_1^{k_2+1} T(k_2+1,-k_3,0) x_2^{k_3} v(x_1,x_2)\prod _{j=0}^{i-1} T_{i-j}\label{eqar3}\\
 \equiv&x_3^{k_1-k_3} x_1^{k_2+1} \left[y_1^{k_2+1},y_3^{-k_3}\right] x_2^{k_3} v(x_1,x_2)\prod _{j=0}^{i-1} T_{i-j}\label{eqar4}\\
 \equiv&x_3^{k_1-k_3} x_1^{k_2+1} x_2^{k_3} v(x_1,x_2) \left[y_1^{k_2+1},y_3^{-k_3}\right] \prod _{j=0}^{i-1} T_{i-j}\label{eqar5}\\
 \equiv&x_3^{k_1-k_3} x_1^{k_2+1} x_2^{k_3} v(x_1,x_2) T(k_2+1,-k_3,0) \prod _{j=0}^{i-1} T_{i-j}\label{eqar6}
\end{align}
\endgroup

Setting $T_{i+1}=T(k_2+1,-k_3,0)$ completes the $i+1$-th step. We remark that in the case $x_1^{-1}$ we obtain new terms $x_3^{+k_3}$ and $T(k_2,k_3,0)$.

Using that $|k_2|+|k_3|+1+\ell(v(x_1,x_2))\leqslant n$ we obtain that the number of relations required to obtain consecutive lines of the equation is bounded as follows:
\begin{itemize}[leftmargin = 2.25cm]
\item[\eqref{eqar1}] $C n^2$ (by Claim \ref{claim:creating-x3s-p-4})
\item[\eqref{eqar2}] $n^2$ (using the relation $\left[x_2,x_3\right]=1$)
\item[\eqref{eqar4} \& \eqref{eqar6}] $C_0 n^2$ (by Lemma \ref{lem:ChangingFactors-p-4})
\item[\eqref{eqar5}] $4 n^2$ (using the relations $\left[x_i,y_j\right]=1$)
\end{itemize} 
In particular, there is a constant $C_3>0$ such that the total cost of this transformation is $\leqslant C_3 n^2$. Since we repeat this process $I\leqslant n$ times, this provides the desired area estimate in Theorem \ref{thmDehnL4ZH3}.

The subgroup $\langle x_3 \rangle \leqslant \Gamma_{4,3}$ is $n^{\frac{1}{2}}$-distorted \cite{Osindistort} (or Lemma \ref{lem:sec-efficient-tech-1} below), meaning that the prefix word of all of our transformations has diameter in $O(\sqrt{i\cdot n}+n)= O(n)$.  Thus, by combining the linear diameter bounds in Lemma \ref{lem:ChangingFactors-p-4} and Claim \ref{claim:creating-x3s-p-4} with Lemma \ref{lem:diameter}, we obtain that all of our transformations satisfy a linear diameter bound, completing the proof of Theorem \ref{thmDehnL4ZH3}.

\subsection{Developing a strategy for the proof for general \texorpdfstring{$p$}{p}}\label{sec:strategyp}

In some sense what made our proof work for $p=4$ is that this degree is low enough so that we could conveniently shift powers of $x_3$ to the left, central words of the form $T(n,m,l)$ to the right and keep the remainder of our word in $x_1$ and $x_2$ in the middle.  This allowed us to elegantly avoid and hide a key difficulty that makes any brute force attempt to generalize our approach to arbitrary values of $p$ fail: the distortion of terms in $\gamma_i(\Gamma_{p,p-1})$ being $n^{\frac{1}{i}}$,  the cost of ``naively'' creating and reordering powers of the $x_i$ will be much too high.  On the other hand the commuting trick exploiting the second factor (generated by the $y_i$'s) will only work for central words. 

We overcome these difficulties through a sequence of results that on the surface seem like a long list of technical lemmas, but really follow a concrete strategy designed to avoid the above obstacles. Moreover, it will turn out to be of great use to switch to the setting of compact presentations and work in the real Malcev completion $G_{p,p-1}$ rather than in the discrete group $\Gamma_{p,p-1}$. But for now let us pretend we work in $\Gamma_{p,p-1}$. For $k\geqslant 1$ and $\nnn=(n_1, \dots, n_k)\in \mathbb Z^k$, we let  $ \Omega_{k}(\nnn)$ be the following word in $x_1$ and $x_2$
\[
 \Omega_{k}(\nnn):= \left[x_1^{n_1},\dots,x_1^{n_{k-1}},x_2^{n_k}\right].
\]
We observe that  $\Omega_k(\nnn)$ corresponds to an element  of the $k$-th term of the lower central series of the free group generated by $x_1$ and $x_2$. In particular, for $k=p$ it defines a relation in $\Lambda_p$, and therefore in $\Gamma_{p,p}$ and $\Gamma_{p,p-1}$.
The non-technical key steps of our proof for general $p$ are:

\vspace{.2cm}\noindent {\bf Step 0:} Similar arguments as above allow us to reduce to words $w(x_1,x_2)$.

\vspace{.2cm}\noindent {\bf Step 1:} We use the results on efficient sets of words presented in \S \ref{subsec:efficient-words} to argue that we can reduce to null-homotopic words of the form
\[
 w(x_1,x_2)= x_1^{n_1}x_2^{m_1} \cdots x_1^{n_k}x_2^{m_k}
\]
with $|n_i|,|m_i|\leqslant n$ and $k$ uniformly bounded by some constant $C>0$.

\vspace{.2cm}\noindent {\bf Step 2:} By shifting the $x_1^{n_i}$'s to the left in blocks, we transform the word $w$ into a product of $\leqslant C'$ iterated commutators of the form $\Omega_{k_i}(\nnn_i)^{\pm 1}$, with $2\leqslant k_i\leqslant p-1$ and $\nnn_i\in \RR^{k_i}$ and order them by the size of the $k_i$ (for a suitable constant $C'>0$). This provides us with a word of length $\lesssim n$ that (at least morally) is very similar to a word in the Malcev normal form of \S \ref{subsec:compact-presentations}.

\vspace{.2cm}\noindent {\bf Step 3:} We consecutively merge all terms of the form $\Omega_{k}(\nnn_i)$ for increasing $k$, starting with $k=2$. Using that $w$ is null-homotopic this process will terminate in the trivial word. At any stage we will make sure that the remaining word stays of length $\lesssim n$.

\vspace{.2cm}

Note that for technical reasons the above steps don't appear in the precisely same order in \S \ref{sec:upper-bound-Dehn}. However, keeping them in mind when reading the proof should be helpful in understanding its structure. 

\vspace{.3cm}
The most difficult steps are Steps 2 and 3. Performing them essentially requires us to be able to do two things at sufficiently low cost:
\begin{enumerate}
 \item Merge two words of the form $\Omega_{k}(\nnn_1)$ and $\Omega_{k}(\nnn_2)$ into a new word of a similar form and of length $\lesssim n$.
 \item Commute certain types of words. In particular, we will have to commute words of the form $\Omega_{k_1}(\nnn_1)$ with words of the form $\Omega_{k_2}(\nnn_2)$ at cost $\lesssim n^{p-1}$.
\end{enumerate}

The bulk of the technical work in \S \ref{sec:upper-bound-Dehn} is concerned with resolving these two problems. Concretely, (1) will be resolved by Lemma \ref{lem:Strong-k-Lemma}, which we will often refer to as the Cancelling Lemma, while (2) will be resolved by Lemma \ref{lem:MainLemma}, which we will often refer to as the Main commuting Lemma. Note that the Cancelling Lemma and the Main commuting Lemma are in some sense beefed-up and considerably harder to prove versions of Lemma \ref{lem:merging-central-words-p-4} and of the commutation of terms enabled by Lemma \ref{lem:ChangingFactors-p-4}.

In fact we will first prove the Main commuting Lemma and then the Cancelling Lemma, as the former will be required in the proof of the latter. The proofs of both will be by a rather subtle double induction in $p$ and $k$ and will be divided into several auxiliary technical lemmas. Throughout the proofs of these results we will rely heavily on applying the fact that, by induction, $\delta_{\Gamma_{p-1,p-1}}\asymp n^{p-2}$ to rewrite words in the generators of the canonically embedded subgroup $\Gamma_{p-1,p-1}\hookrightarrow \Gamma_{p,p-1}$. Similar to the use of Lemma \ref{lem:ChangingFactors-p-4} in \S \ref{subsec:Proof-upper-bound-p-4}, we will also make essential use of the fact that we can replace words of length $n$ in $x_1$ and $x_3$ that are contained in $\gamma_{p-1}(\Gamma_{p-1,p-1})$ by words in $y_1$ and $y_3$ at cost $\lesssim n^{p-2}$, to enable us to commute them with words in the $x_i$ at a low cost. In particular, we will use this to start the induction in some of the technical Lemmas leading up to the Main commuting Lemma.

\section{Preliminaries for the general case}
\label{sec:preliminaries}

In this section we set the stage for the proof of the upper bound on the Dehn functions of $G_{p,p}$ and $G_{p,p-1}$ for general $p$. In \S  \ref{subsec:compact-presentations} we start by constructing explicit compact presentations. In \S \ref{subsec:efficient-words} we recall the notion of efficient words, which will allow us to restrict to certain families of simpler words when proving upper bounds on the Dehn functions. We then explain how to obtain such a set of efficient words with respect to our presentations. Finally, in \S \ref{subsec:techresultsUpperbounds} we prove some technical results that we will require in \S \ref{sec:upper-bound-Dehn} to compute upper bounds on diameters of fillings.

\subsection{Compact presentations of the groups \texorpdfstring{$\Gamma_{p,q}$}{Gammapq} and \texorpdfstring{$G_{p,q}$}{Gpq}}
\label{subsec:compact-presentations}

Recall from the introduction that $\Lambda_p$ denotes the model filiform group with presentation 
\begin{equation*}
\mathcal P (\Lambda_{p}) =
\left\langle 
\begin{array}{cccc} 
x_1,& x_2,\ldots & x_{p-1},& z \end{array} 
\left\mid  
\begin{array}{l} 
\left[x_1,x_i\right]x_{i+1}^{-1}, i = 2, \ldots, p-2\\
\left[x_i,x_j\right], i,j = 2,\ldots, p-1 \\ 
\left[x_1, x_{p-1}\right]z^{-1}, \left[x_i,z\right],~ i=1,\dots, p-1 \\
\end{array}  \right. 
\right\rangle
\end{equation*}
and $L_p$ denotes its real Malcev completion.
The group $\Gamma_{p,q}$ is defined as the central product of $\Lambda_p$ with $\Lambda_q$ for  $3\leqslant q\leqslant p$. 
We deduce the following finite presentation of $\Gamma_{p,q}$:
\begin{equation*}
\mathcal P(\Gamma_{p,q})=
\left\langle 
\begin{array}{cccc} 
x_1,& x_2,\ldots & x_{p-1},& z\\ 
y_1,& y_{p-q+2},\ldots & y_{p-1}, & \end{array} 
\left\mid  
\begin{array}{l} 
\left[x_1,x_i\right]x_{i+1}^{-1}, \left[y_1,y_i \right]y_{i+1}^{-1}, i = 2, \ldots, p-2\\
\left[x_i,y_j\right], i,j = 1,\ldots, p-1 \\ 
\left[x_1, x_{p-1}\right]z^{-1}, [y_1, y_{p-1}]z^{-1} \\
z \mbox{ central}
\end{array}  \right. 
\right\rangle.
\label{eqn:PresJpq}
\end{equation*}

Observe that for $\Gamma_{p,p-1}$ we purposefully used the notation $y_1,y_{p-q+2},\ldots, y_p$ instead of $y_1,y_2,\ldots, y_{q-1}$ as it allows us to see $\Gamma_{p,q}$ as a subgroup of $\Gamma_{p,p}$.
Actually, it will be more convenient to work with {\it compact presentations} of their respective Malcev completions  $G_{p,q}$. 
 We describe below a way to deduce a compact presentation of the group from a finite presentation of a lattice.

Let $\Gamma$ be a finitely-generated torsion-free nilpotent group. Then $\Gamma$ is strongly polycyclic, i.e. admits a composition series 
$
\Gamma = P^0 \Gamma \triangleright P^1 \Gamma \triangleright \cdots \triangleright P^n \Gamma = \lbrace 1 \rbrace
$
with $P^{i} \Gamma / {P^{i+1} \Gamma} = \mathbf Z$. It can be chosen to refine the lower central series, i.e. there exist integers $k_i$ such that $\gamma_i \Gamma = P^{k_i} \Gamma$ for all $i$ with suitable $k_i$.
Choosing representatives of the generators of the quotients $P^{i} \Gamma / {P^{i+1} \Gamma}$, one can build a generating set $S = \lbrace \gamma_1, \ldots, \gamma_n \rbrace$ such that $[\gamma_i, \gamma_j] \in \langle \gamma_{j+1}, \ldots, \gamma_n \rangle$ whenever $i \less j$ and every $\gamma \in \Gamma$ uniquely writes as $\gamma_1^{\ell_1} \cdots \gamma_n^{\ell_n}$ with $\ell_i \in \mathbf Z$.
$S$ is called a Malcev basis for $\Gamma$. 
\begin{example}\label{ex:Malcev}
Note that $S=\{x_1,x_2\}$ forms a generating subset of $\Lambda_p$ and that $\widehat S=\{x_1,\ldots, x_p\}$ is a Malcev basis. 
Similarly $T=\{x_1,x_2,y_1,y_{p-q+2}\}$ is a generating subset of $\Gamma_{p,q}$ and the set $\widehat{T}=\{x_1,x_2,\ldots, x_{p-1}, z, 
y_1, y_{p-q+2},\ldots, y_{p-1}\}$ is a Malcev basis.
\end{example}
With respect to the integer coordinates $\ell_i$ one can prove that the multiplication law is polynomial, i.e. that there are polynomials $M_1, \ldots, M_n \in \mathbf Z [X_1, \ldots, X_n, X'_1, \ldots, X'_n]$ such that $(\gamma_1^{\ell_1} \cdots \gamma_n^{\ell_n})\cdot (\gamma_1^{\ell'_1} \cdots \gamma_n^{\ell'_n}) \equiv \gamma_1^{M_1(\ell_1,\ell'_1)} \cdots \gamma_n^{M_n(\ell_n, \ell'_n)}$ \cite[5.1]{BuserKarcher}.
An effective way of constructing the Malcev completion of $\Gamma$ is to extend this polynomial law (denote it $\star$) from $\mathbf Z^n$ to $\mathbf R^n$. Let $G$ be any simply connected nilpotent Lie group containing $\Gamma$ as a lattice. Then the isomorphism $\Gamma \to (\mathbf Z^n, \star)$ extends to an isomorphism $G \to (\mathbf R^n, \star)$. This can be established independently of the existence part of the Malcev theorem \cite[Corollary 2 p.34]{RagDS} by Zariski-density arguments.

We shall use the following notation throughout: 
for $\gamma \in \Gamma$ and $a \in \mathbf R$ we denote $\gamma^a = \exp(a \log \gamma)$ and for all subsets $S\subset G$ and $A>0$ we define $S_A = \lbrace \gamma^a : a \in [-A,A], \gamma\in S \rbrace$. The subsequent result explains how one can obtain a compact presentation for a simply connected nilpotent Lie group $G$ starting with a Malcev basis of a lattice $\Gamma<G$.
\begin{proposition}
\label{prop:Malcev}
Let $\Gamma$ be a lattice in a simply connected nilpotent Lie group $G$ and let $A>0$. 
Let $\widehat{S}=\left\{\gamma_1,\dots,\gamma_n\right\}$ be a Malcev basis of $\Gamma$.
\begin{enumerate}
\item
\label{item:existence-monomials}
For $1\leqslant i<j\leqslant n$ there exist polynomials $P_{j+1}, \ldots, P_n \in \mathbf Z[X,Y]$ such that for all $\ell,m\in \ZZ$ the following equality holds in $\Gamma$:
\begin{equation*}
[\gamma_i^\ell, \gamma_j^m] \equiv \gamma_{j+1}^{P_{j+1}(\ell,m)} \cdots  \gamma_{n}^{P_{n}(\ell,m)}.
\end{equation*}
\item
\label{item:finite-presentation-of-lattice}
The set of freely reduced words $[\gamma_j, \gamma_i] \gamma_{j+1}^{P_{j+1}(1,1)} \cdots  \gamma_{n}^{P_{n}(1,1)}$ for $1\leqslant i \less j\leqslant n$ determines a presentation for $\Gamma$ over the generating set $\widehat{S}$.
\item
\label{item:compact-presentation-of-Lie}
The set of freely reduced words  $R_A = \lbrace \sigma_{i}(a,b) \rbrace \cup \lbrace\rho_{i,j}(a,b) \rbrace$ with
\begin{equation*}
\sigma_{i}(a,b) = \gamma_i^a \gamma_i^b (\gamma_i^{a+b})^{-1} \text{ and } \rho_{i,j}(a,b) = [\gamma_j^a, \gamma_i^b] \gamma_{j+1}^{P_{j+1}(a,b)} \cdots  \gamma_{n}^{P_{n}(a,b)} 
\end{equation*} for $i\less j$, $a,b \in [-A,A]$ determines a presentation for $G$ over the generating set $\widehat{S}_A$.
\end{enumerate}
\end{proposition}

\begin{proof}
\eqref{item:existence-monomials} is a direct consequence of the existence of the polynomials $M_1, \ldots, M_n$ and the construction of $\widehat{S}$ from a refinement of the lower central series.
For \eqref{item:finite-presentation-of-lattice} note that these relations allow us to transform any word over $\widehat{S}$ into its Malcev normal form $\gamma_1^{\ell_1} \cdots \gamma_n^{\ell_n}$. Finally, we prove \eqref{item:compact-presentation-of-Lie} in three steps:
\begin{itemize}
\item
$\widehat{S}_A$ is a generating set: this is clear from the isomorphism $G\to (\mathbf R^n, \star)$. Moreover, $\widehat{S}_A$ is compact as image of a compact set under the exponential map.
\item
The relations in $R_A$ hold in $G$, i.e. they lie in $\ker (F_{\widehat S_A} \to G)$: $a \mapsto \gamma^a$ defines a group homomorphism by construction, so the $\sigma_i(a,b)$ hold. To prove that the $\rho_{i,j}(a,b)$ hold let $\varphi$ be any linear form on the Lie algebra $\mathfrak{g}$ of $G$ and define $\pi(a,b)  := \varphi(\log [\rho_{i,j}]_G)$ (where $[\cdot]_G$ denotes the evaluation in $G$). Then \eqref{item:finite-presentation-of-lattice} implies that $\varphi(a,b) = 0$ for all $(a,b) \in \mathbf Z^2$. On the other hand $\pi$ is a polynomial function by the Baker-Campbell-Hausdorff formula. We deduce that it is identically $0$ on $\mathbf R^2$ and therefore that $\rho_{i,j}(a,b)$ holds for all $a,b$.
\item
As in \eqref{item:finite-presentation-of-lattice} the relations in $R_A$ allow us to transform any product of powers of elements in $\widehat{S}_A$ into its normal form $\gamma_1^{a_1} \cdots \gamma_n^{a_n}$. Hence, the normal subgroup of $G$ generated by $R_A$ coincides with $\ker (F_{\widehat S_A} \to G)$. \qedhere
\end{itemize}
\end{proof}

\begin{remark}
Compact presentations offer a technical advantage over finite presentations when manipulating words as they allow to reduce length.
For instance, representing a central element in $H_5(\mathbf{Z})$ by a short length word over $\widehat{S}$ needs a product of two commutators due to divisibility issues (compare \cite{OlsSapCombDehn} and \S \ref{subsec:Proof-upper-bound-p-4}) while a single one is sufficient over $\widehat{S}_A$. 
\end{remark}

\begin{remark}
\label{rmk:Fixing-choice-for-A}
 For our purposes it will suffice to consider only the case $A=1$ and we will restrict to it in \S \ref{sec:upper-bound-Dehn}. However, producing a presentation for general $A$ is no harder and might be useful for future applications. Hence, we write our results in this general context in this section.
\end{remark}

\begin{convention*}
From now on we will omit the relations $\sigma_i(a,b)$ from our compact presentations to simplify notation, as they are rather self-explanatory.
\end{convention*}

To obtain an explicit compact presentation for $G_{p,q}$ we compute the polynomials $P_{i,j}$ corresponding to the Malcev basis $\widehat{S}$. 

\begin{proposition}
\label{prop:binom-Lie}
For $a,b \in \mathbf R$ the following relation holds in $L_p$:
\begin{equation}
\label{eq:Malcev-commutator-monomials-in-the-filiform}
[x_1^a, x_i^b]\equiv x_{i+1}^{ab} x_{i+2}^{-\binom{a}{2}b} x_{i+3}^{\binom{a}{3}b} \cdots z^{(-1)^{p+i+1}\binom{a}{p-i}b}.
\end{equation}
In particular, let $S=\{x_1, x_2\}$ and $\widehat{S}=\{x_1,\ldots,x_{p-1},z\}$. Then for every $A>0$ the set $S_A$ is a compact generating subset of 
 $L_p$ and the latter admits a compact presentation $\mathcal P_{A}(L_p)$ given by the generating subset $\widehat{S}_A$ and the relators 
\[ R_A=\{[x_1^a, x_i^b] = x_{i+1}^{ab} x_{i+2}^{-\binom{a}{2}b} x_{i+3}^{\binom{a}{3}b} \cdots z^{(-1)^{p+i+1}\binom{a}{p-i}b}\}, \]
for $2\leqslant i\leqslant p-1$ and $a,b\in [-A,A]$.
Moreover, for $a,b\in \mathbf R$ the identity \eqref{eq:Malcev-commutator-monomials-in-the-filiform} admits a filling of area $\lesssim_{p,A} a^{p-i+1}b^2$ and diameter $\lesssim_{p,A} |a|+|b|$ in $\mathcal{P}_{A}(L_p)$.
\end{proposition}

\begin{proof}[Proof of Proposition \ref{prop:binom-Lie}]
It suffices to prove the formula and area estimate for $i=2$ since $\langle x_1, x_i \rangle \cong \Lambda_{p+2-i}$ with $x_1 \mapsto x_1$ and $x_2 \mapsto x_i$ defines an isomorphism. The first step is to prove $[x_1, x_2^b] = x_3^b$ for every $b$; this is obtained by induction on $b$ (for $b$ an integer) and we deduce the area and diameter estimates $O(b^2)$ and $O(b)$ respectively. We now assume the formula for $(a,b)$, denoting its area by $\operatorname{Area}(a,b)$, and consider $x_1^{a+1} x_2^b$. In the following calculation we record the cost on the right.
\begin{align}
x_1^{a+1}x_2^b  = x_1^a x_1x_2^b 
& \equiv x_1^a x_2^b x_1 x_3^b \tag{Area $b$} \\
& \equiv x_2^b x_1^a x_{3}^{ab} x_{4}^{-\binom{a}{2}b} x_{5}^{\binom{a}{3}b} \cdots z^{(-1)^{p+1}\binom{a}{p-2}b} x_1 x_3^b \tag{$\operatorname{Area}(a,b)$} \\
& \equiv x_2^b x_1^{a+1} x_{3}^{ab} x_{4}^{-\binom{a+1}{2}b} x_{5}^{\binom{a+1}{3}b} \cdots z^{(-1)^{p+1}\binom{a+1}{p-2}b} x_3^b \tag{Area $b^2 \sum_{j = 1}^{p-2} \binom{a}{j}$} \\
& \equiv x_2^b x_1^{a+1} x_{3}^{(a+1)b} x_{4}^{-\binom{a+1}{2}b} x_{5}^{\binom{a+1}{3}b} \cdots z^{(-1)^{p+1}\binom{a+1}{p-2}b}  \tag{Area $b \sum_{j = 1}^{p-2} \binom{a+1}{j}$}.
\end{align}
We provide some explanations for our transformations: on the third line the rightmost $x_1$ is brought to the left which creates $x_j$-terms for $j \geqslant 4$; they are gathered with the previous ones. On the fourth line the rightmost $x_3^b$ is brought to the left and no new term is produced since $x_3$ commutes with all the $x_j$ for $j \geqslant 4$. 

We deduce from our estimates that 
\begin{align*}
\operatorname{Area}(a+1,b) & \leqslant \operatorname{Area}(a,b) + b + C b^2 a^{p-2} + C'b (a+1)^{p-2},
\end{align*}
where $C$ and $C'$ are positive constants, and thus that $\operatorname{Area}(a,b) =O_{p,A}( a^{p-1}b^2)$ by induction on $a$.

For the diameter bound observe that the $i$-th term of the lower central series is $n^{\frac{1}{i}}$-distorted \cite{Osindistort} (see also Lemma \ref{lem:sec-efficient-tech-1} below). Thus all prefix words of transformations appearing above have diameter in $O_{p,A}(|a|+|b|)$ and we conclude by Lemma \ref{lem:diameter} that our filling for \eqref{eq:Malcev-commutator-monomials-in-the-filiform} has diameter $\lesssim_{p,A} |a|+|b|$.

Finally, the remaining properties follow from Example \ref{ex:Malcev} and Proposition \ref{prop:Malcev}.
\end{proof}
Combining Proposition \ref{prop:binom-Lie} and the fact that $G_{p,q}$ is the central product of $L_p$ with $L_q$, we deduce the following compact presentation of $G_{p,q}$.
\begin{corollary}\label{cor:presentation-of-Gpq}
For $3\leqslant q\leqslant p$, a compact presentation of $G_{p,q}$ is given for every $A>0$ by $\mathcal P_A(G_{p,q})=\langle \widehat{T}_A \mid R_A \rangle$, where $\widehat{T}=\{x_1,x_2,\ldots, x_{p-1},x_p, z, 
y_1, y_{p-q+1},\ldots, y_{p-1},y_p\}$, and
\[
R_A  = \left\{\begin{array}{l} \left[x_1^a, x_i^b\right] = x_{i+1}^{ab} x_{i+2}^{-\binom{a}{2}b} x_{i+3}^{\binom{a}{3}b} \cdots z^{(-1)^{p+i+1}\binom{a}{p-i}b},~ 2\leqslant i \leqslant p \\
 \left[y_1^a, y_i^b\right] = y_{i+1}^{ab} y_{i+2}^{-\binom{a}{2}b} y_{i+3}^{\binom{a}{3}b} \cdots z^{(-1)^{p+i+1}\binom{a}{p-i}b},~ p-q+1\leqslant i \leqslant p,\\ 
\left[x_i^a, y_j^b\right]=1,~ 1\leqslant i,j\leqslant p\\ 
 z^a = x_p^a = y_p^a,  ~ a,b \in [-A,A] \end{array} \right\}.
\]
\end{corollary}

We end this section by recalling the following well-known free equalities that hold in every group and that we will require at many points throughout the remainder of this work.
\begin{lemma}
\label{lem:FreeIdentities}
 Let $G$ be a group and let $u,v,w$ be words in some generating set for $G$. Then the following free identities hold:
 \begin{enumerate}
  \item $\left[u\cdot v, w\right] \equiv \left[u,w\right]^v\cdot \left[v,w\right]$;
  \item $\left[u, v\cdot w\right] \equiv \left[u,w\right] \cdot \left[u,v\right]^w$;
  \item $u^w\equiv u\left[u,w\right]$.
 \end{enumerate}
\end{lemma}
\subsection{A family of special words}\label{sec:Omega}
We now introduce a family of words that will play a crucial role in the following sections. For $p\geqslant j\geqslant 2$, $k\geqslant 1$ and $\nnn=(n_1, \dots, n_k)\in \RR^k$ we let  $ \Omega_{k}^j(\nnn)$ be the word  
\[
 \Omega_{k}^j(\nnn):= \left[x_1^{n_1},\dots,x_1^{n_{k-1}},x_j^{n_k}\right].
\]
for $k\geq 2$ and $x_j^{n_1}$ if $k=1$. For $j=2$, we shall simply denote it\footnote{A notation that we had already introduced in our sketch of proof in \S \ref{sec:strategyp}.} by $\Omega_k(\nnn)$.

We observe that although $\Omega_{k}^j(\nnn)$ is a priori defined as a word in $S_\infty$ we can view it as an element of $F_{S_A}$ by identifying $x_i^{n_i}$ with a product of $\lceil |n_i|/A\rceil$ letters of the form $x_1^{t_i}$ with $|t_i|\leqslant A$. In what follows such identifications will be made implicitly.  Using that $G_{p,p-1}$ is $(p-1)$-nilpotent and Proposition \ref{prop:binom-Lie}, we easily deduce the following useful identities. 
\begin{lemma}\label{lem:OmegaIdentity}
For all $2\leq j \leq p-1$ and $\nnn=\left(n_1,\dots, n_{p-j+1}\right)\in \RR^{p-j+1}$ 
\[\Omega_{p-1}^j(\nnn)=z^{n_1\cdots n_{p-j+1}}.\]
In particular, for all $\nnn\in \RR^{p-1}$ there exists $m\in \RR$ with $|m|\lesssim_p |\nnn|:=|n_1|+\dots + |n_{p-1}|$, such that 
\[\Omega_{p-1}(\nnn)=\Omega_{p-2}^3(|\nnn|,\dots,|\nnn|)^{m}.\]
\end{lemma}

\subsection{Reduction to products of efficient words}
\label{subsec:efficient-words} 
We build here on \cite{CornulierTesseraDehnBaums}.  We let $S$ be a finite alphabet and let $F_S$ denote the free group on $S$. Given a subset $\mathcal F \subset F_S$ and an integer $k\geqslant 1$, we denote $\mathcal F[k]$ the collection of concatenations of at most $k$ words in $\mathcal F$.

\begin{definition}
\label{def:efficient-family-of-words}
Given an integer $r\geqslant 1$, a subset $\mathcal F \subset F_S$ is called $r$-efficient with respect to a presentation $\langle S \mid R \rangle$ of a group $G$ if there exists a constant $C$ such that for every $w\in F_S$ there exists $w' \in \mathcal F[r]$ such that $w \equiv w' \mod \langle \langle R \rangle \rangle$ and $\ell(w') \leqslant C \ell(w)$. 
\end{definition}

Given a set $\mathcal F$ of words in $S$, we shall say that we have a filling pair $(f,g)$ for $G$ \emph{in restriction to words in $\mathcal F$} if every relation of length $n$ that lies in $\mathcal F$ admits a filling of area in $O(f(n))$ and filling diameter in $O(g(n)).$

The following is based on an original observation of Gromov \cite[{5.$A''_3$}]{AsInv}.

\begin{proposition}[{\cite[Proposition 4.3]{CornulierTesseraDehnBaums}}]
\label{prop:CorTesEff1}
Let  $s \ggreater 1$. Assume that $\mathcal F$ is $r$-efficient for some $r\geqslant 1$ and that $(n^s,n)$ is a filling pair for $G$ in restriction to 
$\mathcal F[k]$ for all $k\geqslant 1$. Then $(n^s,n)$ is a filling pair for $G$.
\end{proposition}
\begin{proof}
The  statement of  \cite[Proposition 4.3]{CornulierTesseraDehnBaums} is that $n^s$ is an isoperimetric function for $G$. However, it is easy to deduce from its proof that $(n^s,n)$ is a filling pair. Indeed, the proof consists of filling a loop of length $n$ using $k$ loops of length in $O(n/k)$ and a loop $\gamma'$ in $\mathcal F[k]$ of length in $O(n)$. 
While the argument used in \cite{CornulierTesseraDehnBaums} to obtain
the desired area bounds applies for very general functions, it is not
hard to check that using their methods one can actually produce a
filling of area $\lesssim n^{s}$ by iterating this procedure $\log_k(n)$
times. In particular, this yields the existence of such a filling of
$\gamma$ of diameter in $O(\sum_{j\geqslant 1} n/k^j)=O(n)$.
\end{proof}

We recall that $S=\{x_1,x_2\}\subset \Lambda_p$. We define the subset $\mathcal F_A\subset F_{S_A}$ of all powers of elements in $S_A$:
\[
\mathcal F_A:= \left\{s^n\mid s\in S_A, n\in \mathbf{N}\right\}.
\]
The main result of this section is
\begin{proposition}\label{prop:F_AB-is-efficient}
For all $p\geqslant 3$ and $A>0$ the subset $\mathcal F_A$ is $O_p(1)$-efficient with respect to the compact presentation $\mathcal P_A(L_{p})$ of $L_p$ provided by Proposition \ref{prop:binom-Lie}.
\end{proposition}
We immediately deduce the following corollary, which is the statement we shall need in our proof of the upper bound of the Dehn function. Define
\[
T_A:= \left\{ x_1^{a_1},x_2^{a_2}, y_1^{a_3},y_3^{a_4}\mid |a_1|,|a_2|,|a_3|,|a_4|\leqslant A\right\} \subset G_{p,p-1}
\]
and
\[
\mathcal G_A:= \left\{s^n\mid s\in T_A, n\in \mathbf{N}\right\}.
\]

\begin{corollary}\label{prop:G_A-is-efficient}
For all $p\geqslant 4$ and $A>0$ the subset $\mathcal G_A$ is $O_p(1)$-efficient with respect to the compact presentation of $G_{p,p-1}$ provided by Corollary \ref{cor:presentation-of-Gpq}.
\end{corollary}

Cyclic subgroups of the $i$-th term of the descending central series have distortion in $n^{1/i}$ (\cite{Osindistort}). The following lemmas provide related estimates that will be required in various places of our proof.

\begin{lemma}\label{lem:sec-efficient-tech-1}
Let $b\in \RR$ and let $2\leqslant i \leqslant p$. Then $x_i^b\equiv w \mbox{ mod } \left\langle\left\langle R_A\right\rangle\right\rangle$ for a word $w\in \mathcal{F}_{A}[O_p(1)]$ satisfying
\[
 \ell(w)=O_p(b^{\frac{1}{i-1}}) + O_p(1).
\]
In particular, $S_A$ is a generating subset of $L_p$.
\end{lemma}

\begin{proof}

The proof is by descending induction on $i$. 
By Lemma \ref{lem:OmegaIdentity}, we have
\begin{equation}\label{eq:Omega_p-1}
 \Omega_{p-1}\left(b^{\frac{1}{p-1}},\dots,b^{\frac{1}{p-1}}\right) \equiv z^b \mbox{ mod } \left\langle\left\langle R_A\right\rangle\right\rangle,
\end{equation}
proving the case $i=p$.

Now assume that the result holds for $i=i_0+1$ and let  $\beta:=b^{\frac{1}{i_0-1}}$. Observe that an iterated application of Proposition \ref{prop:binom-Lie}, Lemma \ref{lem:FreeIdentities}(2) and the fact that $L_p$ is metabelian to the innermost commutator yields the following identities in $L_p$ (i.e.\ modulo $R_A$):
\begin{align*}
 \Omega_{i_0-1}(\beta,\dots,\beta) &\equiv \prod_{j_1=3}^p \Omega_{i_0-2}^{j_1}\left(\beta,\dots,\beta,(-1)^{j_1+1} \binom{\beta}{j_1-2} \beta\right)\\
 &\equiv \dots\\
 &\equiv \prod_{3\leqslant j_1<\dots<j_{i_0-2}\leqslant p} x_{j_{i_0-2}}^{(-1)^{j_1+1}(-1)^{j_2-j_1+1}\cdots (-1)^{j_{i_0-2}-j_{i_0-3}+1}\binom{\beta}{j_1-2}\binom{\beta}{j_2-j_1}\cdots\binom{\beta}{j_{i_0-2}-j_{i_0-3}} \beta}.
\end{align*}
Because $\binom{x}{j}$ is a polynomial of degree $j$ in $x$, we deduce that for any choice of $3\leqslant j_1<\dots<j_{i_0-2}\leqslant p$ the exponent of $x_{j_{i_0-2}}$ is a polynomial of degree $j_{i_0-2}-1$ in $\beta$. Since there are only finitely many terms for each index $i_0\leqslant j_{i_0-2}\leqslant p$, we deduce that there are polynomials $q_j(\beta)$ of degree $j-1$ for $i_0\leqslant j \leqslant p$ such that
\[
 \Omega_{i_0-1}(\beta,\dots,\beta)\equiv\prod_{j=i_0}^p x_j^{q_j(\beta)}.
\]
Finally, an explicit evaluation shows that $q_{i_0}(\beta)=\binom{\beta}{1}^{i_0-2}\cdot \beta=\beta^{i_0-1}$ and we deduce that
\[
x_{i_0}^b \equiv \Omega_{i_0-1}(\beta,\dots,\beta) \cdot \prod_{j=i_0+1}^p x_j^{-q_j(\beta)} \mbox{ mod } \left\langle\left\langle R_A\right\rangle\right\rangle.
\]  
The result now follows by applying the induction hypothesis to the $x_j^{-q_j(\beta)}$.
\end{proof}

\begin{lemma}\label{lem:sec-efficient-tech-2}
 For $m_1,\dots,m_k,n_1,\dots,n_k\in \RR$ let $w=x_2^{m_1}x_1^{n_1}\cdot \dots \cdot x_2^{m_k} x_1^{n_k}$ and let $l:= \sum_{i=1}^k(|m_i|+|n_i|)$. There exist $b_1,\dots,b_p\in \RR$, with $|b_1|=O_p(l) + O_p(1)$ and $|b_i|=O_p(l^{i-1})+O_p(1)$ for $2\leqslant i \leqslant p$, such that
 \[
  w\equiv x_1^{b_1}\cdot \dots \cdot x_p^{b_p} \mbox{ mod } \left\langle \left \langle R_A\right \rangle \right \rangle.
 \] 
\end{lemma}
\begin{proof}
We will move all $x_1$'s in $w$ to the left to put the word in normal form. Setting $n_0=0$ and introducing the notation $\widetilde{n}_i:= \sum_{j=i}^k n_j$ we first observe that, by Proposition \ref{prop:binom-Lie}, the identity
\[
 x_1^{n_{i-1}} x_2^{m_i} x_1^{\widetilde{n}_i} \equiv x_1^{\widetilde{n}_{i-1}}x_2^{m_i}x_3^{-\widetilde{n}_i m_i}x_4^{\binom{\widetilde{n}_i}{2} m_i}\cdot \dots \cdot x_p^{(-1)^p \binom{\widetilde{n}_i}{p-2} m_i}
\]
holds in $L_p$ for $1\leqslant i \leqslant k$. Thus, moving powers of $x_1$ to the left, starting with the rightmost one, and $\left[x_i,x_j\right]=1$ for $i,j\geqslant 2$ imply that
\[
 w\equiv x_1^{\widetilde{n}_1}\cdot x_2^{\sum_{i=1}^k m_i}\cdot x_3^{-\sum_{i=1}^k\widetilde{n}_i m_i} \cdot x_4^{\sum_{i=1}^k \binom{\widetilde{n}_i}{2}m_i} \cdot \dots \cdot x_p^{(-1)^p \sum_{i=1}^k \binom{\widetilde{n}_i}{p-2}m_i} \mbox{ mod } \left\langle\left\langle R_A\right\rangle\right\rangle.
\]
Set $b_1:=\widetilde{n}_1$ and $b_j:= (-1)^j \sum_{i=1}^k \binom{\widetilde{n}_i}{j-2}m_i$. Using that $\binom{x}{j}$ is a polynomial of degree $j$ in $x$ and that $|\widetilde{n}_i|\leqslant l$, it is now easy to deduce that $|b_i|=O_p(l^{i-1})+O_p(1)$. This completes the proof.
\end{proof}

We will now explain how to derive Proposition \ref{prop:F_AB-is-efficient} from Lemmas \ref{lem:sec-efficient-tech-1} and \ref{lem:sec-efficient-tech-2}.

\begin{proof}[Proof of Proposition \ref{prop:F_AB-is-efficient}]
Since $S_A$ is a compact generating subset of $L_p$, it is enough to consider words in $S_A$.
Let $w=x_2^{m_1}x_1^{n_1}\cdot \dots \cdot x_2^{m_k}x_1^{n_k}$ be a word in $S_A$ of length $\ell(w)$.
By Lemma \ref{lem:sec-efficient-tech-2} there exist $b_1,\dots,b_p\in \RR$ with $|b_1|=O_p(\ell(w))+O_p(1)$ and $|b_i|= O_p(\ell(w)^{i-1})+O_p(1)$, $2\leqslant i\leqslant p$, such that
\[
w\equiv x_1^{b_1}\cdot \dots \cdot x_p^{b_p}\;{\rm mod} \left\langle\left\langle R_A\right\rangle\right\rangle.
\]
Lemma \ref{lem:sec-efficient-tech-1}  implies that there exist words $u_j\in \mathcal{F}_{A}[O_p(1)]$ with 
\[
 x_j^{b_j}\equiv u_j \; {\rm mod} \left\langle\left\langle R_A \right\rangle\right\rangle
\]
and $\ell(u_j)=O_p(b_j^{\frac{1}{j-1}}) + O_p(1)$ for $2\leqslant j \leqslant p$. Note, moreover, that $u_1=x_1^{b_1}\in \mathcal{F}_{A}$ and $\ell(u_1)=O_p(\ell(w)) +O_p(1)$.

Observe that the word $u:= u_1\cdot \dots \cdot u_p$ satisfies $w\equiv u \mbox{ mod } \left\langle \left\langle R_A\right\rangle\right\rangle$ and $u\in \mathcal{F}_{A}[O_p(1)]$. Moreover, a direct calculation shows that  $\ell(u)=O_p( \ell(w)) + O_p(1).$
This shows that $\mathcal{F}_{A}$ is $O_p(1)$-efficient, ending the proof of the proposition.
\end{proof}

\subsection{Upper bounds on diameters}
\label{subsec:techresultsUpperbounds}
We conclude this section by recording a few results which we will require to show that all fillings in \S \ref{sec:upper-bound-Dehn} have linearly bounded diameter.

\begin{lemma}\label{lem:distance-estimates}
 Let $I\geqslant 0$,  and let $j\leqslant k$ be two integers in $\{2,\ldots, p-1\}$, and, for $1\leqslant i \leqslant I$, let $u_i=u_i(x_1,x_j)$ be a word of word length $n_i=\ell(u_i)\geqslant 1$ such that $u_i$ represents an element in $\gamma_k(L_p)$. Then the element $g\in L_p$ represented by the word $w=\prod_{i=1}^I u_i$ satisfies
 \[
   |g|_{S_A}\lesssim_{p} \sum_{m=k+1}^p\left(\sum_{i=1}^I n_i^{m-j+1}\right)^{\frac{1}{m-1}}.
 \]
Moreover, $w$ has word diameter $\lesssim_p \sum_{m=k+1}^p\left(\sum_{i=1}^I n_i^{m-j+1}\right)^{\frac{1}{m-1}}+ {\rm max}_{i\in I} n_i$.
\end{lemma}
\begin{proof}
 The subgroup of $L_p$ generated by $x_1$ and $x_j$ is isomorphic to $L_{p-j+2}$ and there is a canonical embedding $L_{p-j+2}\hookrightarrow L_p$ induced by an embedding of presentations. Thus, by Lemma \ref{lem:sec-efficient-tech-2} for $L_{p-j+2}$, there are $b_{m,i}\in \RR$ such that
 \[
  u_i\equiv x_{k+1}^{b_{k+1,i}}\cdots x_p^{b_{p,i}} \mbox{ mod } \langle\langle R_A \rangle \rangle
 \]
 with $|b_{m,i}|\lesssim_p 1+ n_i^{m-j+1}\lesssim_p n_i^{m-j+1}$, for $k+1\leqslant m \leqslant p$ and $1\leqslant i\leqslant I$. We deduce that
 \[
  \prod_{i=1}^I u_i \equiv x_{k+1}^{b_{k+1}}\cdots x_p^{b_p} \mbox{ mod } \langle\langle R_A \rangle \rangle 
 \]
 for $b_m:= \sum_{i=1}^I b_{m,i}$. In particular,
 \[
  |b_m|\lesssim_p  \sum_{i=1}^I n_i^{m-j+1}.
 \]
 By Lemma \ref{lem:sec-efficient-tech-1} there is a word $w=w_{k+1}\cdot \dots \cdot w_p$ with
 \[
  w\equiv \prod_{i=1}^I u_i \mbox{ mod } \langle \langle R_A \rangle \rangle,
 \]
 $w_m \equiv x_m^{b_m} \mbox{ mod } \langle \langle R_A \rangle \rangle$ and
 \[
  \ell(w)\leqslant \sum_{m=k+1}^p \ell(w_m)\lesssim_p \sum_{m=k+1}^p|b_m|^{\frac{1}{m-1}} \lesssim_p \sum_{m=k+1}^p \left(\sum_{i=1}^I n_i^{m-j+1}\right)^{\frac{1}{m-1}}.
 \]
 This completes the proof.
\end{proof}

\begin{corollary}\label{cor:distance-estimates}
 Assume that in Lemma \ref{lem:distance-estimates}, $n$ is a positive integer such that $n_i\leqslant n$ for $1\leqslant i \leqslant I$. Then in the conclusion we obtain
\[
  |g|_{S_A}\lesssim_{p} \sum_{m=k+1}^p \left( I^{\frac{1}{m-1}}\cdot n^{1-\frac{j-2}{m-1}}\right).
\]
 In particular, for $I\leqslant n$ and $j=3$ we deduce that
\[
  |g|_{S_A}\lesssim_{p} n.
\]
 and $w$ has word diameter $\lesssim_{p} n$.
\end{corollary}

In a second application of Lemma \ref{lem:distance-estimates} we will require the following estimate.
\begin{lemma}\label{lem:Geom-series-diameter}
 For $n,k,B\geqslant 1$, $p\geqslant k+1$ and $1\leqslant j \leqslant \left \lceil \log_2(n)\right \rceil=:l$  there is a constant $C=C(p,B)$ such that the following inequality holds:
\[
 \sum_{m=k+1}^p \left( 2^{jk} \cdot \left(\frac{n}{2^j}\right) ^{m-2} + B\cdot \sum_{i=1}^j 2^{(i-1)k} \left(\frac{n}{2^i}\right)^{m-2}\right)^{\frac{1}{m-1}}\leqslant C \cdot n
\]
\end{lemma}
\begin{proof}
 First observe that by definition of $l$:
 \[
  \sum_{m=k+1}^p \left( 2^{jk} \cdot \left(\frac{n}{2^j}\right) ^{m-2} + B\cdot \sum_{i=1}^j 2^{(i-1)k} \left(\frac{n}{2^i}\right)^{m-2}\right)^{\frac{1}{m-1}}
  \leqslant n \cdot \sum_{m=k+1}^p \left(\frac{2^{j(k-m+1)}}{2^{l-j}} + B\cdot \sum_{i=1}^j \frac{2^{i(k-m+1)}}{2^{l-i}}\right)^{\frac{1}{m-1}}.
 \]
Since $k-m+1\leqslant 0$ and $l\geqslant j$ we now deduce from the geometric series that
\[
\sum_{m=k+1}^p \left(\frac{2^{j(k-m+1)}}{2^{l-j}} + B\cdot \sum_{i=1}^j \frac{2^{i(k-m+1)}}{2^{l-i}}\right)^{\frac{1}{m-1}}\leqslant \sum_{m=k+1}^p \left(1 + 2\cdot B\right)^{\frac{1}{m-1}}\leqslant 2 \cdot p \cdot B.
\]
This completes the proof.
\end{proof}

\section{Upper bounds on the Dehn functions of \texorpdfstring{$G_{p,p}$}{G{p,p}} and \texorpdfstring{$G_{p,p-1}$}{G{p,p-1}}}
\label{sec:upper-bound-Dehn}

In this section we will derive upper bounds on the Dehn functions of $G_{p,p}$ and $G_{p,p-1}$. In \S \ref{sec:Intro-proof-main-theorem} we state a sequence of auxiliary results and explain how they are used to prove the desired upper bounds by induction on $p$. This will be visualized by Figure \ref{fig:structure-proof-upper-bound}. In the remaining sections we then prove these results in the described order, finishing with the proof of the main result in \S \ref{subsec:PfMainThm}.

\subsection{Main theorem and structure of the proof}\label{sec:Intro-proof-main-theorem}
 The goal of this section is to prove the following key result of our paper. 
\begin{theorem}[{Main Theorem}]
 \label{thm:Upperbound}
  For $p\geqslant 4$, $(n^{p-1},n)$ is a filling pair for both $G_{p,p}$ and $G_{p,p-1}$.  
\end{theorem}
 The proof proceeds by induction on $p$.
We will see that for both groups we can reduce to null-homotopic words of the form $w(x_1,x_2)$, where $x_1$ and $x_2$ generate the first factor  (see \S \ref{subsec:PfMainThm}). In view of the canonical embedding $G_{p,p-1}\hookrightarrow G_{p,p}$, we deduce that it is enough to show that $\delta_{G_{p,p-1}}(n)\preccurlyeq n^{p-1}$ (see Lemma \ref{lem:OneFactor}). The core of the proof consists in deducing from $\delta_{G_{p-1,p-1}}(n)\preccurlyeq n^{p-2}$ that $\delta_{G_{p,p-1}}(n)\preccurlyeq n^{p-1}$. 

We recall the following notation, for every $A>0$: 
\[
T_{A}:= \left\{ x_1^{a_1},x_2^{a_2}, y_1^{a_3},y_3^{a_4}\mid |a_1|,|a_2|,|a_3|,|a_4|\leqslant A\right\} \quad \text{and}\quad S_{A}:= \left\{ x_1^{a_1},x_2^{a_2}\mid |a_1|,|a_2|\leqslant A\right\},
\]

and
\[
\mathcal G_{A}:= \left\{s^n\mid s\in T_A, n\in \mathbf{N}\right\} \quad \text{and}\quad \mathcal F_{A}:= \left\{s^n\mid s\in S_A, n\in \mathbf{N}\right\}.
\]

We will fix $A=1$ once and forever and will omit the prefix $A$ in all expressions, as one fixed choice for $A$ will suffice for the remainder of our proof (cf. Remark \ref{rmk:Fixing-choice-for-A}).	

By Propositions \ref{prop:CorTesEff1} and Corollary \ref{prop:G_A-is-efficient},  it suffices to prove that for every $\alpha>1$ we have $\delta_{\mathcal{G}[\alpha]}(n)\preccurlyeq n^{p-1}$, where we recall that by definition the set $\mathcal{G}[\alpha]$ (resp. $\mathcal{F}[\alpha]$) is the set consisting of all words obtained by concatenating at most $\alpha$ words from the set $\mathcal{G}$ (resp. from $\mathcal{F}$).

We now describe the structure of the proof via a list of technical lemmas. In what follows, saying that an identity between words in $T_A$ holds in $G_{p,p-1}$ will be shorthand for saying that it holds in $\mathcal{P}(G_{p,p-1}).$
 It is easy to deduce from its presentation that $G_{p,p-1}$ is a metabelian group. The first important step is to prove that the commutation relations in $G_{p,p-1}$, induced by its metabelian structure, have area $\lesssim_p n^{p-1}$ and diameter $\lesssim_p n$.  More generally we prove

\begin{lemma}[{Main commuting Lemma}]
\label{lem:MainLemma}
Let $p\geqslant 5$, $\alpha, n\geqslant 1$. Let $w_1,w_2$ be either powers of $x_2$ or words in $\mathcal{F}\left[\alpha\right]$ representing elements of the derived subgroup, such that $\ell(w_1),\ell(w_2)\leqslant n$.  Then the identity $\left[w_1,w_2\right]\equiv 1$ holds in $G_{p,p-1}$ with area $\lesssim_{\alpha,p} n^{p-1}$ and diameter $\lesssim_{\alpha,p} n$. 
\end{lemma}
This result will be the consequence of four more specific lemmas. Before stating them we shall recall and introduce some additional notation.

We will denote by $\nnn=(n_1,\dots,n_k)\in \RR^k$ a $k$-tuple of real numbers and $|\nnn|:=\sum_{i=1}^k |n_i|$ its $\ell^1$-norm. 
As before, for $p-1\geqslant k\geqslant 2$ and $p \geqslant j \geqslant 2$, we denote
\[
 \Omega_{k}^j(\nnn):= \left[x_1^{n_1},\dots,x_1^{n_{k-1}},x_j^{n_k}\right],
\]
and
\[
 \widetilde{\Omega}^j_{k}(\nnn):=\left[y_1^{n_1},\dots,y_1^{n_{k-1}},y_j^{n_k}\right],
\]
while for $k=1$ we define $\Omega_{k}^j(\nnn):= x_j^{n_1}$, $\widetilde{\Omega}_k^j(\nnn):=y_j^{n_1}$.
To simplify notation, when $j=2$, we shall simply write $ \Omega_{k}(\nnn)$ and $\widetilde{\Omega}_{k}(\nnn).$

We record the following key observation.
\begin{lemma}[{Substitution Lemma}]\label{lem:ChangingFactors} 
Let $p\geqslant 5$. For $\nnn\in \RR^{p-2}$ the word $\Omega_{p-2}^3(\nnn)$ is central in $G_{p,p-1}$. In particular, the identity $$\Omega_{p-2}^3(\nnn)\equiv \widetilde{\Omega}^3_{p-2}(\nnn)$$ holds in $G_{p,p-1}$ with area $\lesssim_{p} |\nnn|^{p-2}$ and diameter $\lesssim_p |\nnn|$.
\end{lemma}
\begin{proof}
 This is a direct consequence of the fact that there is a canonical embedding of presentations $\mathcal{P}(G_{p-1,p-1})\hookrightarrow \mathcal{P}(G_{p,p-1})$ such that the word $\Omega_{p-2}^3(\nnn)\cdot \left(\widetilde{\Omega}_{p-2}^3(\nnn)\right)^{-1}$ is contained in the image of $\mathcal{P}(G_{p-1,p-1})$ and the fact that null-homotopic words of length $n$ in $\mathcal{P}(G_{p-1,p-1})$ admit a filling of area $\lesssim_{p} n^{p-2}$ and diameter $\lesssim_{p} n$.
\end{proof}
Despite being very basic, this result is the fundamental reason for why the Dehn functions of $G_{p,p}$ and $G_{p,p-1}$ are bounded by $n^{p-1}$ rather than $n^p$. Indeed, it allows us to ``push'' words in the first factor which represent central elements into the second factor at a cost that is bounded by the Dehn function of $G_{p-1,p-1}$. Using that the $y_i$ commute with the $x_i$ we can then commute them with words in the $x_i$ at a lower cost than one might a priori expect. We use Lemma \ref{lem:ChangingFactors} at various points and, in particular, in the proof of Lemma \ref{lem:second(k,j)-Lemma} to kick-start our induction step from $p-1$ to $p$.

As mentioned above, the Main commuting Lemma \ref{lem:MainLemma} will result from four sublemmas, dealing with specific commuting relations involving words of type $\Omega_{k}^j$. These lemmas depend on a parameter $k\leqslant p-1$. By $k$-lemma, we mean the statement of the corresponding lemma for a specific value of $k$.

 The first one deals with commutators of words of type $\Omega_{k}^j$ with words representing elements of the derived subgroup.
 \begin{lemma}[{First commuting $k$-Lemma}]
\label{lem:Basic-k-Lemma}
 Let $p\geqslant 5$, $n,\alpha \geqslant 1$, $j\geqslant 3$, $1\leqslant k\leqslant p-2$. Let $w=w(x_1,x_2)$ be a word of length at most $n$ in $\mathcal{F}[\alpha]$ corresponding to an element of $\left[G_{p,p-1},G_{p,p-1}\right]$, and let $\nnn\in \RR^k$ with $|\nnn|\leqslant n$. Then the relation $\left[\Omega_{k}^j(\nnn),w\right]\equiv 1$ holds in $G_{p,p-1}$ with area $\lesssim_{p,\alpha} n^{p-2}$ and diameter $\lesssim_{p,\alpha} n$.
\end{lemma}

Our second lemma treats commutators of words of type $\Omega_{k}^j$ with powers of $x_2$.
\begin{lemma}[{Second commuting $k$-Lemma}]
\label{lem:second(k,j)-Lemma}
 Let $p\geqslant 5$, $n\geqslant 1$, $j\geqslant 3$, $1\leqslant k\leqslant p-2$, $\nnn \in \RR^k$ and $m\in \RR$ with $|\nnn|\leqslant n$. 
 Then the relation $\left[\Omega_{k}^j(\nnn),x_2^m\right]\equiv 1$ holds in $G_{p,p-1}$ with area $\lesssim_{p} |m|\cdot n^{p-3} + n^{p-2}$ and diameter $\lesssim_{p} n+|m|$.
\end{lemma}

The following lemmas are versions of Lemmas \ref{lem:Basic-k-Lemma} and \ref{lem:second(k,j)-Lemma} for $\Omega_{k}$ instead of $\Omega_{k}^j$.

\begin{lemma}[{Third commuting $k$-Lemma}]
\label{lem:Weak-k-Lemma}
 Let $p\geqslant 5$, $n,\alpha  \geqslant 1$ and $2\leqslant k\leqslant p-1$. Let $w=w(x_1,x_2)$ be a word in $\mathcal{F}[\alpha]$ of length at most $n$ corresponding to an element of $\left[G_{p,p-1},G_{p,p-1}\right]$, and let $\nnn\in \RR^k$ with $|\nnn|\leqslant n$. Then the relation $\left[\Omega_{k}(\nnn),w\right]\equiv 1$ holds in $G_{p,p-1}$ with area $\lesssim_{p,\alpha} n^{p-1}$ and diameter $\lesssim_{p,\alpha} n$.
\end{lemma}

\begin{lemma}[{Fourth commuting $k$-Lemma}]
\label{lem:second(k)-Lemma}
	 Let $p\geqslant 5$, $n\geqslant 1$, $1\leqslant k\leqslant p-1$, $\nnn \in \RR^k$ and $m\in \RR$ with $|\nnn|\leqslant n$. Then the relation $\left[\Omega_{k}(\nnn),x_2^m\right]\equiv 1$ holds in $G_{p,p-1}$ with area $\lesssim_{p} |m|\cdot n^{p-2} + n^{p-1}$ and diameter $\lesssim_{p} n+|m|$.
\end{lemma}

To prove the Main commuting Lemma \ref{lem:MainLemma}, we shall need a further reduction step, reducing to words of bounded length in elements of type $\Omega_{k}$.
\begin{lemma}[{Reduction Lemma}]
\label{lem:derived-words-prod-Rpk-Mainthm}
 Let $p\geqslant 5$, $\alpha\geqslant 1$ and let $w=w(x_1,x_2)$ be a word of length at most $n$ in $\mathcal{F}[\alpha]$ corresponding to an element of $\left[G_{p,p-1},G_{p,p-1}\right]$. Then there exists $L=O_{\alpha,p}(1)$ such that the identity
 \[
    w(x_1,x_2)\equiv\prod_{j=1}^L \Omega_{l_j} (\mmm_{j})^{\pm 1},
 \]
 holds in $G_{p,p-1}$ with area $\lesssim_{\alpha,p}  n^{p-1}$ and diameter $\lesssim_{\alpha,p} n$, for some $|\mmm_j|\lesssim_{\alpha,p} n$, and $2\leqslant l_j \leqslant p-1$. 
\end{lemma}

The Main Theorem \ref{thm:Upperbound} will be a consequence of the Reduction Lemma \ref{lem:derived-words-prod-Rpk-Mainthm} and the following more subtle technical result, which deals with products of $\Omega_{k}$-terms with different values of $k$. 

\begin{lemma}[{Cancelling $k$-Lemma}]
\label{lem:Strong-k-Lemma}
 Let $p\geqslant 5$, $n\geqslant 1$, $2\leqslant k\leqslant p-1$ and for all $1\leqslant j\leqslant p-1$, let $M_j$ be a positive integer. Consider a word $w(x_1,x_2)$ of the form 
\[
 w(x_1,x_2)= \left( \prod_{i=1}^{M_k} \Omega_{k}(\nnn_{k,i})^{\pm 1}\right)\left( \prod_{i=1}^{M_{k+1}} \Omega_{k+1}(\nnn_{k+1,i})^{\pm 1}\right)\dots  \left( \prod_{i=1}^{M_{p-1}} \Omega_{p-1}(\nnn_{p-1,i})^{\pm 1}\right),
\]
where  $\nnn_{l,i}\in \RR^j$ satisfies $|\nnn_{l,i}|\leqslant n$.

If $w$ is null-homotopic, then it admits a filling of area $\lesssim_{p,M} n^{p-1}$ and diameter $\lesssim_{p,M} n$ in $G_{p,p-1}$, where $M=\max_j M_j$.
\end{lemma}

Finally we record the following technical result which plays a key role at various stages of the proof.

\begin{lemma}[{Cutting in half $k$-Lemma}]
\label{lem:Cutinhalf}
	For $p\geqslant 4$ consider the group $G_{p,p-1}$. 
Let $k\geqslant 2$ and $\nnn=\left(n_1,\dots,n_k\right)\in \RR^k$. Identities of the form
$
\Omega_{k}(2\nnn) \equiv \Omega_{k}(\nnn)^{2^k} \cdot w_k(\underline{n})
$
and $\Omega_{k}(2\nnn)  \equiv w_k(\underline{n})\cdot \Omega_{k}(\nnn)^{2^k}$
hold in $G_{p,p-1}$, where $w_k=\prod_{i=1}^L \Omega_{l_i}(\mmm_i)^{\pm 1}$ with  $L=O_p(1)$, $l_i\geqslant k+1$ and $|\mmm_i|\lesssim_p |\nnn|$ for $1\leqslant i \leqslant L$. Moreover,  these identities have area $\lesssim_p |\nnn|^{p-1}$ and diameter $\lesssim_p |\nnn|$ in $G_{p,p-1}$.
\end{lemma}

The way the Cutting in half $k$-Lemma is used throughout the proof is a bit subtle: we will require its version for $p-1$ as part of the induction step when proving the commuting $k$-Lemmas for $p$. On the other hand, its $p$-version will be obtained as a corollary of the Main commuting Lemma \ref{lem:MainLemma} that results from the four commuting $k$-lemmas. Finally its $p$-version will be instrumental in the proof of the Cancelling $k$-Lemma \ref{lem:Strong-k-Lemma} for $p$. 

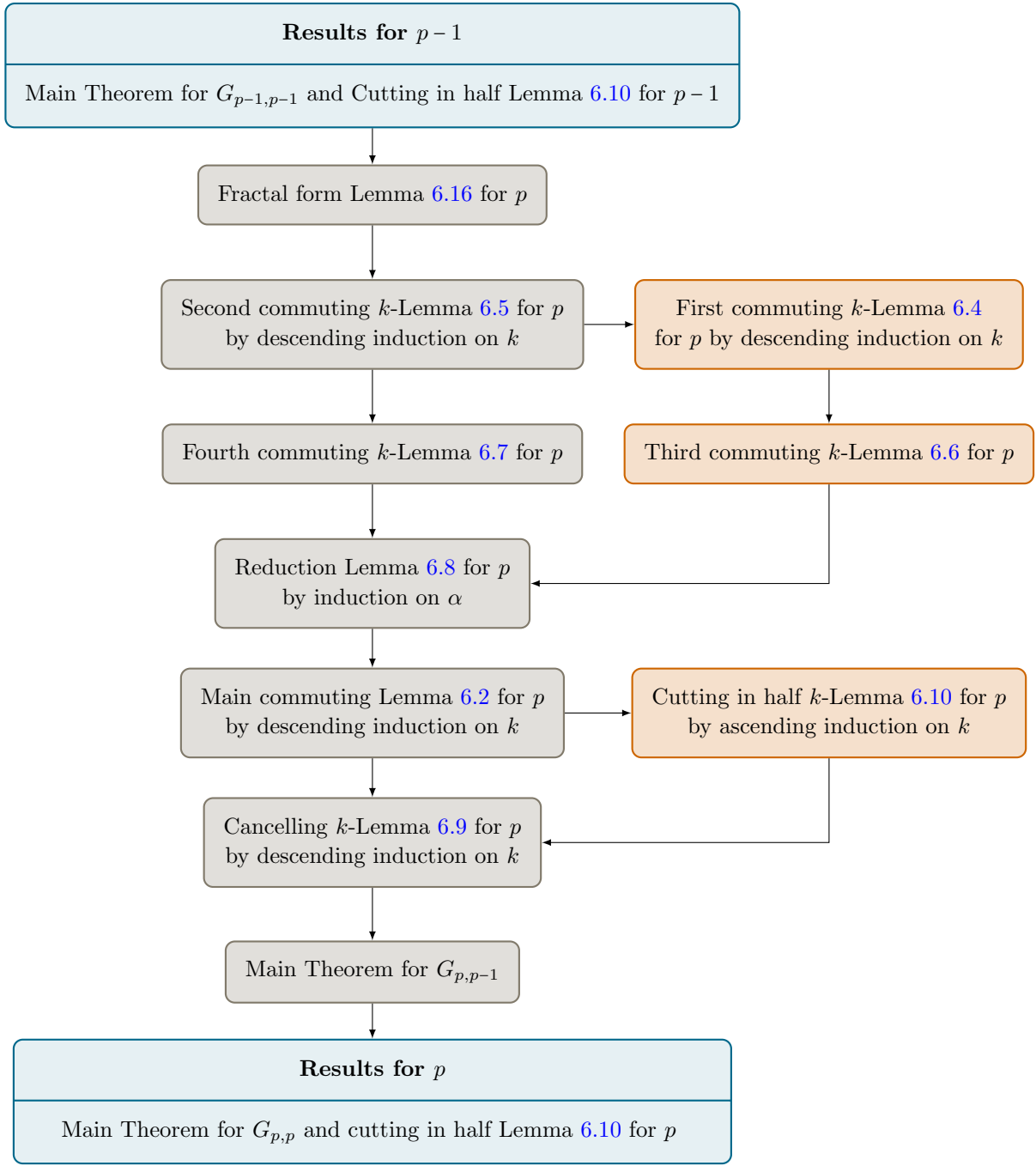
\begin{figure}
\begin{tikzpicture}[
  % Style des nœuds de départ et d’arrivée
  result/.style={rectangle, rounded corners, inner sep=3mm,
    fill=bleufonce!10, draw=bleufonce, thick },
  % Style des nœuds « mis en valeur »
   mev/.style={rectangle, rounded corners, inner sep=3mm,
    fill=orangefonce!20, draw=orangefonce, thick},
  % Style des nœuds normaux
   normaux/.style={rectangle, rounded corners, inner sep=3mm, 
     fill=marronfonce!20, draw=marronfonce, thick},
   % Pour que es nœuds soient reliés par des flèches
   edge from parent/.style = {draw, -latex},
  ]
  % ===============================================================
  % Arborescence
  % ===============================================================
  % Noeud du haut

  \node[result, rectangle split, rectangle split parts=2, minimum width=11cm]{
    \textbf{Results for $p-1$}
    \nodepart{second}
    Main Theorem for $G_{p-1,p-1}$ and
    Cutting in half Lemma \ref{lem:Cutinhalf} for $p-1$
  }
  [level distance=20mm] %règle l’espacement vertical entre les noœuds
  child{
    node[normaux]{Fractal form Lemma {\ref{lem:FractalForm}} for $p$}
    child{
      node[normaux, align=center]{Second commuting $k$-Lemma \ref{lem:second(k,j)-Lemma}
        for $p$ \\ by descending induction on $k$}
      % Le 1e niveau où il y a un embranchement
      % Partie de gauche
      child[grow=right, level distance=70mm]{
        node[mev, align=center]{First commuting $k$-Lemma \ref{lem:Basic-k-Lemma}\\
          for $p$ by descending induction on $k$}
        child[grow=down, level distance=20mm]{
          node[mev, align=center] (third) {Third commuting $k$-Lemma \ref{lem:Weak-k-Lemma} for $p$}
        }
      }
      % Partie principale où la preuve continue
      child[grow=down]{
        node[normaux, align=center] (fourth) {Fourth commuting $k$-Lemma \ref{lem:second(k)-Lemma} for $p$}
        child{
          node[normaux,align=center] (reduction) {Reduction Lemma \ref{lem:derived-words-prod-Rpk-Mainthm} for $p$\\ by induction on $\alpha$}
          child{
            node[normaux,align=center]{Main commuting Lemma \ref{lem:MainLemma} for $p$\\ by descending induction on $k$}
            % Le nœud de droite, mis en valeur
            child[grow=right,level distance=70mm]{
              node[mev, align=center] (cuttinginhalf) {Cutting in half $k$-Lemma  \ref{lem:Cutinhalf} for $p$ \\ by ascending induction on $k$} 
            }
            % Le nœud central, où la preuve continue
            child[grow=down]{
              node[normaux,align=center] (cancelling) {Cancelling $k$-Lemma \ref{lem:Strong-k-Lemma} for $p$ \\ by descending induction on $k$}
              child{
                node[normaux]{Main Theorem for $G_{p,p-1}$}
                child{ %Le dernier nœud : résultat pour p
                  node[result,align=center, rectangle split, rectangle split parts=2, minimum width=11cm]{
                    \textbf{Results for $p$}
                    \nodepart{second}
                    Main Theorem for $G_{p,p}$  and cutting in half Lemma  \ref{lem:Cutinhalf} for $p$
                  }
                }
              }
            }
          }
        }
      }
    }
  };
  % Fleche reliant le cadre orange et le cadre du dessous
  \draw[->, >=latex] (cuttinginhalf.south) |- (cancelling.east);
  %Fleche reliant le cadre « third commuting lemma » au cadre « réduction »
  \draw[->, >=latex] (third.south) |- (reduction.east);
\end{tikzpicture}
\caption{Main steps and structure of the proof of Theorem \ref{thm:Upperbound} (by induction on $p$).}
\label{fig:structure-proof-upper-bound}
\end{figure}

In the proof of the Cutting in half $k$-Lemma we will use the following immediate consequence of the Main commuting Lemma \ref{lem:MainLemma} for $p\geq 5$ (resp. Theorem \ref{thmDehnL4ZH3} for $p=4$). We record it here, as we will require its $(p-1)$-version in the proof of the third and fourth commuting $k$-Lemmas for $p\geq 5$.	

\begin{remark}\label{rem:dealing-with-inverses}	
Let $p\geq 4$ and let $u$ and $v$ be words in $\mathcal{F}[\alpha]$ representing elements of $G_{p,p-1}$ and $\left[G_{p,p-1},G_{p,p-1}\right]$ respectively, with $\ell(u),\ell(v)\leqslant n$. Then the identity	
 \[	
  \left[u,v^{-1}\right] \equiv  \left[u,v\right]^{-1}	
 \]	
 holds with area $\lesssim_{\alpha,p} n^{p-1}$ and diameter $\lesssim_{\alpha,p} n$.	
Indeed, we have the group identity $\left[u,v^{-1}\right] \equiv u^{-1}v u v^{-1} \equiv  [v,u]^{v^{-1}}$.	
We deduce from the fact that $G_{p,p-1}$ is metabelian that $v$ commutes with $[v,u]$. For $p\geq 5$ the Main commuting lemma \ref{lem:MainLemma} for $p$ then implies that the relation	
$[v,u]^{v^{-1}} \left[u,v\right]$ has area $\lesssim_p n^{p-1}$ and diameter $\lesssim_p n$. For $p=4$ the same area and diameter estimates follow from Theorem \ref{thmDehnL4ZH3}.
\end{remark}

Regarding the proof of the diameter bounds we will adopt the following
\begin{convention*}
 Throughout this section the diameter bounds for our fillings will follow from Lemma \ref{lem:diameter}. In most cases this will be obvious, since the transformations used, as well as their prefix words, will satisfy evident linear diameter bounds. To keep the proofs as simple as possible we will only add detailed explanations for the diameter bounds where this is not the case.
\end{convention*}

Throughout the remainder of this section we will assume that by induction $\delta_{G_{p-1,p-1}}(n)\asymp n^{p-2}$ and that every null-homotopic word of length $\leqslant n$ in $\mathcal{P}(G_{p-1,p-1})$ admits a filling of area $\lesssim_{p-1} n^{p-2}$ and diameter $\lesssim_{p-1} n$. %The way the induction procedure works is explained in Figure \ref{fig:structure-proof-upper-bound}.

\begin{initial*}	
As explained in  Figure \ref{fig:structure-proof-upper-bound}, the initial step (for $p=4$) only needs to be settled for the Main Theorem \ref{thm:Upperbound}, and the Cutting in half lemma \ref{lem:Cutinhalf}. The former is provided by Theorem \ref{thmDehnL4ZH3} in the case of $G_{4,3}$. We also observe that the area and diameter estimates of the Cutting in half lemma \ref{lem:Cutinhalf} follow from Theorem \ref{thmDehnL4ZH3}. Hence, in order to initiate the induction, two facts need to be established:	

\begin{enumerate} 	
\item show that the identities of Lemma \ref{lem:Cutinhalf} hold for $k=2,3$  in $G_{4,3}$;	
 \item prove  the Main Theorem \ref{thm:Upperbound} for $G_{4,4}$.	
 \end{enumerate}	
 Let us start by checking (1).	
 For $k=3$ the 3-nilpotency of $G_{4,3}$ implies that	
$\Omega_{3}(2\nnn) = \Omega_{3}(\nnn)^{2^3}$ for all $\nnn=\left(n_1,n_2,n_3\right)\in \RR^3$. The case $k=2$ requires a slightly longer argument.  By Proposition \ref{prop:binom-Lie} and since $z$ is central for $\underline{n}=(n_1,n_2)$ the identities	
\[	
\Omega_2(2 \underline{n}) 	
= x_3^{4n_1n_2}z^{- \binom{2n_1}{2} 2n_2} 	
= (x_3^{n_1n_2}z^{- \binom{n_1}{2} n_2} )^4 z^{4\binom{n_1}{2} n_2} z^{- \binom{2n_1}{2} 2n_2} 	
= (\Omega_2(\underline{n}))^4 z^m	
\]	
hold for some $m\in\RR$ with $|m|\lesssim n^3$. By Lemma \ref{lem:OmegaIdentity}, we have $z^m=\Omega_{3}(m^{1/3},m^{1/3},m^{1/3})$. So writing $n'=m^{1/3}$ we deduce that 	
\[\Omega_{2}(2\nnn) = \Omega_{2}(\nnn)^{2^2}\Omega_{3}(n',n',n')=\Omega_{3}(n',n',n')\Omega_{2}(\nnn)^{2^2},\]	
where $|n'|\lesssim n$, so we are done.	

We now turn to the proof of the Main Theorem \ref{thm:Upperbound} for $G_{4,4}$. It is a direct consequence of Main Theorem \ref{thm:Upperbound} for $G_{4,3}$ and Lemma \ref{lem:OneFactor}. However some explanation is required as the proof of Lemma \ref{lem:OneFactor} itself relies on two statements: Lemma  \ref{lem:ChangingFactors} and Corollary \ref{cor:Usingx3stoReplaceRpp-1s}. Lemma \ref{lem:ChangingFactors} has a short self-contained proof which has already been given. Corollary \ref{cor:Usingx3stoReplaceRpp-1s} asserts that the second identity of Lemma \ref{lem:OmegaIdentity} holds in $G_{4,3}$ with area $\lesssim |\nnn|^{3}$ and diameter $\lesssim |\nnn|$, which is a consequences of the Main Theorem \ref{thm:Upperbound} for $G_{4,3}$.	
\end{initial*}	

{\noindent \bf Induction hypothesis:} Throughout the remainder of this section we will now assume that $p\geqslant 5$ and that the Main Theorem \ref{thm:Upperbound} and the Cutting in half lemma \ref{lem:Cutinhalf} hold for $p-1$. In particular, by induction $\delta_{G_{p-1,p-1}}(n)\asymp n^{p-2}$ and every null-homotopic word of length $\leqslant n$ in $\mathcal{P}(G_{p-1,p-1})$ admits a filling of area $\lesssim_{p-1} n^{p-2}$ and diameter $\lesssim_{p-1} n$. The way the induction procedure works is explained in Figure \ref{fig:structure-proof-upper-bound}.	

\subsection{Preliminary results}
We will now record a few simple preliminary results which we will require at different points in the subsequent sections.

\begin{lemma}
\label{lem:Introducing-x3s}
 The following identities hold in $G_{p,p}$ and $G_{p,p-1}$ for all $p\geqslant 3$, $\beta,n,m \in \RR$ and $|\beta|\leqslant 1$:
 \begin{enumerate}
  \item $\left[x_1,x_2^n\right] \equiv x_3^n$ with area $\lesssim_p n^2$ and diameter $\lesssim_p |n|$;
  \item $\left[x_1^m,x_2^n\right] \equiv x_3^n\left[x_3^n,x_1^{m-1}\right]\cdot \left[x_1^{m-1},x_2^n\right]$ and $\left[x_1^m,x_2^n\right]\equiv \left[x_1^m,x_3^n\right]x_3^{-n}\left[x_1^{m+1},x_2^n\right]$  with area $\lesssim_p n^2$ and diameter $\lesssim_p |n|+|m|$;
  \item $\left[x_1^{\beta},x_2^n\right]\equiv x_3^{\beta n}x_4^{t_4}\cdots x_{p-1}^{t_{p-1}}z^{t_p}$ for $|t_i|\lesssim_p n$ with area $\lesssim_p n^2$ and diameter $\lesssim_p |n|$.
 \end{enumerate}
\end{lemma}
\begin{proof}

Identities (1) and (3) are immediate consequences of Proposition \ref{prop:binom-Lie}.
For the first identity in (2) observe that by (1) $x_2^{-n} x_1 \equiv  x_1 x_3^n x_2^{-n}$ with area $\lesssim_p n^2$ and diameter $\lesssim_p |n|$. Thus, we obtain
 \begin{align}
  \left[x_1^m,x_2^n\right]&\equiv  x_1^{-m}x_2^{-n}x_1 x_1^{m-1} x_2^n\\
  &\equiv x_1^{-(m-1)}x_3^nx_2^{-n}x_1^{m-1}x_2^n\\
  &\equiv x_3^n\left[x_3^n,x_1^{m-1}\right]\left[x_1^{m-1},x_2^n\right].
 \end{align}
 The second identity follows from the first one by replacing $m$ by $m+1$ and rearranging the terms.
\end{proof}

We will also require the following:
\begin{lemma}
\label{lem:extracting-exponents-of-xi}
For $n\geqslant 1$, $k\geqslant 3$, and $w=x_k^{t_k}\cdot \dots  x_{p-1}^{t_{p-1}}z^{t_p}$ with $|t_i|\leqslant n^{i-1}$ there are $\nnn_i\in \RR^i$, $k-1\leqslant i \leqslant p-1$, with $|\nnn_i|\lesssim_p n$  such that the identity
 \[
  w\equiv \prod_{i=k-1}^{p-1} \Omega_{i}(\nnn_i)
 \]
 holds in $G_{p,p-1}$ (and in $G_{p,p}$). 
\end{lemma}
\begin{proof}
This is a direct consequence of Lemma \ref{lem:sec-efficient-tech-1} and its proof.
\end{proof}

As a consequence of Lemma \ref{lem:extracting-exponents-of-xi} and the induction hypothesis for $p-1$ we obtain:
\begin{lemma}
\label{lem:reducing-products-of-Rpp-2s}
  Let $n,I \geqslant 1$. If $\mmm_1,\dots, \mmm_k\in \RR^{p-2}$, with $|\mmm_i|\leqslant n$, satisfy the identity
 \begin{equation}
 \label{eqn:prodRpp-2s-null-hom}
  \prod_{i=1}^I \Omega_{p-2}^3(\mmm_i)^{\epsilon_i} \equiv 1
 \end{equation}
 in $G_{p-1,p-1}$  (and thus in $G_{p,p-1}$) for $\epsilon_i\in \left\{\pm 1\right\}$, then the corresponding relation admits a filling of area $\lesssim_p I\cdot n^{p-2}$ and diameter $\lesssim_{p} n+\sum_{m=3}^p\left(I\cdot n^{m-2}\right)^{\frac{1}{m-1}}$. In particular, if $I\leqslant n$ then the area is $\lesssim_p n^{p-1}$ and the filling diameter is $\lesssim_p n$.
\end{lemma}
\begin{proof}
By definition the $\Omega_{p-2}^3(\mmm_i)^{\epsilon_i}$ are central in $G_{p-1,p-1}$. Thus there are $q_i\in \RR$ with $\Omega_{p-2}^3(\mmm_i)^{\epsilon_i}\equiv  z^{q_i}$. Since the distortion of $\left\langle z\right\rangle \leqslant G_{p-1,p-1}$ is $\simeq n^{\frac{1}{p-2}}$ we deduce that $|q_i|\lesssim_p n^{p-2}$. Since the right hand side of \eqref{eqn:prodRpp-2s-null-hom} is trivial we must have $\sum_{i=1}^I q_i = 0$. In particular, there is $i_0$ such that $|q_{i_0}+q_{i_0+1}|\leqslant \mathrm{max}\left\{|q_{i_0}|,|q_{i_0+1}|\right\}$. Thus, Lemma \ref{lem:extracting-exponents-of-xi} implies that there is $|\mmm'|\lesssim_p |q_{i_0}+q_{i_0+1}|^{\frac{1}{p-2}}\lesssim_p n$ such that 
 \[
   \Omega_{p-2}^3(\mmm_{i_0})^{\epsilon_{i_0}} \Omega_{p-2}^3(\mmm_{i_0+1})^{\epsilon_{i_0+1}} \equiv  \Omega_{p-2}^3(\mmm')
 \]
 in $G_{p-1,p-1}$.
Since this is an identity of length $\lesssim_p n$ in $G_{p-1,p-1}$ it has area $\lesssim_p n^{p-2}$ and diameter $\lesssim_p n$. We can thus reduce to a null-homotopic product
 \[
  \left(\prod_{i=1}^{i_0-1} \Omega_{p-2}^3(\mmm_i)^{\epsilon_i}\right) \cdot \Omega_{p-2}^3(\mmm')\cdot \left(\prod_{i=i_0+2}^{I} \Omega_{p-2}^3(\mmm_i)^{\epsilon_i}\right)
 \]
 of $I-1$ terms such that every factor is of length $\lesssim_p n$ and equal to $z^{r}$ with $|r|\lesssim_p n^{p-2}$. Repeating this argument a further $I-1$ times shows that our initial word can be reduced to the trivial word at cost $\lesssim_p I \cdot n^{p-2}$. Noting that by Corollary \ref{cor:distance-estimates} all prefix words of our transformations satisfy the asserted diameter bound of $n+\sum_{m=3}^p\left(I\cdot n^{m-2}\right)^{\frac{1}{m-1}}$ completes the proof.
\end{proof}

We finish with two more technical results which we will require later.
\begin{lemma}
\label{lem:prodformLiegroupx3}
 Let $n\geqslant 1$, $k\geqslant 2$, and let $|n_i|,|n_{k,j}|\leqslant n$ for $1\leqslant i \leqslant k-1$ and $3\leqslant j\leqslant p$. Denote $u=x_3^{n_{k,3}}\dots x_{p-1}^{n_{k,p-1}} z^{n_{k,p}}$.  The identity 
 \[
  \left[x_1^{n_1},\dots,x_1^{n_{k-1}},u\right]\equiv  \prod_{j=3}^{p-1} \left[x_1^{n_1},\dots,x_1^{n_{k-1}},x_j^{n_{k,j}}\right],
 \]
 holds in $G_{p-1,p-1}$ (and thus in $G_{p,p-1}$) with area $\lesssim_p n^{p-2}$ and diameter $\lesssim_p n$.
\end{lemma}
\begin{proof}
 It follows readily from Lemma \ref{lem:FreeIdentities} that this identity holds in $G_{p-1,p-1}$ and we obtain the area and diameter estimates using the induction hypothesis for $G_{p-1,p-1}$.
\end{proof}
This result will allow us to commute  elements of the form $$ \left[x_1^{n_1},\dots,x_1^{n_{k-1}},x_3^{n_{k,3}}\dots x_{p-1}^{n_{k,p-1}} z^{n_{k,p}}\right]$$ with words $w(x_1,x_2)$ in the derived subgroup of $G_{p,p-1}$ using the First commuting $k$-Lemma \ref{lem:Basic-k-Lemma} (see \S \ref{sec:PfBasickLemma}).
We end this section with the following converse of Lemma \ref{lem:extracting-exponents-of-xi}.
\begin{lemma}
 \label{lem:expansion-of-Rpk}
Let $k\geqslant 1$. For every $\nnn\in \RR^k$ there are $t_i\in \RR$ with $|t_i|\lesssim_p |\nnn|^{i-1}$, $k\leqslant i \leqslant p$, which satisfy the following identity in $G_{p,p-1}$
 \[
  \Omega_{k+1}(\nnn) \equiv x_{k+1}^{t_{k+1}}x_{k+2}^{t_{k+2}}\cdots x_{p-1}^{t_{p-1}} z^{t_p}.
 \]
\end{lemma}
\begin{proof}
 This is an immediate consequence of Lemma \ref{lem:sec-efficient-tech-2}.
\end{proof}

\subsection{First and Second commuting $k$-Lemmas}\label{sec:PfBasickLemma}
For simplicity of notation, we will assume that $j=3$. The proof for $j>3$ is the same. Recall that we are  proceeding by induction on $p$ as shown in Figure \ref{fig:structure-proof-upper-bound}: i.e.\  the Main commuting Lemma \ref{lem:MainLemma} (which is a special case of the Main Theorem \ref{thm:Upperbound}) and the Cutting in half Lemma \ref{lem:Cutinhalf} can be used in the group $G_{p-1,p-1}$ with area $\lesssim_{p} n^{p-2}$ and diameter $\lesssim_p n$. Note that formally the Cutting in half Lemma \ref{lem:Cutinhalf} is stated in the subgroup $G_{p-1,p-2}\leqslant G_{p-1,p-1}$. However, the natural inclusion of the corresponding presentations means that it also holds in $G_{p-1,p-1}$.

A crucial step in the proof of the First commuting Lemma \ref{lem:Basic-k-Lemma} will be the following technical result, allowing us to cut $\Omega_k^3(\nnn)$ into pieces. We notice that $\Omega_k^3(\nnn)$  is a word in $x_1$ and  $x_3$ which therefore belongs to  $G_{p-1,p-1}.$
\begin{lemma}[{Fractal form Lemma}]
\label{lem:FractalForm}
Let $1\leqslant k \leqslant p-3$ and let $\underline{n}\in \RR^k$. In $G_{p,p-1}$, $\Omega_k^3(\nnn)$ is equal to a word $w$ consisting of $\lesssim_{p} |\nnn|^k$ copies of $\Omega_k^3\left(\frac{\nnn}{2^{\left\lceil\log_2(|\nnn|)\right\rceil}}\right)$ and $2^{(j-1)k}$ ``error terms" $w_{k,j}$ for $1\leqslant j \leqslant \left\lceil\log_2(|\nnn|)\right\rceil$. Each $w_{k,j}$ is a product of  $O_p(1)$ commutators of the form $\Omega_l^3(\mmm)^{\pm 1}$ with $|\mmm|\lesssim_p \frac{|\nnn|}{2^j}$ and $k+1\leqslant l \leqslant p-2$. In $G_{p,p-1}$, the area of this identity is $\lesssim_{p} |\nnn|^{p-2}$ and its diameter is $\lesssim_p|\nnn|$. Moreover, the word diameter of $w$ is $\lesssim_p |\nnn|$.
\end{lemma}
\begin{figure}[ht]
 \centering
 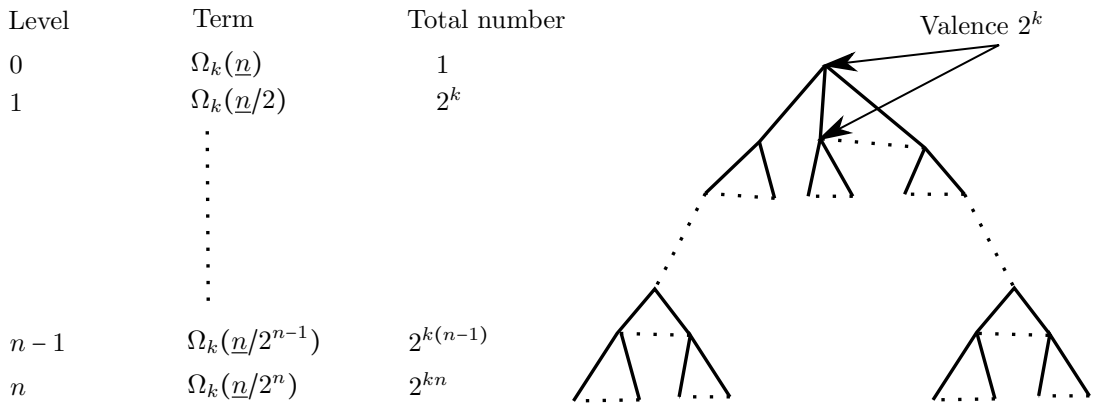
 \caption{Sketch of the $n$ steps leading to the fractal form with $n:=\left\lceil\log_2(|\nnn|)\right\rceil$. It omits the error terms $w_{k,j}$ for simplicity.}
\end{figure}
\begin{proof}
 Note that for $k=1$ this is obvious and the error terms are trivial. The proof for $k\neq 1$ is in $\left\lceil \log_2(\nnn)\right\rceil$ stages. At the $j$-th stage we will be left with $2^{jk}$ terms of the form $\Omega_k^3(\nnn/ 2^{j})$ interlaced with $2^{(i-1)k}$ error terms $w_{k,i}$ for $1\leqslant i \leqslant j$. By Lemmas \ref{lem:distance-estimates} and \ref{lem:Geom-series-diameter} the diameter of this word is
 \begin{equation}\label{eqn:lin-bd-lem-fractal}
   \lesssim_p \sum_{m=k+1}^p \left( 2^{jk} \cdot \left(\frac{|\nnn|}{2^j}\right) ^{m-2} + O_p(1) \cdot \sum_{i=1}^j 2^{(i-1)k} \left(\frac{|\nnn|}{2^i}\right)^{m-2}\right)^{\frac{1}{m-1}} \lesssim_p |\nnn|
 \end{equation}
 for a constant $O_p(1)$ as in Lemma \ref{lem:Cutinhalf}. The same reasoning shows that the word diameter of the word obtained at every stage is $\lesssim_p |\nnn|$. 
 
We apply the Cutting in half Lemma \ref{lem:Cutinhalf} for $p-1$ to each of the words $\Omega_k^3(\nnn/ 2^{j})$ starting with the right-most one; it holds by induction hypothesis. As a consequence we obtain $2^{(j+1)k}$ words of the form $\Omega_k^3(\nnn/ 2^{j+1})$ and $2^{jk}$ error terms of the form $w_{k,j+1}$. By the Cutting in half Lemma \ref{lem:Cutinhalf}, Lemma \ref{lem:diameter} and \eqref{eqn:lin-bd-lem-fractal} the total area and diameter of the identities performed in the $(j+1)$-th iteration are $\lesssim_p \left|\frac{\nnn}{2^j}\right|^{p-2}$ and $\lesssim_p |\nnn|$ respectively. After $\left\lceil\log_2(|\nnn|)\right\rceil$ iterations we obtain the asserted word of word diameter $\lesssim_p |\nnn|$. 

The total area of all identities used in the proof is 
 \begin{align*}
  &\lesssim_p \sum_{j=1}^{\left\lceil \log_2(|\nnn|)\right\rceil} 2^{(j-1)k} \cdot \left(\frac{|\nnn|}{2^{j-1}} \right)^{p-2}\\
  &= |\nnn|^{p-2} \sum_{j=1}^{\left\lceil \log_2(|\nnn|)\right\rceil} 2^{(j-1)(k-(p-2))}\\
  & \lesssim_{p}|\nnn|^{p-2},
 \end{align*}
 where the last inequality follows since the sum is a convergent geometric series. Indeed, by assumption, $k\leqslant p-3$ and thus $k-(p-2)<0$. This completes the proof.
\end{proof}

\begin{proof}[{Proof of the Second commuting $k$-Lemma \ref{lem:second(k,j)-Lemma}}]
Observe that the Second commuting $(p-2)$-Lemma is an easy consequence of Lemma \ref{lem:ChangingFactors} and the fact that $\left[x_i,y_j\right]=1$ $\forall i,j$. We now assume that for $k\leqslant p-3$ we proved the Second commuting $(k+1)$-Lemma by induction. We estimate the area and diameter of the null-homotopic word $\left[ \Omega_k^3(\nnn),x_2^m\right]$. By the Fractal form Lemma \ref{lem:FractalForm} we have 
\begin{equation}
\label{eqn:kprop-1}
\Omega_k^3(\nnn)\equiv u(x_1,x_3), 
\end{equation}
where $u$ is a word that is a product of $\lesssim_p |\nnn|^k$ terms of the form $\Omega_k^3\left(\frac{\nnn}{2^{\left\lceil\log_2(|\nnn|)\right\rceil}}\right)$ and, for $1\leqslant j \leqslant \left\lceil\log_2(|\nnn|)\right\rceil$, $2^{(j-1)k}$ error terms $w_{k,j}$; the terms are in no specific order and we will thus commute them with $x_2^m$ one-by-one. 

Note that, by the Fractal form Lemma \ref{lem:FractalForm}, identity \eqref{eqn:kprop-1} has area $\lesssim_p |\nnn|^{p-2}$ and diameter $\lesssim_p |\nnn|$. Moreover, $u$ has word diameter $\lesssim_p |\nnn|$ and thus the same holds for any of its prefix words. Since all transformations used in the remainder of the proof will consist of commuting a piece of the word $u$ with $x_2^m$ and will have diameter $\lesssim_p |\nnn|+|m|$, the diameter bound of $\lesssim_p |\nnn|+|m|$ in the Second commuting $k$-Lemma will follow from Lemma \ref{lem:diameter}.

Observe that the word $\Omega_k^3\left(\frac{\nnn}{2^{\left\lceil\log_2(|\nnn|)\right\rceil}}\right)$ has length in $O_p(1)$ so that  the area of $\left[\Omega_k^3\left(\frac{\nnn}{2^{\left\lceil\log_2(|\nnn|)\right\rceil}}\right),x_2^t\right]$ for $|t|\leqslant 1$ is in $O_p(1)$. Thus, the total cost of commuting the $|\nnn|^k$ terms of the form $\Omega_k^3\left(\frac{\nnn}{2^{\left\lceil log_2(|\nnn|)\right\rceil }}\right)$ with $x_2^m$ is $\lesssim_p |m| \cdot |\nnn|^k \leqslant  |m| \cdot |\nnn|^{p-3}$, where for the last inequality we use that $k\leqslant p-3$ by assumption.

We now estimate the cost of commuting the error terms $w_{k,j}$ with $x_2^m$. For this we distinguish the cases $k=p-3$ and $k<p-3$, starting with the former. An error term $w_{p-3,j}$ consists of $O_p(1)$ words of the form $\Omega_{p-2}^3\left(\underline{l}\right)^{\pm 1}$, with $|\underline{l}|\lesssim_p \frac{|\underline{n}|}{2^j}$. To move it past $x_2^m$ we use the second factor of our central product: by Lemma \ref{lem:ChangingFactors} the identity $\Omega_{p-2}^3\left(\underline{l}\right)^{\pm 1}\equiv \widetilde{\Omega}_{p-2}^3\left(\underline{l}\right)^{\pm 1}$ holds at cost $\lesssim_p \left|\underline{l}\right|^{p-2}$, with $|\underline{l}|\lesssim_p \frac{|\nnn|}{2^j}$. Since 
$\widetilde{\Omega}_{p-2}^3\left(\underline{l}\right)^{\pm 1}$ is a word in the $y_i$'s, we can commute it with $x_2^m$ at cost $\lesssim_p |m| \cdot \frac{|\nnn|}{2^j}$. Considering that there are $2^{(j-1)(p-3)}$ copies of $w_{p-3,j}$ in $u$ we thus obtain the following upper bound for the total cost of commuting all of the error terms with $x_2^m$:
\begin{align*}
  \mathrm{Area} &\lesssim_p \sum_{j=1}^{\left\lceil\log_2(|\nnn|)\right\rceil} 2^{(j-1)(p-3)}\cdot \left( \left(\frac{|\nnn|}{2^j}\right)^{p-2} + |m| \cdot \frac{|\nnn|}{2^j}\right)\\
  &= 2^{-(p-3)} \cdot \left(|\nnn|^{p-2} \cdot \sum_{j=1}^{\left\lceil\log_2(|\nnn|)\right\rceil} 2^{-j}\right) + 2^{-(p-3)} \cdot\left(|m|\cdot |\nnn|\cdot \sum_{j=1}^{\left\lceil\log_2(|\nnn|)\right\rceil} 2^{(p-4)j}\right)\\
  &\lesssim_p 2^{-(p-3)}\cdot \left(|\nnn|^{p-2}\cdot 2 + |m| \cdot |\nnn|\cdot |\nnn|^{p-4}\cdot 2\right)\lesssim_p |\nnn|^{p-2}+ |m| \cdot |\nnn|^{p-3},
\end{align*}
where to obtain the first inequality in the last line we observe that 
\begin{align*}
 \sum_{j=1}^{\left\lceil\log_2(|\nnn|)\right\rceil} 2^{(p-4)j} &= \sum_{j=0}^{\left\lceil\log_2(|\nnn|)\right\rceil-1} 2^{(p-4)\cdot (\left\lceil\log_2(|\nnn|)\right\rceil-j)}\\
 &\lesssim_p |\nnn|^{p-4} \cdot \sum_{j=0}^{\left\lceil\log_2(|\nnn|)\right\rceil-1}2^{-j(p-4)} \lesssim |\nnn|^{p-4}.
\end{align*}
This completes this step of the proof for $k=p-3$.

To complete the same step of the proof for $k<p-3$ we now assume that by induction the Second commuting $l$-Lemma holds for $p-3\geqslant l\geqslant k+1$. In this case an error term $w_{k,j}$ is equal to a product of $O_p(1)$ words of the form $\Omega_{l}^3\left(\frac{\mmm}{2^j}\right)^{\pm 1}$ with $|\mmm|\lesssim_p |\nnn|$ and $k+1 \leqslant l \leqslant p-2$, and there are $2^{(j-1)k}$ error terms of the form $w_{k,j}$.

By the Second commuting $l$-Lemma for $l\geqslant k+1$ the total cost of commuting the $w_{k,j}$ with $x_2^m$ is thus bounded by
\begin{align*}
\mathrm{Area} & \lesssim_{p}  2^{(j-1)k} \cdot \left( |m| \cdot \left(\frac{|\nnn|}{2^j}\right)^{p-3} +\left(\frac{|\nnn|}{2^j}\right)^{p-2}\right)\\
& = 2^{-k} \cdot \left( |m| \cdot |\nnn|^{p-3} \cdot 2^{j(k-(p-3))} + |\nnn|^{p-2} 2^{j(k-(p-2))}\right).
\end{align*} 
We observe that the assumption $k<p-3$ implies that $j(k-(p-3))< j(k-(p-2))<0$. Using the convergence of the geometric series we hence obtain the following bound on the total cost for commuting the $w_{k,j}$, for $1\leqslant j \leqslant \left\lceil\log_2(|\nnn|)\right\rceil$, with $x_2^m$: 
\begin{align*}
 \mathrm{Area} & \lesssim_{p} 2^{-k} \sum_{j=1}^{\left\lceil\log_2(|\nnn|)\right\rceil} \left( |m| \cdot |\nnn|^{p-3} \cdot 2^{j(k-(p-3))} + |\nnn|^{p-2} 2^{j(k-(p-2))}\right) \\
 &\lesssim \left(|m| \cdot |\nnn|^{p-3} +|\nnn|^{p-2}\right).
\end{align*}

We have thus proved that the cost of commuting all of the $w_{k,j}$ in $u$ with $x_2^m$ is $\lesssim_{p}   \left( m\cdot |\nnn |^{p-3} + |\nnn|^{p-2}\right)$ irrespectively of whether $k=p-3$ or $k<p-3$.

Summing up the total cost for all steps in this proof we obtain that 
\begin{align*}
\mathrm{Area}(\left[ \Omega_k^3(\nnn),x_2^m\right])& \lesssim_p |\nnn|^{p-2} + |m| \cdot |\nnn|^{p-3} + |\nnn|^{p-2} + |m|\cdot |\nnn|^{p-3}\\
& \lesssim_p |\nnn|^{p-2} + |m| \cdot |\nnn|^{p-3}
\end{align*} 
This completes the proof of the Second commuting $k$-Lemma.
\end{proof}

After estimating the cost of commuting $\Omega_k^3(\nnn)$ with $x_2^m$ we now need to estimate the cost of commuting $\Omega_k^3(\nnn)$ with a word in $\mathcal{F}[\alpha]$.
\begin{lemma}
\label{lem:CommutingR3withx1l}
 For $1\leqslant k \leqslant p-2$,   $\nnn=\left(n_1,\dots,n_k\right)\in \RR^k$ and $l\leqslant |\nnn|$ the identity
 \[
  \Omega_k^3(\nnn)^{\pm 1}\cdot x_1^l \equiv x_1^l \cdot \left(\Omega_{k+1}^3(l,n_1,\dots,n_k)\right)^{\mp 1} \cdot \Omega_k^3(\nnn)^{\pm 1}
 \]
 holds with area $\lesssim_p |\nnn|^{p-2}$ and diameter $\lesssim_p |\nnn|$ in $G_{p,p-1}$.
\end{lemma}
\begin{proof}
 The identities
 \begin{align*}
  \Omega_k^3(\nnn)^{\pm 1} \cdot x_1^l &\equiv x_1^l \cdot \Omega_k^3(\nnn)^{\pm 1}\cdot \left[\Omega_k^3(\nnn),x_1^l\right]\\
   & \equiv  x_1^l \Omega_k^3(\nnn)^{\pm 1} \cdot \left( \Omega_{k+1}^3(l,n_1,\dots,n_k)\right)^{\mp 1} \\
   & \equiv  x_1^l \left(\Omega_{k+1}^3(l,n_1,\dots,n_k)\right)^{\mp 1} \cdot \Omega_k^3(\nnn)^{\pm 1}
 \end{align*}
 hold in $G_{p-1,p-1}\leqslant G_{p,p-1}$ and thus with area $\lesssim_p |\nnn|^{p-2}$ and diameter $\lesssim_p |\nnn|$ by induction hypothesis.
\end{proof}

\begin{remark}
 For $k=p-2$ we have $\Omega_{k+1}^3(l,n_1,\dots,n_k)\equiv 1$ in $G_{p-1,p-1}$ with area $\lesssim_p |\nnn|^{p-2}$ and diameter $\lesssim_p |\nnn|$. Thus, Lemma \ref{lem:CommutingR3withx1l} reduces to $\left[\Omega_{p-2}^3(\nnn),x_1^l\right]=1$ in this case.
\end{remark}

\begin{lemma}
\label{lem:basicklemmageneralized}
Let $\alpha \geqslant 2$ and $1\leqslant k \leqslant p-2$. Then for $u=u(x_1,x_2)\in \mathcal{F}[\alpha]$ with $\ell(u)\leqslant n$ an identity of the form
\[
 \Omega_k^3(\nnn)^{\pm 1}\cdot u \equiv  u \cdot \left( \prod_{j=1}^{\nu} \Omega_{l_j}^3(\mmm_j)^{\pm 1}\right) \cdot \Omega_k^3(\nnn)^{\pm 1}
\] 
holds in $G_{p,p-1}$ with $\nu=O_{p,\alpha}(1)$, $l_j\geqslant k+1$, $\mmm_j\in \RR^{l_j}$, and $|\mmm_j|\lesssim_p |\nnn|$. Moreover, this identity has area $\lesssim_{p,\alpha} |\nnn|^{p-2}$ and diameter $\lesssim_{p,\alpha} |\nnn|$ in $G_{p,p-1}$.
\end{lemma}
\begin{proof}
We treat the case $\Omega_k^3(\nnn)^{+1}$, the case $\Omega_k^3(\nnn)^{-1}$ being similar.

 The proof is by descending induction on $k$. The case $k=p-2$ is an easy consequence of the identity $\Omega_{p-2}^3(\nnn)\equiv \widetilde{\Omega}_{p-2}^3(\nnn)$ in $G_{p-1,p-1}$. Thus, assume that $k\leqslant p-3$ and assume that the Lemma holds for $k+1, \dots, p-2$. Since $u\in \mathcal{F}[\alpha]$ we have $$u(x_1,x_2)=x_1^{\beta_1}x_2^{\gamma_1}\dots x_1^{\beta_{\mu}}x_2^{\gamma_{\mu}}$$ for $2\mu \leqslant \alpha$ and $\sum_{i=1}^{\mu}\left( |\beta_i| + |\gamma_i| \right) \leqslant n$.
 
 Applying each, the Second commuting $k$-Lemma \ref{lem:second(k,j)-Lemma} and Lemma \ref{lem:CommutingR3withx1l}, $\mu$ times we obtain that the identity
 \begin{align*}
  \Omega_k^3(\nnn)\cdot u &\equiv \Omega_k^3(\nnn)\cdot \prod_{i=1}^{\mu} x_1^{\beta_i}x_2^{\gamma_i}\\
  & \equiv  \left( \prod_{i=1}^{\mu} x_1^{\beta_i} \left(\Omega_{k+1}^3(\beta_i,\nnn)\right)^{-1} \cdot x_2^{\gamma_i} \right) \cdot \Omega_k^3(\nnn) 
 \end{align*}
 holds with area $\lesssim_p 2 \cdot \alpha \cdot |\nnn|^{p-2}$ and diameter $\lesssim_{p} \alpha \cdot |\nnn|$.
 
 In particular, we have produced $\mu \leqslant \alpha$ words $\Omega_{k+1}^3(\beta_i,\nnn)^{-1}$. Applying the induction hypothesis $\mu$ times (once to each $\Omega_{k+1}^3(\beta_i,\nnn)^{-1}$, starting with the rightmost one), we obtain with area $\lesssim_{p,\alpha} \alpha \cdot |\nnn|^{p-2}$ and diameter $\lesssim_{p,\alpha} \alpha \cdot |\nnn|$ an identity of the form:
 \[
  \Omega_k^3(\nnn)\cdot u \equiv  u\cdot \prod_{i=1}^{\mu} \left(\left(\prod_{j=1}^ {L_j} \Omega_{l_{i,j}}^3(\mmm_{i,j},\nnn)^{\pm 1}\right) \cdot \Omega_{k+1}^3(\beta_i,\nnn)^{-1}\right) \cdot \Omega_k^3(\nnn),
 \]
 where $L_j=O_{k,\alpha}(1)$, $|(\mmm_{i,j},\nnn)|\lesssim |\nnn|$ and $l_{i,j}\geqslant k+2$. Hence we are done.

\end{proof}

The First commuting $k$-Lemma \ref{lem:Basic-k-Lemma} is now a straight-forward consequence.
\begin{proof}[{Proof of the First commuting $k$-Lemma \ref{lem:Basic-k-Lemma}}]
 We apply Lemma \ref{lem:basicklemmageneralized} to $w=w(x_1,x_2)\in \mathcal{F}[\alpha]$ with $w\in \left[G_{p,p-1},G_{p,p-1}\right]$ and $\ell(w)\leqslant n$, observing that under these assumptions the identity $\left[\Omega_k^3(\nnn),w\right] \equiv 1 $ holds in $G_{p,p-1}$. 
 
 It follows that there is $\nu=O_{\alpha,p}(1)$ such that with area $\lesssim_{\alpha,p} |\nnn|^{p-2}$ and diameter $\lesssim_{\alpha,p} |\nnn|$ the identity
 \[
  \Omega_k^3(\nnn)\cdot w \equiv  w \cdot \left( \prod_{i=1}^{\nu} \Omega^3_{l_i}(\mmm_i) ^{\pm 1}\right) \cdot \Omega_k^3(\nnn)
 \] 
 holds with $l_i\geqslant k+1$ and $|\mmm_i|\lesssim_p |\nnn|$ and that, moreover,
 \[
  \prod_{i=1}^{\nu} \Omega^3_{l_i}(\mmm_i,\nnn)^{\pm 1}
 \]
 is null-homotopic in $G_{p-1,p-1}$. However, the latter word has length $\lesssim_{\alpha,p} |\nnn|$. By induction hypothesis for $G_{p-1,p-1}$ we deduce that this null-homotopic word has area $\lesssim_{\alpha,p} |\nnn|^{p-2}$  and diameter $\lesssim_{\alpha,p} |\nnn|$. This completes the proof.
\end{proof}

\subsection{Third and Fourth commuting $k$-Lemmas}\label{sec:weak-k-lemma}

Both lemmas will be easy consequences of the first two commuting $k$-lemmas and the following result:
\begin{proposition}
\label{prop:Decompx2intox3}
  Let $2\leqslant k \leqslant p-1$, $\nnn\in \RR^k$ and $\beta= n_{k-1}- \left\lfloor n_{k-1}\right\rfloor$ Then, if $n_{k-1}\geqslant 0$ the equality
 \begin{equation}\label{eqn:prop:Decompx2intox3-1}
  \Omega_{k}(\nnn)\equiv  \left[x_1^{n_1},\dots,x_1^{\beta},x_2^{n_k}\right] \prod_{j=0}^{\left\lfloor n_{k-1}\right\rfloor-1} \left(\left[x_1^{n_1},\dots,x_1^{n_{k-2}},x_1^{j+\beta},x_3^{n_k}\right]^{-1}\cdot \left[x_1^{n_1},\dots,x_1^{n_{k-2}},x_3^{n_{k}}\right]\right)  
 \end{equation}
 holds  in $G_{p,p-1}$ at cost $\lesssim_p |\nnn|^{p-1}$ and with diameter $\lesssim_p |\nnn|$ and if $n_{k-1}<0$ the equality
  \begin{equation}
  \label{eqn:prop:Decompx2intox3-2}
  \Omega_{k}(\nnn)\equiv  \left[x_1^{n_1},\dots,x_1^{\beta},x_2^{n_k}\right] \prod_{j=1}^{-\left\lfloor n_{k-1}\right\rfloor} \left(\left[x_1^{n_1},\dots,x_1^{n_{k-2}},x_1^{\beta-j},x_3^{n_k}\right]\cdot \left[x_1^{n_1},\dots,x_1^{n_{k-2}},x_3^{n_{k}}\right]^{-1}\right)
  \end{equation}
 holds  in $G_{p,p-1}$ at cost $\lesssim_p |\nnn|^{p-1}$ and with diameter $\lesssim_p |\nnn|$.
\end{proposition}

\begin{addendum}
\label{add:prop:Decompx2intox3}
 The words in \eqref{eqn:prop:Decompx2intox3-1} and \eqref{eqn:prop:Decompx2intox3-2} have word diameter $\lesssim_p |\nnn|$.
\end{addendum}
\begin{proof}
 This is a direct consequence of the ``in particular'' part of Corollary \ref{cor:distance-estimates}.
\end{proof}

The key step in the proof of Proposition \ref{prop:Decompx2intox3} is summarized by the next result.
\begin{lemma}
\label{lem:Decompx2intox3}
For $p-2\geqslant  k \geqslant 1$, $\beta\in \RR$, $n\geqslant 1$ and words $u=\Omega_{k}^3(\nnn)^{\pm 1}$, $v=\Omega_{k+1}^3(\mmm)^{\pm 1}$ and $w=\Omega_{k+1}(\underline{l})^{\pm 1}$ with $|\beta|, |\nnn|,|\mmm|,|\underline{l}|\leqslant n$, the identity
 \[
  \left[x_1^{\beta},u\cdot v \cdot w\right] \equiv  \left[x_1^{\beta},w\right] \cdot \left[ x_1^{\beta},v\right]\cdot \left[x_1^{\beta},u\right]
 \]
 holds with area $\lesssim_{p} n^{p-2}$ and diameter $\lesssim_{p} n$ in $G_{p,p-1}$.
\end{lemma}
\begin{proof}
 Applying Lemma \ref{lem:FreeIdentities}(2) twice, we deduce the free identities
\begin{align*}
  \left[x_1^{\beta},u\cdot v \cdot w\right] &\equiv  \left[x_1^{\beta},w\right]\cdot \left[x_1^{\beta},u\cdot v \right]^{w}\\
  &\equiv \left[x_1^{\beta},w\right]\cdot \left[x_1^{\beta}, v \right]^{w}\cdot \left(\left[x_1^{\beta}, u \right]^v\right)^{w}.
\end{align*}
Since $v=\Omega_{k+1}^3(\mmm)^{\pm 1}\in \left[G_{p-1,p-1},G_{p-1,p-1}\right]$ and $u\in G_{p-1,p-1}$ it follows from the induction hypothesis for $p-1$ and the assumptions, that the identity $ \left[x_1^{\beta}, u \right]^v \equiv \left[x_1^{\beta}, u \right]$ holds in $G_{p,p-1}$ with area $\lesssim_p n^{p-2}$ and diameter $\lesssim_p n$. Since $w\in \left[G_{p,p-1},G_{p,p-1}\right]\cap \mathcal{F}\left[ \alpha \right]$ for all $\alpha$ sufficiently large, two applications of the First commuting $k$-Lemma \ref{lem:Basic-k-Lemma} imply that the identity
\[
 \left[x_1^{\beta},w\right]\cdot \left[x_1^{\beta}, v \right]^{w}\cdot \left[x_1^{\beta}, u \right]^{w} \equiv  \left[x_1^{\beta},w\right] \cdot \left[ x_1^{\beta},v\right]\cdot \left[x_1^{\beta},u\right]
\]
holds in $G_{p,p-1}$ with area $\lesssim_p n^{p-2}$ and diameter $\lesssim_p n$. This completes the proof. 
\end{proof}

\begin{proof}[{Proof of Proposition \ref{prop:Decompx2intox3}}]
We will assume that $n_{k-1}\geqslant 0$, the proof for $n_{k-1}<0$ being similar. The proof is by induction on $\left \lfloor n_{k-1}\right \rfloor$.  The case $\left \lfloor |n_{k-1}|\right \rfloor=0$ is trivial, so assume that  $\left\lfloor n_{k-1} \right\rfloor>0$.
By Lemma \ref{lem:Introducing-x3s}(2), the identity
\[
\left[x_1^{n_{k-1}},x_2^{n_{k}}\right] \equiv x_3^{n_{k}} \left[x_3^{n_k},x_1^{n_{k-1}-1}\right] \left[x_1^{n_{k-1}-1},x_2^{n_k}\right]\equiv \Omega_{1}^3(n_{k})\cdot \left(\Omega_{2}^3(n_{k-1}-1,n_k)\right)^{-1}\cdot \Omega_{2}(n_{k-1}-1,n_k)
\]
holds in $G_{p,p-1}$ at cost $\lesssim_p |\nnn|^2$ and with diameter $\lesssim_p |\nnn|$. Applying Lemma \ref{lem:Decompx2intox3} and the $(p-1)$-version of Remark \ref{rem:dealing-with-inverses} 
%\footnote{We only use the Remark \ref{rem:dealing-with-inverses} for $G_{p-1,p-1}$ here. Thus we are allowed to use it at this point and the cost of doing so is $\lesssim_p n^{p-2}$ with filling diameter $\lesssim_p n$.}
a total of $k-2\leqslant p-3$ times to $\Omega_{k}(\nnn)$ we obtain the identities
\begin{align*}
\Omega_{k}(\nnn)&\equiv  \left[x_1^{n_1},\dots,x_1^{n_{k-2}},x_1^{n_{k-1}},x_2^{n_k}\right]\\
& \equiv  \Omega_{k}(n_1,\dots,n_{k-2},n_{k-1}-1,n_k) \cdot \Omega_{k}^3(n_1,\dots,n_{k-2},n_{k-1}-1,n_k)^{-1} \cdot \Omega_{k-1}^3(n_1,\dots,n_{k-2},n_{k})
\end{align*}
in $G_{p,p-1}$ at cost $\lesssim_p  |\nnn|^{p-2}$ and with diameter $\lesssim_p |\nnn|$. Note that a priori the three factors in the last line of the equation may appear in a different order after applying Lemma \ref{lem:Decompx2intox3}. However, since for all $\alpha$ sufficiently large $\Omega_{k}(n_1,\dots,n_{k-2},n_{k-1}-1,n_k)\in \left[G_{p,p-1},G_{p,p-1}\right]\cap \mathcal{F}[\alpha]$, the First commuting $k$-Lemma \ref{lem:Basic-k-Lemma} and the induction hypothesis for $G_{p-1,p-1}$ imply that we can reorder the factors in the given order at cost $\lesssim_p |\nnn|^{p-2}$ and with diameter $\lesssim_p |\nnn|$. 

Applying the induction hypothesis to the word $\Omega_{k}(n_1,\dots,n_{k-2},n_{k-1}-1,n_k)$ concludes the proof (the prefix word being trivial). 
\end{proof}

\begin{proof}[Proof of the Third commuting $k$-Lemma \ref{lem:Weak-k-Lemma}]
Let $w=w(x_1,x_2)\in \left[G_{p,p-1},G_{p,p-1}\right]$ be a word with $w\in \mathcal{F}[\alpha]$ and $\ell(w)\leqslant n$ and let $\nnn\in \RR^k$ with $|\nnn|\leqslant n$. Assume that $n_{k-1}\geqslant 0$, the case $n_{k-1}<0$ being similar. By Proposition \ref{prop:Decompx2intox3} the identity
\begin{equation}
 \label{eqn:decompR2pk}
   \Omega_{k}(\nnn)\equiv  \left[x_1^{n_1},\dots,x_1^{\beta},x_2^{n_k}\right] \prod_{j=0}^{\left\lfloor n_{k-1}\right\rfloor-1} \left(\left[x_1^{n_1},\dots,x_1^{n_{k-2}},x_1^{j+\beta},x_3^{n_k}\right]^{-1}\cdot \left[x_1^{n_1},\dots,x_1^{n_{k-2}},x_3^{n_{k}}\right]\right)  
\end{equation}
holds in $G_{p,p-1}$ with area $\lesssim_p |\nnn|^{p-1}$ and diameter $\lesssim_p |\nnn|$, where $|\beta|\leqslant 1$. Applying the First commuting $k$-Lemma \ref{lem:Basic-k-Lemma} at most $ 2n$ times to commute the terms on the right side of the identity \eqref{eqn:decompR2pk} with $w$ thus yields
\[
  \Omega_{k}(\nnn)\cdot w \equiv   \left[x_1^{n_1},\dots,x_1^{\beta},x_2^{n_k}\right] \cdot w \cdot  \prod_{j=0}^{\left\lfloor n_{k-1}\right\rfloor-1} \left(\left[x_1^{n_1},\dots,x_1^{n_{k-2}},x_1^{j+\beta},x_3^{n_k}\right]^{-1}\cdot \left[x_1^{n_1},\dots,x_1^{n_{k-2}},x_3^{n_{k}}\right]\right)  
\]
in $G_{p,p-1}$ with area $\lesssim_p n  |\nnn|^{p-2}+ |\nnn|^{p-1}$ and diameter $\lesssim_p |\nnn|$; for the diameter estimate we use Addendum \ref{add:prop:Decompx2intox3}. 

Lemma \ref{lem:Introducing-x3s}(3) and Lemma \ref{lem:prodformLiegroupx3} imply that there are $t_3=\beta n_{k}$ and $t_i$ with $|t_i|\lesssim_p n$, $4\leqslant i \leqslant p$ such that the identities
\begin{align*}
 \left[x_1^{n_1},\dots,x_1^{\beta},x_2^{n_k}\right] & \equiv  \left[x_1^{n_1},\dots , x_1^{n_{k-2}},x_3^{t_3}\cdots x_{p-1}^{t_{p-1}}z^{t_p}\right]\\
 & \equiv  \prod _{j=3}^{p-1} \left[x_1^{n_1},\dots , x_1^{n_{k-2}},x_j^{t_j}\right] 
\end{align*}
hold in $G_{p,p-1}$ with area $\lesssim_p |\nnn|^{p-2}$ and diameter $\lesssim_p |\nnn|$. Applying the First commuting $k$-Lemma \ref{lem:Basic-k-Lemma} $p-3$ times yields that 
\begin{align*}
 \Omega_{k}(\nnn)\cdot w\equiv & w\cdot  \prod _{j=3}^{p-1} \left(\left[x_1^{n_1},\dots , x_1^{n_{k-2}},x_j^{t_j}\right]\right)  &\cdot  \prod_{j=0}^{\left\lfloor n_{k-1}\right\rfloor-1} \left(\left[x_1^{n_1},\dots,x_1^{n_{k-2}},x_3^{n_k},x_1^{j+\beta}\right]\cdot \left[x_1^{n_1},\dots,x_1^{n_{k-2}},x_3^{n_{k}}\right]\right)  
\end{align*}
in $G_{p,p-1}$ with area $\lesssim_p n |\nnn|^{p-2}+ |\nnn|^{p-1}$ and diameter $\lesssim_p |\nnn|$. Finally, a further application of Lemma \ref{lem:prodformLiegroupx3}, Lemma \ref{lem:Introducing-x3s}(3) and Proposition \ref{prop:Decompx2intox3} to the right hand side yields that
\[
 \Omega_{k}(\nnn) \cdot w \equiv  w\cdot \Omega_{k}(\nnn)
\]
holds in $G_{p,p-1}$ with area $\lesssim_p n|\nnn|^{p-2} +  |\nnn|^{p-1}$ and diameter $\lesssim_p |\nnn|$. This completes the proof of the Third commuting $k$-Lemma.
\end{proof}

\begin{proof}[Proof of the Fourth commuting $k$-Lemma \ref{lem:second(k)-Lemma}]
 Note that the same proof demonstrates the Fourth commuting $k$-Lemma, except that in this case the area is $ \lesssim_p |m|\cdot |\nnn|^{p-2} +|\nnn|^{p-1}$ and the diameter is $\lesssim_p |m|+|\nnn|$. Indeed for $k=1$ the result is trivial and for $k\geqslant 2$ we simply replace $w$ by $x_2^m$ everywhere in the above proof and use the Second commuting $k$-Lemma \ref{lem:second(k,j)-Lemma} instead of the First commuting $k$-Lemma \ref{lem:Basic-k-Lemma}.
\end{proof}

We also record the following useful consequence of the arguments presented in this section.

\begin{corollary}
\label{cor:Usingx3stoReplaceRpp-1s}
 For all $\nnn\in\RR^{p-1}$, an identity of the form
 \[
  \Omega_{p-1}(\nnn)\equiv  \left(\Omega_{p-2}^3(|\nnn|,\ldots,|\nnn|)\right)^{m}
 \]
  with $m\in \RR$, $|m|\lesssim_p |\nnn|$, holds in $G_{p,p-1}$ with area $\lesssim_p |\nnn|^{p-1}$ and diameter $\lesssim_p |\nnn|$.
\end{corollary}
\begin{proof}
The identity itself follows from Lemma \ref{lem:OmegaIdentity}. On the other hand, for $k=p-1$, Proposition \ref{prop:Decompx2intox3}, Proposition \ref{prop:binom-Lie} and Lemma \ref{lem:prodformLiegroupx3} yield $\nnn',\nnn''\in \RR^{p-2}$ with $|\nnn'|,|\nnn''|\lesssim_p |\nnn|$ such that the identity	
 \[	
  \Omega_{p-1}(\nnn) \equiv  \Omega_{p-2}^3(\nnn')\cdot \left(\Omega_{p-2}^3(\nnn'')\right)^{\left\lfloor n_{p-2}\right\rfloor}	
 \]	
 holds in $G_{p,p-1}$ with area $\lesssim_p  |\nnn|^{p-1}$ and diameter $\lesssim_p |\nnn|$. Combining these two identities for $\Omega_{p-1}(\nnn)$, we deduce the area and diameter estimates from Lemma \ref{lem:reducing-products-of-Rpp-2s}.
\end{proof}

\subsection{Reduction  and Main commuting Lemmas}

\begin{lemma}
\label{lem:standard-relations-of-length-n-hold}
For $l\geqslant 3$, $\nnn\in \RR^{p}$ and $\mmm\in \RR^{l}$ with $|\nnn|,|\mmm|\leqslant n$ the identities
\begin{enumerate}
 \item $\Omega_{p}(\nnn)\equiv 1$; and
 \item $\left[ x_2^{m_1}, \Omega_{l-1}\left(m_2,\dots, m_l\right)\right]\equiv 1$
\end{enumerate}
hold with area $\lesssim_p n^{p-1}$ and diameter $\lesssim_p n$ in $G_{p,p-1}$.
\end{lemma}
\begin{proof}
Assertion (2) is an immediate consequence of the Fourth commuting $k$-Lemma \ref{lem:second(k)-Lemma}.
We turn to the proof of (1). We focus on the case $n_{p-1}\geqslant 0$, the case $n_{p-1}<0$ being similar. 
	
By Proposition \ref{prop:Decompx2intox3} and Addendum \ref{add:prop:Decompx2intox3} the identities 	
\begin{align*}	
	\Omega_p(\nnn)= &\left[x_1^{n_1},\Omega_{p-1}\left(n_2,\dots,n_p\right)\right]\\	
	 \equiv & \left[x_1^{n_1},\Omega_{p-1}(n_2,\dots,\beta,n_p)\cdot \prod_{j=0}^{\lfloor n_{p-1}\rfloor -1}\left( \Omega_{p-1}^3(n_2,\dots,n_{p-2},j+\beta,n_p)^{-1}\cdot \Omega_{p-1}^3(n_1,\dots,n_{p-2},n_p)\right)\right]	
\end{align*}	
hold at cost $\lesssim_p |\nnn|^{p-1}$ and with diameter $\lesssim_p|\nnn|$. In fact Addendum \ref{add:prop:Decompx2intox3} shows that the last word has word diameter $\lesssim_p|\nnn|$.  We will implicitly use this in all further diameter estimates of this proof.	

Using that $\Omega_{p-1}^3(n_2,\dots,n_{p-2},j+\beta,n_p)\equiv 1$ in $G_{p-1,p-1}$ at cost $\lesssim_p |\nnn|^{p-2}$ and with diameter $\lesssim_p |\nnn|$, and applying Lemma \ref{lem:Introducing-x3s}(3), we obtain that	
\[	
	\Omega_p(\nnn) \equiv \left[x_1^{n_1},\left[x_1^{n_2},\dots,\left[x_1^{n_{p-2}},x_3^{t_3}\cdots z^{t_p}\right]\right]\cdot \prod_{j=0}^{\lfloor n_{p-1}\rfloor -1}\left( \Omega_{p-1}^3(n_1,\dots,n_{p-2},n_p)\right)\right]	
\]	
at cost $\lesssim_p|\nnn|^{p-1}$ and diameter $\lesssim |\nnn|$.	

From Lemma \ref{lem:prodformLiegroupx3} and the fact that $\Omega_{p-2}^j(\cdot ) \equiv 1$ in $G_{p-1,p-1}$ for $j\geq 4$ we can now deduce that	
\begin{align*}	
	\Omega_p(\nnn)\equiv& \left[x_1^{n_1},\left(\prod_{j=3}^{p} \left[x_1^{n_1},\dots \left[x_1^{n_{p-2}},x_j^{t_{j}}\right]\dots\right]\right)\left(\Omega_{p-2}^3\left(n_2,\dots,n_{p-3},n_{p-2},n_p \right) \right)^{\left\lfloor n_{p-1}\right\rfloor} \right]\\	
  \equiv &  \left[x_1^{n_1},\Omega_{p-2}^3(n_1,\dots,n_{p-2},t_3)\cdot \left(\Omega_{p-2}^3 \left(n_2,\dots,n_{p-3},n_{p-2},n_p \right) \right)^{\left\lfloor n_{p-1}\right\rfloor} \right]	
\end{align*}	
with cost $\lesssim_p|\nnn|^{p-1}$ and diameter $\lesssim_p|\nnn|$ (where we use the identification $z=x_p$ to simplify notation in the first line).

 We apply Lemma \ref{lem:ChangingFactors} $\left\lfloor n_{p-1}\right\rfloor +1$ times at cost $\lesssim_p |\nnn|^{p-2}$ and diameter $\lesssim_p |\nnn|$ to obtain
 \begin{align*}
  &\left[x_1^{n_1},\Omega_{p-1}\left(n_2,\dots,n_p\right)\right]\\
  \equiv & \left[x_1^{n_1},\widetilde{\Omega}_{p-2}^3(n_1,\dots,n_{p-2},t_3)\cdot \left(\widetilde{\Omega}_{p-2}^3\left(n_2,\dots,n_{p-3},n_{p-2},n_p \right) \right)^{\left\lfloor n_{p-1}\right\rfloor} \right].
 \end{align*}
 Commuting the $\widetilde{\Omega}_{p-2}^3$ with $x_1^{n_1}$ at cost $\lesssim_p |n_1| \cdot n_{p-1}^2\lesssim_p n^3$ completes the proof, since the total area of all steps is $\lesssim_p |\nnn|^{p-1}$ and the diameter is $\lesssim_p |\nnn|$.
\end{proof}

\begin{lemma}
\label{lem:commutinwithRpks-main-theorem}
  Let $\alpha\geqslant 0$, let $w=w(x_1,x_2)\in \mathcal{F}[\alpha]$ with $\ell(w)\leqslant n$, and let $k\geqslant 1$, $\nnn\in \RR^k$ with $|\nnn|\leqslant n$. Then there exists a positive integer $L=O_{\alpha,p}(1)$ such that the  identity
   \[
    \Omega_{k}(\nnn)^{\pm 1} \cdot w(x_1,x_2)\equiv w(x_1,x_2)\cdot \prod_{j=1}^L \Omega_{l_j} (\mmm_{j})^{\epsilon_j}
 \]
 holds  in $G_{p,p-1}$ with area $\lesssim_{\alpha,p} n^{p-1}$ and diameter $\lesssim_{\alpha,p} n$, for suitable $\epsilon_j\in \left\{\pm 1\right\}$, $|\mmm_j|\lesssim_{\alpha,p}n$, and $k\leqslant l_j \leqslant p-1$. 
\end{lemma}
\begin{proof}
The proof proceeds by induction on $\alpha$. The case $\alpha=0$ is trivially true for $L=1$ and $\Omega_{l_1}(\mmm_1)= \Omega_{k}(\nnn)$. Assume that the result holds for some $2\cdot \alpha \geqslant 0$ and let  
 \[
  w(x_1,x_2)=x_1^{n_1}x_2^{m_1}\cdot \dots \cdot x_1^{n_\kappa} x_2^{m_\kappa}\in \mathcal{F}\left[2(\alpha+1)\right]
 \]
be a word with $\ell(w)\leqslant n$. If $\kappa\leqslant \alpha$ then the result holds by induction hypothesis. We may thus assume $\kappa=\alpha +1$. The following identities hold in $G_{p,p-1}$
 \begin{align*}
  \Omega_{k}(\nnn)^{\pm 1} w(x_1,x_2) & \equiv  \Omega_{k}(\nnn)^{\pm 1} x_1^{n_1}x_2^{m_1}\cdot \dots \cdot x_1^{n_\kappa} x_2^{m_\kappa}\\
  \equiv & x_1^{n_1} \Omega_{k}(\nnn)^{\pm 1} \left[\Omega_{k}(\nnn)^{\pm 1},x_1^{n_1}\right]x_2^{m_1}\cdot \dots \cdot x_1^{n_\kappa} x_2^{m_\kappa}\\
  \equiv & x_1^{n_1} \Omega_{k}(\nnn)^{\pm 1} \left(\left[x_1^{n_1},\Omega_{k}(\nnn)^{\pm 1}\right]\right)^{-1} x_2^{m_1}\cdot \dots \cdot x_1^{n_\kappa} x_2^{m_\kappa}\\
  \overset{(\ast 1)}{=}& x_1^{n_1} \Omega_{k}(\nnn)^{\pm 1} \left(\left[x_1^{n_1},\Omega_{k}(\nnn)\right]\right)^{\mp 1} x_2^{m_1}\cdot \dots \cdot x_1^{n_\kappa} x_2^{m_\kappa}\\
  \equiv & x_1^{n_1} \Omega_{k}(\nnn)^{\pm 1} \left(\Omega_{k+1}(n_1,\nnn)\right)^{\mp 1} x_2^{m_1}\cdot \dots \cdot x_1^{n_\kappa} x_2^{m_\kappa}\\
  \equiv & x_1^{n_1} x_2^{m_1} \Omega_{k}(\nnn)^{\pm 1} \left[\Omega_{k}(\nnn)^{\pm 1},x_2^{m_1}\right] \left(\Omega_{k+1}(n_1,\nnn)\right)^{\mp 1} \left[\left(\Omega_{k+1}(n_1,\nnn)\right)^{\mp 1}, x_2^{m_1}\right]\\ &\cdot x_1^{n_2}x_2^{m_2}\cdot \dots \cdot x_1^{n_\kappa} x_2^{m_\kappa}\\
  \overset{(\ast 2)}{\equiv }& x_1^{n_1} x_2^{m_1} \Omega_{k}(\nnn)^{\pm 1} \left(\Omega_{k+1}(n_1,\nnn)\right)^{\mp 1} x_1^{n_2} x_2^{m_2} \cdot \dots \cdot x_1^{n_\kappa} x_2^{m_\kappa}
 \end{align*}
 where $(\ast 1)$ holds by applying the Third (or Fourth) commuting $k$-Lemma \ref{lem:Weak-k-Lemma} to the right hand side of the identity $\left[x_1^{n_1},\Omega_{k}(\nnn)^{\mp}\right]\equiv \left[\Omega_{k}(\nnn)^{\pm 1},x_1^{n_1}\right]^{\Omega_{k}(\nnn)^{\mp}}$ with area $\lesssim_{p} n^{p-1}$ and diameter $\lesssim_p n$, $(\ast 2)$ holds with area $\lesssim_{p}n^{p-1}$ and diameter $\lesssim_p n$ by Lemma \ref{lem:standard-relations-of-length-n-hold}(2), and the remaining identities are free. Note that if $k=p-1$ then $\Omega_{k+1}(n_1,\nnn)\equiv 1$ with area $\lesssim_{p} n^{p-1}$ and diameter $\lesssim_p n$ by Lemma \ref{lem:standard-relations-of-length-n-hold}(1) and we can thus get rid of it in this case.
 
 Applying the induction hypothesis first to $\left(\Omega_{k+1}(n_1,\nnn)\right)^{\mp 1}$ (if $k\leqslant p-2$) and then to $\Omega_{k}(\nnn)^{\pm 1}$ yields an identity of the form
 \begin{align*}
   x_1^{n_1} x_2^{m_1} \Omega_{k}(\nnn)^{\pm 1} \left(\Omega_{k+1}(n_1,\nnn)\right)^{\mp 1} x_1^{n_2} x_2^{m_2} \cdot \dots \cdot x_1^{n_\kappa} x_2^{m_\kappa} \equiv  w\cdot \prod_{j=1}^{L_1} \Omega_{l_{1,j}} (\mmm_{1,j})^{\epsilon_{1,j}} \cdot \prod_{j=1}^{L_2} \Omega_{l_{2,j}} (\mmm_{2,j})^{\epsilon_{2,j}}
 \end{align*}
 with $L_1,L_2=O_{\alpha,p}(1)$, $k\leqslant l_{1,j},l_{2,j}\leqslant p-1$ and $|\mmm_{1,j}|,|\mmm_{2,j}|\lesssim_{\alpha,p}|\nnn|+|n_1| \lesssim_{\alpha,p} n$ with area $\lesssim_{\alpha,p}n^{p-1}$ and diameter $\lesssim_{\alpha,p} n$. This completes the proof. 
\end{proof}

We now turn to the proof of the Reduction Lemma \ref{lem:derived-words-prod-Rpk-Mainthm}.

\begin{proof}[Proof of the Reduction Lemma \ref{lem:derived-words-prod-Rpk-Mainthm}]
 The proof is by induction on $2\cdot \alpha$. The case $\alpha=1$ is trivial, since $w(x_1,x_2)=x_1^{n_1}x_2^{m_1}\in \left[G_{p,p-1},G_{p,p-1}\right]$ implies that $n_1=m_1=0$. Thus assume that by induction the result holds for some $2\cdot \alpha\geqslant 1$ and consider a word of length at most $n$
 \[
  w(x_1,x_2)=x_1^{n_1}x_2^{m_1}\cdot \dots \cdot x_1^{n_k} x_2^{m_k}\in  \mathcal{F}[2(\alpha+1)]
 \]
corresponding to an element of $\left[G_{p,p-1},G_{p,p-1}\right]$. By induction hypothesis we may assume that $k=\alpha+1$. 
 
The following identities hold:
 \begin{align*}
   w(x_1,x_2) &\equiv  x_1^{n_1}x_2^{m_1}\cdot \dots \cdot x_1^{n_k} x_2^{m_k}\\
   &\equiv  x_1^{n_1+n_2}x_2^{m_1}\left[x_2^{m_1},x_1^{n_2}\right] x_2^{m_2}x_1^{n_3}\cdot \dots \cdot x_1^{n_k}x_2^{m_k}\\
   &\equiv  x_1^{n_1+n_2} x_2^{m_1+m_2} \left[x_2^{m_1},x_1^{n_2}\right] \left[\left[x_2^{m_1},x_1^{n_2}\right],x_2^{m_2}\right]x_1^{n_3}\cdot \dots \cdot x_1^{n_k}x_2^{m_k}\\
   &\equiv  x_1^{n_1+n_2} x_2^{m_1+m_2} \Omega_{2}(n_2,m_1)^{-1}  \left[\left[x_2^{m_1},x_1^{n_2}\right],x_2^{m_2}\right]x_1^{n_3}\cdot \dots \cdot x_1^{n_k}x_2^{m_k}\\
   &\equiv  x_1^{n_1+n_2} x_2^{m_1+m_2} \Omega_{2}(n_2,m_1)^{-1}  x_1^{n_3}\cdot \dots \cdot x_1^{n_k}x_2^{m_k},
 \end{align*}
 where the last identity holds with area $\lesssim_p n^{p-1}$ and diameter $\lesssim_p n$ by Lemma \ref{lem:standard-relations-of-length-n-hold}(2). We apply Lemma \ref{lem:commutinwithRpks-main-theorem} and obtain that
 \[
  x_1^{n_1+n_2} x_2^{m_1+m_2} \Omega_{2}(n_2,m_1)^{-1} x_1^{n_3}\cdot \dots \cdot x_1^{n_k}x_2^{m_k} \equiv x_1^{n_1+n_2} x_2^{m_1+m_2} x_1^{n_3}\cdot \dots \cdot x_1^{n_k}x_2^{m_k} \prod_{j=1}^L \Omega_{l_j} (\mmm_{j})^{\pm 1}
 \]
 with $L=O_{\alpha,p}(1)$, $|\mmm_j|\lesssim_{\alpha,p}  n$, and $2\leqslant l_j\leqslant p-1$, with area $\lesssim_{\alpha,p} n^{p-1}$ and diameter $\lesssim_{\alpha,p} n$. The word $v(x_1,x_2)= x_1^{n_1+n_2} x_2^{m_1+m_2} x_1^{n_3}\cdot \dots \cdot x_1^{n_k}x_2^{m_k}$ has the same exponent sums for $x_1$ and $x_2$ as the word $w$ and  thus also corresponds to an element of $\left[G_{p,p-1},G_{p,p-1}\right]$ of length $\leqslant n$. Moreover, $v\in \mathcal{F}[2\cdot \alpha]$ and hence we can apply the induction hypothesis for $2\cdot \alpha$ to $v$. 
This completes the proof.
\end{proof}

We are now in position to prove the Main commuting Lemma \ref{lem:MainLemma}.
\begin{proof}[Proof of the Main commuting Lemma \ref{lem:MainLemma}]
The case where both $w_1$ and $w_2$ are powers of $x_2$ is obvious.  Else, we may assume that $w_1$ is in the derived subgroup. We then apply Lemma \ref{lem:derived-words-prod-Rpk-Mainthm} to rewrite it as a product of $O_{\alpha,p}(1)$ many terms of type $\Omega_{k}^{\pm 1}$, with $k\geqslant 2$, and we conclude thanks to the Third commuting $k$-Lemma \ref{lem:Weak-k-Lemma} if $w_2$ is in the derived subgroup and the Fourth commuting $k$-Lemma \ref{lem:second(k)-Lemma} if $w_2$ is a power of $x_2$.
\end{proof}

\subsection{Cutting in half Lemma}
The two identities of Lemma \ref{lem:Cutinhalf} are proved in the same way\footnote{Note that we can also deduce one from the other using the Main commuting lemma that has already been established.} so we focus on $
\Omega_{k}(2\nnn) \equiv \Omega_{k}(\nnn)^{2^k} \cdot w_k(\underline{n})
$. The proof is by induction on $k$.

The case $k=1$ is trivial. We thus assume that Lemma \ref{lem:Cutinhalf} holds for some $k\geqslant 1$ and consider the commutator $ \Omega_{k+1}(2\nnn)=\left[x_1^{2n_1},x_1^{2n_2},\dots,x_1^{2n_{k}},x_2^{2n_{k+1}}\right]$. We introduce the notation $v_k=\Omega_{k}(n_2,\ldots,n_{k+1})=\left[x_1^{n_2},\dots,x_1^{n_{k}},x_2^{n_{k+1}}\right]$. By induction hypothesis the identity
\[
\left[x_1^{2n_1},\dots,x_1^{2n_{k}},x_2^{2n_{k+1}}\right]\equiv  \left[x_1^{2n_1},v_k^{2^k}\cdot w_k\right]
\]
holds with area $\lesssim_{p}|\nnn|^{p-1}$ and diameter $\lesssim_p |\nnn|$ for $w_k=\prod_{i=1}^L \Omega_{l_i}(\mmm_i)^{\pm 1}$ with  $L=O_p(1)$, $l_i\geqslant k+1$ and $|\mmm_i|\lesssim_p |\nnn|$ for $1\leqslant i \leqslant L$.

Using Remark ~\ref{rem:dealing-with-inverses}, Lemma ~\ref{lem:FreeIdentities} and the Main commuting Lemma \ref{lem:MainLemma} we observe that the following identities hold with area $\lesssim_p |\nnn|^{p-1}$ and diameter $\lesssim_p |\nnn|$:
\begin{align*}
 \left[x_1^{2n_1},v_k^{2^k}\cdot w_k\right] &\equiv \left[x_1^{2n_1},w_k\right]\cdot \left[x_1^{2n_1},v_k^{2^k}\right]^{w_k}\\
 &\equiv  \left[x_1^{2n_1},w_k\right]\cdot \left[x_1^{2n_1},v_k^{2^k}\right] \cdot \underbrace{\left[\left[x_1^{2n_1},v_k^{2^k}\right],w_k\right]}_{\mbox{$\equiv 1$ by $(\Delta)$}}\\
 &\overset{(\Delta)}{\equiv } \left[x_1^{2n_1},v_k^{2^k}\right]\cdot \left[x_1^{2n_1},w_k\right]\\
 &\equiv \left[x_1^{n_1},v_k^{2^k}\right]^{x_1^{n_1}}\cdot \left[x_1^{n_1},v_k^{2^k}\right]\cdot \left[x_1^{n_1},w_k\right]^{x_1^{n_1}}\cdot \left[x_1^{n_1},w_k\right]\\
 &\overset{(\ast 1)}{\equiv }\left(\left[x_1^{n_1},v_k\right]^{2^k}\right)^{x_1^{n_1}}\cdot \left[x_1^{n_1},v_k\right]^{2^k}\cdot \left[x_1^{n_1},w_k\right]^{x_1^{n_1}}\cdot \left[x_1^{n_1},w_k\right]\\
 &\equiv  \left(\Omega_{k}(\nnn)\left[\Omega_{k}(\nnn),x_1^{n_1}\right]\right)^{2^k}\cdot \Omega_{k}(\nnn)^{2^k}\cdot \left[x_1^{n_1},w_k\right]^{x_1^{n_1}}\cdot\left[x_1^{n_1},w_k\right]\\
 &\overset{(\ast 2)}{\equiv }  \Omega_{k}(\nnn)^{2^{k+1}}\left[\Omega_{k}(\nnn),x_1^{n_1}\right]^{2^k}\cdot \left[x_1^{n_1},w_k\right]^{x_1^{n_1}}\cdot \left[x_1^{n_1},w_k\right]
\end{align*}
Here we wrote $(\Delta)$ whenever we applied the Main commuting Lemma \ref{lem:MainLemma} to words of length $\lesssim_p |\nnn|$. In step $(\ast 1)$ we iteratively applied Lemma \ref{lem:FreeIdentities} and $(\Delta)$ $2\cdot(2^k-1)$ times to words of length $\lesssim_p |\nnn|$ at cost $\lesssim_p 2\cdot(2^k-1)\cdot |\nnn|^{p-1}$. In step $(\ast 2)$ we apply $(\Delta)$ $2\cdot 2^k \cdot 2^k$ times to terms of length $\lesssim_p |\nnn|$, the cost of which is also $\lesssim_p 2^{2k+1} |\nnn|^{p-1}$.

To complete the proof we need to write the error term 
\[
 \left[\Omega_{k+1}(\nnn),x_1^{n_1}\right]^{2^k}\cdot \left[x_1^{n_1},w_k\right]^{x_1^{n_1}}\cdot \left[x_1^{n_1},w_k\right] 
\]
as a product of $O_{p}(1)$ commutators of the form $\Omega_{l'}(\underline{m}')^{\pm 1}$ with $|\mmm'|\lesssim_p |\nnn|$ and $l'\geqslant k+2$ at cost $\lesssim_p |\nnn|^{p-1}$ and with diameter $\lesssim_p |\nnn|$. To see this let $w_k=\prod_{i=1}^L \Omega_{l_i}(\mmm_i)^{\pm 1}\in \gamma_{k+1}(G)$ with  $L=O_p(1)$, $l_i\geqslant k+1$ and $|\mmm_i|\lesssim_p |\nnn|$ for $1\leqslant i \leqslant L$ and consider the following identities:
\begin{align*}
 &\left[\Omega_{k}(\nnn),x_1^{n_1}\right]^{2^k}\cdot \left[x_1^{n_1},w_k\right]^{x_1^{n_1}}\cdot \left[x_1^{n_1},w_k\right]\\
 \overset{(\ast 1)}{\equiv }& \left[x_1^{n_1},\Omega_{k}(\nnn)\right]^{-2^k} \prod_{i=1}^m \left[x_1^{n_1},\Omega_{l_i}(\mmm_i)^{\pm 1}\right]^{x_1^{n_1}} \cdot \prod_{i=1}^m \left[x_1^{n_1},\Omega_{l_i}(\mmm_i)^{\pm 1}\right]\\
 \equiv & \left[x_1^{n_1},\Omega_{k}(\nnn)\right]^{-2^k} \cdot \left(\prod_{i=1}^m \left[x_1^{n_1},\Omega_{l_i}(\mmm_i)^{\pm 1}\right]\cdot \left[\left[x_1^{n_1},\Omega_{l_i}(\mmm_i)^{\pm 1}\right], x_1^{n_1}\right]\right) \cdot  \prod_{i=1}^m \left[x_1^{n_1},\Omega_{l_i}(\mmm_i)^{\pm 1}\right]\\
 \overset{(\ast 2)}{\equiv }& \left[x_1^{n_1},\Omega_{k}(\nnn)\right]^{-2^k} \cdot \left(\prod_{i=1}^m \left[x_1^{n_1},\Omega_{l_i}(\mmm_i)\right]^{\pm 1}\cdot \left[x_1^{n_1}, \left[x_1^{n_1},\Omega_{l_i}(\mmm_i)\right]\right]^{\mp 1}\right) \cdot  \prod_{i=1}^m \left[x_1^{n_1},\Omega_{l_i}(\mmm_i)\right]^{\pm 1}=: w_{k+1}(\nnn).\\
\end{align*}
Observe that in $(\ast 1)$ we apply Lemma \ref{lem:FreeIdentities} and $(\Delta)$ $\leqslant 2O_p(1)$ times and that in $(\ast 2)$ we apply Remark \ref{rem:dealing-with-inverses} $\leqslant 4\cdot O_p(1)$ times to words of length $\lesssim_p |\nnn|$. It follows that these identities hold with area $\lesssim_p |\nnn|^{p-1}$ and diameter $\lesssim_p |\nnn|$. This completes the proof of the Cutting in half $k$-Lemma \ref{lem:Cutinhalf}.

\subsection{Cancelling $k$-Lemma}

The proof of the Cancelling $k$-Lemma is by descending induction on $k$. The Cancelling $(p-1)$-Lemma is a straight-forward consequence of Lemma \ref{lem:reducing-products-of-Rpp-2s}, Corollary \ref{cor:Usingx3stoReplaceRpp-1s} and Lemma \ref{lem:diameter}. Thus assume that the Cancelling $l$-Lemma holds for all $p-1\geqslant l\geqslant k+1$. 

The induction step in the proof of the Cancelling $k$-Lemma is one of the most subtle parts of our proof of Main Theorem \ref{thm:Upperbound}. Our goal is to manoeuvre ourselves into a position where we can use that for a word being null-homotopic implies that in its Malcev normal form in $G_{p,p-1}$ the exponent sum of the $x_k$ must vanish. In particular, this requires extracting the $x_k$ from the word. Pursuing a naive approach using the Fractal form Lemma \ref{lem:FractalForm} will lead to a word that consists of powers of $x_k$ that cumulatively have word length $n^{k-1}$, as well as many ``error terms'' in the form of short iterated commutators that cumulatively have non-linearly bounded word length. Commmuting them using our Main commuting Lemma to assemble the $x_k$ on the left and the error terms on the right would be much too expensive. To circumvent this problem we perform the extraction of powers of $x_k$ using a more intricate procedure which can be seen as beefed-up version of the Fractal form Lemma: rather than producing a word in fractal form we merge error terms whenever we create them and thereby keep their numbers low. We emphasize that it is only at this point of the proof that we can do this, as it will require the $p$-versions of the Cutting in half Lemma \ref{lem:Cutinhalf} and the Main commuting Lemma \ref{lem:MainLemma}. 

We will now perform the core part of the proof of the induction step from $k+1$ to $k$, where we overcome the aforementioned difficulties. This will provide us with the following technical result.
\begin{lemma}
 \label{lem:key-tech-result-strong-k-lemma}
For $p-2\geqslant k \geqslant 2$, $n \geqslant 1$ and $\nnn\in \RR^{k}$ with $|\nnn|\leqslant n$ an identity of the form
 \[
  \Omega_{k}(\nnn)^{\pm 1}\equiv x_{k+1}^{\beta} \cdot E_{p,k}\left(\nnn\right)
 \]
 holds in $G_{p,p-1}$ with area $\lesssim_{p}n^{p-1}$ and diameter $\lesssim_p n$, where $E_{p,k}(\nnn)$ is a product of the form $\prod_{i=k+1}^{p-1} \Omega_{i}(\mmm_i)^{\pm 1}$ with $|\mmm_i|\lesssim_{p} n$. Moreover, $|\beta|\lesssim_p n^k$.\footnote{Note that an iterated application of Proposition \ref{prop:binom-Lie} shows that in fact $\beta=\pm n_1\cdots n_k$ for $\nnn=\left(n_1,\dots,n_k\right)$.}
\end{lemma}
We shall focus here on the identity  $ \Omega_{k}(\nnn)\equiv x_{k+1}^{\beta} \cdot \left(E_{p,k}\left(\nnn\right)\right)$, the other one (with $\Omega_{k}(\nnn)^{-1}$) having the same proof\footnote{The only difference lies in the fact that we would have to use the second identity of the Cutting in half Lemma instead of the first one.}.

\begin{proof}
 The proof is by an inductive procedure in $m:=\left\lceil\log_2(|\nnn|)\right\rceil$. When $m= 1$ the result is an immediate consequence of Lemma \ref{lem:extracting-exponents-of-xi} and our choice of relations, since for $m=1$ we have $|\nnn|\leqslant 1$. The inductive step is encoded in the following claim.
\begin{claim}\label{claim:recur(m)}
There exists a constant $C=C(p)$ such that
if Lemma \ref{lem:key-tech-result-strong-k-lemma} holds for an element $\nnn/2\in \RR^k$ satisfying $\left\lceil\log_2(|\nnn/2|)\right\rceil =m$ with
$\beta=\beta_m$,
cost at most $\delta_m$, and diameter at most $d_m$, then it also holds for the element $\nnn\in \RR^k$ satisfying $\left\lceil\log_2(|\nnn|)\right\rceil=m+1$ with 
\[\beta=\beta_{m+1}=2^k\beta_m,\]
cost at most
\[
\delta_{m+1}\leqslant 2^k\delta_m+C2^{m(p-1)},
\]
and diameter at most
\[d_{m+1}\leqslant d_m+C2^m.\]
\end{claim}
Before proving the claim, let us see why it implies Lemma \ref{lem:key-tech-result-strong-k-lemma}. We immediately deduce that 
$\beta_m\leqslant 2^{km}\beta_1=O_p(|\nnn|^{k})$ and $d_m\leqslant d_1+C\sum_{i=1}^m2^{i-1}=d_1+C2^m=O_p(|\nnn|)$. 
Letting $v_m=2^{-km}\delta_m$, we obtain
\[v_{m+1}\leqslant v_m+C2^{-(m+1)k}2^{m(p-1)} \leqslant v_m+C2^{-mk}2^{m(p-1)} =  v_m + C2^{m(p-1-k)}.\]
Using that $k<p-1$ we deduce that $v_m\leqslant v_1+C\sum_{i=1}^{m-1}2^{i(p-1-k)}=O_p(2^{m(p-1-k)})$, and therefore that $\delta_m=O_p(2^{m(p-1)})=O_p(|\nnn|^{p-1})$. So Lemma \ref{lem:key-tech-result-strong-k-lemma} follows.
\end{proof}

\begin{proof}[Proof of Claim \ref{claim:recur(m)}]
Let $\nnn\in \RR^k$ with $\left\lceil\log_2(|\nnn|)\right\rceil = m+1$. By Lemma \ref{lem:Cutinhalf} for $p$, the identity
 \[
  \Omega_{k}(\nnn)\equiv \left(\Omega_{k}\left(\nnn/2\right)\right)^{2^k} \cdot w_k(\nnn/2)
 \]
 holds with area $\lesssim_{p} |\nnn|^{p-1}$ and diameter $\lesssim_p |\nnn|$, where $w_k(\nnn/2)$ is a product of $O_p(1)$ iterated commutators of the form $\Omega_l(\mmm)^{\pm 1}$ with $|\mmm|\lesssim_p \frac{|\nnn|}{2}$ and $l\geqslant k+1$. 
 
 We apply the induction hypothesis for $m$ to each of the $\Omega_{k}\left(\nnn/2\right)$ successively, starting with the left-most one and moving error terms to the right. After the $i$-th application we obtain an identity of the form
 \[
  \Omega_{k}(\nnn)\equiv x_{k+1}^{i\cdot \beta_{m}} \cdot \left(\Omega_{k}\left(\nnn/2\right)\right)^{2^k-i} \cdot \left(E_{p,k}\left(\nnn/2\right)\right)^{i}\cdot  w_k(\nnn/2).
 \]
Since $|\beta_m|\lesssim_p|\nnn|^k$, Lemma \ref{lem:sec-efficient-tech-1} implies that $x_{k+1}^{i\cdot \beta_m}$ has word diameter $\lesssim_p |\nnn|$. 
 
 An $(i+1)$-th application of the induction hypothesis for $m$ yields
 \[
  \Omega_{k}(\nnn)\equiv x_{k+1}^{i\cdot \beta_m} \cdot x_{k+1}^{\beta_m} \cdot \left(E_{p,k}\left(\nnn/2\right)\right) \cdot \left(\Omega_{k}\left(\nnn/2\right)\right)^{2^k-i-1} \cdot \left(E_{p,k}\left(\nnn/2\right)\right)^{i}\cdot  w_k(\nnn/2)
 \]
 with area $\delta_m$ and diameter $d_m$. 
 
 By applying the Main commuting Lemma \ref{lem:MainLemma} $\leqslant p \cdot 2^k$ times we can commute the terms making up $E_{p,k}\left(\nnn/2\right)$ with the $\left(\Omega_{k}\left(\nnn/2\right)\right)$ and obtain the identity
 \[
  \Omega_{k}(\nnn)\equiv x_{k+1}^{(i+1)\cdot \beta_m} \cdot \left(\Omega_{k}\left(\nnn/2\right)\right)^{2^k-i-1} \cdot \left(E_{p,k}\left(\nnn/2\right)\right)^{i+1}\cdot  w_k(\nnn/2)
 \]
 in $G_{p,p-1}$ with area $\lesssim_p |\nnn|^{p-1}$ and diameter at most $d_m +O_p(|\nnn|)$ (for the latter we use Lemma \ref{lem:diameter} and the fact that $|x_{k+1}^{(i+1)\beta_m}|_{G_{p,p-1}}\lesssim_p |\nnn|$ by the induction hypothesis for $m$ and Lemma \ref{lem:sec-efficient-tech-1}).
 
Putting all of the above steps together, we deduce that the identity
 \begin{equation}\label{eqn:extracting-xk-mod-error}
   \Omega_{k}(\nnn)\equiv x_{k+1}^{2^k\cdot \beta_m} \cdot \left(E_{p,k}\left(\nnn/2\right)\right)^{2^{k}}\cdot  w_k(\nnn/2)
 \end{equation}
 holds in $G_{p,p-1}$ with area 
 \[2^k \cdot \delta_m + O_p(|\nnn|^{p-1}),
 \]
and diameter  at most $d_m +O_p(|\nnn|)$.

 We now apply the  Cancelling $(k+1)$-Lemma  \ref{lem:Strong-k-Lemma} to prove:
 \begin{lemma}\label{lem:errorterms}
  The word $\left(E_{p,k}\left(\nnn/2\right)\right)^{2^{k}}\cdot  w_k(\nnn/2)$ can be tranformed in $G_{p,p-1}$ into an error term of the form $E_{p,k}\left(\nnn\right)$ at cost $\lesssim_{p} |\nnn|^{p-1}$ and with diameter $\lesssim_p |\nnn|$.
 \end{lemma}
 \begin{proof}
We need some preparation that merely involves identities in $G_{p,p-1}$, without considerations of cost.
By Lemma \ref{lem:expansion-of-Rpk} there are $t_i\in \RR$ with $|t_i|\lesssim_{p} |\nnn|^{i-1}$ such that the identity
   \[
    \Omega_{k}(\nnn)\equiv  x_{k+1}^{t_{k+1}}\cdot x_{k+2}^{t_{k+2}} \dots x_{p-1}^{t_{p-1}}z^{t_p}
   \]
holds in $G_{p,p-1}$. Modding out by the $(k+1)$-th term of the central series, we deduce from \eqref{eqn:extracting-xk-mod-error} that
$2^k\beta_m=t_{k+1}$, and so 
   \[
     \left(E_{p,k}\left(\nnn/2\right)\right)^{2^{k}}\cdot  w_k(\nnn/2) \equiv  x_{k+2}^{t_{k+2}} \dots x_{p-1}^{t_{p-1}}z^{t_p}.
   \]
 Finally by Lemma \ref{lem:extracting-exponents-of-xi}, we deduce the following identity
\[ \left(E_{p,k}\left(\nnn/2\right)\right)^{2^{k}}\cdot  w_k(\nnn/2)\equiv \prod_{i=k+1}^{p-1} \Omega_{i}(\mmm_i)^{\pm 1}\]
in $G_{p,p-1}$, with $|\mmm_i|\lesssim_{p} |\nnn|$.
Recall that both $E_{p,k}(\nnn/2)$ and $w_k(\nnn/2)$ are products of $O_p(1)$ many terms of the form $\Omega_l(\underline{m})^{\pm 1}$, with $|\underline{m}|\lesssim_p |\nnn|$, and $l\geqslant k+1$. By the Cancelling $(k+1)$-Lemma, this identity holds with area $\lesssim_{p}|\nnn|^{p-1}$ and diameter $\lesssim_p |\nnn|$.
 \end{proof}
We resume the proof of Claim \ref{claim:recur(m)}. Recall that by definition, we have $|\nnn|\leqslant 2^m$. Choosing $\beta_{m+1}:=2^k\beta_m$, we deduce from (\ref{eqn:extracting-xk-mod-error}) and Lemma \ref{lem:errorterms} that the identity
 \[
   \Omega_{k}(\nnn)^{\pm 1}\equiv  x_{k+1}^{\beta_{m+1}} \cdot E_{p,k}\left(\nnn\right)
 \]
 holds with diameter bounded by $d_{m+1}=d_m+O_p(2^m)$ and area bounded by
 $
  \delta_{m+1}= 2^k \delta_m +  O_p(2^{m(p-1)}),$
thus ending the proof of Claim \ref{claim:recur(m)} (and therefore of Lemma \ref{lem:key-tech-result-strong-k-lemma}). 
\end{proof}

As a consequence of Lemma \ref{lem:key-tech-result-strong-k-lemma} we can complete the proof of the Cancelling $k$-Lemma.
\begin{proof}[{Proof of the Cancelling $k$-Lemma \ref{lem:Strong-k-Lemma} }]

Recall that by induction hypothesis the Cancelling $(k+1)$-Lemma holds. We fix $M\geqslant 1$. Let
\[
 w(x_1,x_2)= \left( \prod_{i=1}^{M_k} \Omega_{k}(\nnn_{k,i})^{\pm 1}\right)\cdot \left( \prod_{i=1}^{M_{k+1}} \Omega_{k+1}(\nnn_{k+1,i})^{\pm 1}\right)\cdot \dots \cdot  \left( \prod_{i=1}^{M_{p-1}} \Omega_{p-1}(\nnn_{p-1,i})^{\pm 1}\right).
\]
be a null-homotopic word in $G_{p,p-1}$ with $\nnn_{j,l}\in \RR^j$, $|\nnn_{j,l}|\leqslant n$, $1\leqslant l\leqslant M_j$,  $k\leqslant l \leqslant p-1$ and $M_j\leqslant M$.

We reduce to the Cancelling $(k+1)$-Lemma by applying Lemma \ref{lem:key-tech-result-strong-k-lemma} iteratively to the terms $\Omega_{k}(\nnn_{k,i})$, $1\leqslant i \leqslant M_k$, starting with the left-most one and then moving the error terms right. At the beginning of the $i_0$-th step of this process we will have an identity of the form
\begin{align*}
 w(x_1,x_2)\equiv &\left(\prod_{i=1}^{i_0-1} x_k^{\beta_{i}}\right) \cdot \left( \prod_{i=i_0}^{M_k} \Omega_{k}(\nnn_{k,i})^{\pm 1}\right)\cdot \left(\prod_{i=1}^{i_0-1} E_{p,k}\left(\nnn_{k,i}\right)
 \right)\\
 &\cdot \left( \prod_{i=1}^{M_{k+1}} \Omega_{k+1}(\nnn_{k+1,i})^{\pm 1}\right)\cdot \dots \cdot  \left( \prod_{i=1}^{M_{p-1}} \Omega_{p-1}(\nnn_{p-1,i})^{\pm 1}\right),
\end{align*}
with $|\beta_i|\lesssim_p n^{k-1}$. In particular, Lemma \ref{lem:sec-efficient-tech-1} implies that all prefix words of transformations will have diameter $\lesssim_p n$.

We apply Lemma \ref{lem:key-tech-result-strong-k-lemma} to obtain 
\begin{align*}
 w(x_1,x_2)\equiv &\left(\prod_{i=1}^{i_0-1} x_k^{\beta_{i}}\right)\cdot x_k^{\beta_{i_0}} \cdot E_{p,k}\left(\nnn_{k,i_0}\right)\cdot  \left( \prod_{i=i_0+1}^{M_k} \Omega_{k}(\nnn_{k,i})^{\pm 1}\right)\cdot \left(\prod_{i=1}^{i_0-1} E_{p,k}\left(\nnn_{k,i}\right) \right)\\
 &\cdot \left( \prod_{i=1}^{M_{k+1}} \Omega_{k+1}(\nnn_{k+1,i})^{\pm 1}\right)\cdot \dots \cdot  \left( \prod_{i=1}^{M_{p-1}} \Omega_{p-1}(\nnn_{p-1,i})^{\pm 1}\right)
\end{align*}
with area  $\lesssim_p  n^{p-1}$ and diameter $\lesssim_p n$.
  
Recall that the $E_{p,k}(\nnn_{k,i})$ are products of $\leqslant p$ terms of the form $\Omega_l(\mmm)^{\pm 1}$ with $|\mmm|\lesssim_{p} |\nnn_{k,i}|\lesssim_{p} n$ and $l\geqslant k+1$. We can thus apply the Main commuting Lemma \ref{lem:MainLemma} a total of $\lesssim_{p} M_k$ times to obtain
\begin{align*}
  w(x_1,x_2)\equiv &\left(\prod_{i=1}^{i_0} x_k^{\beta_{i}}\right) \cdot  \left( \prod_{i=i_0+1}^{M_k} \Omega_{k}(\nnn_{k,i})^{\pm 1}\right)\cdot \left(\prod_{i=1}^{i_0} E_{p,k}\left(\nnn_{k,i}\right) \right)\\
 &\cdot \left( \prod_{i=1}^{M_{k+1}} \Omega_{k+1}(\nnn_{k+1,i})^{\pm 1}\right)\cdot \dots \cdot  \left( \prod_{i=1}^{M_{p-1}} \Omega_{p-1}(\nnn_{p-1,i})^{\pm 1}\right)
\end{align*}
with area $\lesssim_{p}n^{p-1}$ and diameter $\lesssim_p n$.

We obtain the identity
\begin{align*}
 w(x_1,x_2)\equiv &\left(\prod_{i=1}^{M_k} x_k^{\beta_{i}}\right) \cdot \left(\prod_{i=1}^{M_k} E_{p,k}\left(\nnn_{k,i}\right) \right) \cdot \left( \prod_{i=1}^{M_{k+1}} \Omega_{k+1}(\nnn_{k+1,i})^{\pm 1}\right)\cdot \dots \cdot  \left( \prod_{i=1}^{M_{p-1}} \Omega_{p-1}(\nnn_{p-1,i})^{\pm 1}\right)
\end{align*}
in $G_{p,p-1}$ from the original null-homotopic word $w(x_1,x_2)$ with total area $\lesssim_{p} M_k n^{p-1}$ and diameter $\lesssim_p M_k \cdot n$.

By definition of the $E_{p,k}\left(\nnn_{k,i}\right)$, after applying the Main commuting Lemma \ref{lem:MainLemma} at most $p\cdot M_k \cdot M(p-k)$ more times, we obtain
\[
  w(x_1,x_2)\equiv \left(\prod_{i=1}^{M_k} x_k^{\beta_{i}}\right) \cdot \left( \prod_{i=1}^{\widetilde{M}_{k+1}} \Omega_{k+1}(\nnn_{k+1,i})^{\pm 1}\right)\cdot \dots \cdot  \left( \prod_{i=1}^{\widetilde{M}_{p-1}} \Omega_{p-1}(\nnn_{p-1,i})^{\pm 1}\right)
\]
 with area $\lesssim_{p} M n^{p-1}$ and diameter $\lesssim_p M\cdot n$, for suitable $\nnn_{l,i}$, where $\widetilde{M}_l \lesssim_{p} M_l + M_k$ for $l\geqslant k+1$. However, it now follows from the assumption that $w$ is null-homotopic, that $\sum_{i=1}^{M_k} \beta_i = 0$. Thus, we have reduced to the cancellling $(k+1)$-Lemma for some $\widetilde{M} \lesssim_{p} M$ and, by induction hypothesis, the null-homotopic word
\[
 \left( \prod_{i=1}^{\widetilde{M}_{k+1}} \Omega_{k+1}(\nnn_{k+1,i})^{\pm 1}\right)\cdot \dots \cdot  \left( \prod_{i=1}^{\widetilde{M}_{p-1}} \Omega_{p-1}(\nnn_{p-1,i})^{\pm 1}\right)
\]
admits a filling of area $\lesssim_{p,M} n^{p-1}$ and diameter $\lesssim_{p,M} n$. Thus all words $w$ satisfying the hypothesis of the Cancelling $k$-Lemma have area $\lesssim_{p,M} n^{p-1}$ and diameter $\lesssim_p n$. This completes the proof.
\end{proof}

\subsection{Proof of the Main Theorem}\label{subsec:PfMainThm}
We are now ready to complete the proof of the Main Theorem \ref{thm:Upperbound} for $G_{p,p-1}$ and $G_{p,p}$. 

We start by treating the case $p=3$, observing that $G_{3,3}=\HH_5(\mathbf R)$. The fact that $\HH_5(\mathbf R)$ has quadratic Dehn function was originally proved by Allcock using symplectic geometry. His proof is short and elegant and actually proves a stronger statement: any smooth horizontal $L$-Lipschitz map from $S^1$ to $\HH_5(\mathbf R)$ extends to a $O(L)$-Lipschitz map defined on the disc. Here, ``horizontal'' has the following meaning: we consider a ``horizontal'' distribution defined as orthogonal vector complement $\mathfrak m$ of the (one dimensional) derived subalgebra of $\mathfrak h_5(\mathbf R)$, and a path is horizontal if it is tangent to $\mathfrak m$ at every point. 

Allcock's proof can easily be adapted to show that any $L$-Lipschitz {\it piecewise smooth and horizontal} map defined on $S^1$ extends to an $O(L)$-Lipschitz map on the disc. In particular, this applies to ``relation loops", i.e. loops that are obtained by concatenation of paths of the form $\gamma(t)=\gamma(0)u^t$, where $u$ is an element of the generating set
\[T_{1}:= \left\{ x_1^{a_1},x_2^{a_2}, y_1^{a_3},y_3^{a_4}\mid |a_1|,|a_2|,|a_3|,|a_4|\leqslant 1\right\}\]
of $G_{3,3}$ (see \S \ref{subsec:efficient-words}). 
One easily deduces from the Lipschitz filling of such a loop that the corresponding relation admits a Van Kampen diagram of linear diameter and quadratic area. This shows that $G_{3,3}$ admits a $(n^2,n)$
-filling couple. 

\begin{remark}
This also provides a proof of Theorem \ref{thm:5-dim-Heisenberg}: indeed, $\HH_5(\mathbf Z)$ being a uniform lattice in  $\HH_5(\mathbf R)$, the two groups are quasi-isometric, so we can deduce for instance from Lemma \ref{lem:filling-transfer-SBE} (with $e=0$ and $\underline{s}=1$) that $\HH_5(\mathbf Z)$ admits a $(n^2,n)$
-filling couple.  
\end{remark}

We may thus now complete the induction step for $p\geqslant 4$. In particular, we may assume that $G_{p-1,p-1}$ admits $(n^{p-2},n)$ as a filling pair. In the previous sections we have proved that under this assumption all auxiliary results in \S \ref{sec:Intro-proof-main-theorem} hold for $p$ and it remains to put them together. Indeed, as we shall now see, the Main Theorem \ref{thm:Upperbound} for $p$ is a straightforward consequence of the $p$-versions of the Reduction Lemma \ref{lem:derived-words-prod-Rpk-Mainthm} and the Cancelling 2-Lemma \ref{lem:Strong-k-Lemma}.

\begin{proof}[{Proof of the Main Theorem \ref{thm:Upperbound} for $G_{p,p-1}$}]
As explained at the beginning of this section, it suffices to proof that for all $\alpha \geqslant 1$ every null-homotopic word of length $\leqslant n$ in $\mathcal{G}[\alpha]$ admits a filling of area $\lesssim_{\alpha,p} n^{p-1}$ and diameter $\lesssim_{\alpha,p} n$. Let $w=w(x_1,x_2,y_1,y_3)\in \mathcal{G}\left[\alpha\right]$ be a null-homotopic word of length $\ell(w)\leqslant n$. Using that the $x_i$ and $y_i$ commute, there are words $u(x_1,x_2)$ and $v(y_1,y_3)$ such that the identity $w\equiv u\cdot v$ holds in $G_{p,p-1}$ with area $\leqslant n^2$ and diameter $\leqslant n$. The word $v(y_1,y_3)$ represents a central element of length $\leqslant n$ in $G_{p-1,p-1}$. Thus, by induction hypothesis, $u(x_1,x_2)\cdot v(y_1,y_3)\equiv u(x_1,x_2)\cdot v(x_1,x_3)$ in $G_{p,p-1}$ with area $\lesssim_{\alpha,p} n^{p-2}$ and diameter $\lesssim_{\alpha,p} n$. Using again that $v(x_1,x_3)$ represents a central element of length $\leqslant n$, we deduce from Lemma \ref{lem:sec-efficient-tech-2} and Lemma \ref{lem:extracting-exponents-of-xi} that there is $\nnn\in \RR_{p-2}$ with $v(x_1,x_3)\equiv \Omega_{p-2}^3(\nnn)$ and $|\nnn| \lesssim_{\alpha,p} n$, and that this identity holds in $G_{p,p-1}$ with area $\lesssim_{\alpha,p} n^{p-2}$ and diameter $\lesssim_{\alpha,p} n$. Finally, Lemma \ref{lem:Introducing-x3s}(1) implies that the identities
\[
v(x_1,x_3)\equiv \Omega_{p-2}^3(\nnn) \equiv  \Omega_{p-1}(n_1,\dots,n_{p-3},1,n_{p-2})
\]
hold in $G_{p,p-1}$ with area $\lesssim_{\alpha,p} n^{p-2}$ and diameter $\lesssim_{\alpha,p} n$.

Observe that, on enlarging $\alpha$ (twice) if necessary, we may assume that $\Omega_{p-1}(n_1,\dots,n_{p-3},1,n_{p-2})\in \mathcal{F}[\alpha]$ and thus that
\[
 u(x_1,x_2) \cdot \Omega_{p-1}(n_1,\dots,n_{p-3},1,n_{p-2}) \in \mathcal{F}[\alpha].
\]
It follows that we may assume that $w=w(x_1,x_2)$ is a null-homotopic word in $\mathcal{F}\left[\alpha\right]$, at cost $\lesssim_{\alpha,p} n^{p-2}$ and diameter $\lesssim_{\alpha,p} n$.
We apply the Reduction Lemma \ref{lem:derived-words-prod-Rpk-Mainthm} to obtain an identity of the form
   \[
    w(x_1,x_2) \equiv \prod_{j=1}^L \Omega_{l_j} (\mmm_{j})^{\pm 1}
   \]
in $G_{p,p-1}$ with area $\lesssim_{\alpha,p} n^{p-1}$ and diameter $\lesssim_{\alpha,p} n$, and $L=O_{\alpha,p}(1)$. Since $w$ is null-homotopic, the same holds for the right hand side. 

By applying the Main commuting Lemma \ref{lem:MainLemma} at most $L ^2$ times we obtain that in $G_{p,p-1}$
   \begin{equation}\label{eqn:FinalStepMainTheorem}
    w(x_1,x_2) \equiv   \left( \prod_{i=1}^{M_2} \Omega_{2}(\nnn_{2,i})^{\pm 1}\right)\cdot \left( \prod_{i=1}^{M_{3}} \Omega_{3}(\nnn_{3,i})^{\pm 1}\right)\cdot \dots \cdot  \left( \prod_{i=1}^{M_{p-1}} \Omega_{p-1}(\nnn_{p-1,i})^{\pm 1}\right),
   \end{equation}
    with area $\lesssim_{\alpha,p}  n^{p-1}$ and diameter $\lesssim_{\alpha,p} n$, where $M_i\leqslant L$ for $L$ as above. 
    
    The right-hand side of \eqref{eqn:FinalStepMainTheorem} remains null-homotopic. The Cancelling $2$-Lemma \ref{lem:Strong-k-Lemma} thus implies that the right-hand side of \eqref{eqn:FinalStepMainTheorem} has area $\lesssim_{\alpha,p} n^{p-1}$ and diameter $\lesssim_{\alpha,p} n$ in $G_{p,p-1}$. 
    
    Summing up the total area of all tranformations we deduce that $w$ is null-homotopic with area $\lesssim_{\alpha,p} n^{p-1}$ and diameter $\lesssim_{\alpha,p} n$ in $G_{p,p-1}$. In particular, we have proved that every null-homotopic word in $\mathcal{G}[\alpha]$ of length $\leqslant n$ admits a filling of area $\lesssim_{\alpha,p} n^{p-1}$ and filling diameter $\lesssim_{\alpha,p} n$. By Proposition \ref{prop:CorTesEff1}, this implies that $G_{p,p-1}$ admits $(n^{p-1},n)$ as a filling pair. This completes the proof.
 \end{proof}
 
The Main Theorem \ref{thm:Upperbound} for $G_{p,p}$ is a direct consequence of the Main Theorem for $G_{p,p-1}$ and the following result.
\begin{lemma}
 \label{lem:OneFactor}
  Let $v(x_1, x_2, y_1, y_2)$ be a null-homotopic word in $\mathcal{P}(G_{p,p})$ with $\ell(v)\leqslant n$. 
  
  Then there are null-homotopic words $w(x_1, x_2)$ and $w'(y_1, y_2)$ of length $\ell(w),\ell(w')\lesssim_p n$, which satisfy the identity $v(x_1, x_2,y_1,y_2) \equiv  w(x_1, x_2) w'(y_1, y_2)$ in $G_{p,p}$ with area $\lesssim_{p}n^{p-1}$ and diameter $\lesssim_p n$.
\end{lemma}

\begin{proof}
  Using that the $x_i$ commute with the $y_i$, we deduce that there are words $w_1(x_1,x_2)$ and $w_2(y_1,y_2)$ such that the identity $v\equiv w_1 \cdot w_2$ holds with area $\leqslant n^2$ and diameter $\leqslant n$ in $G_{p,p}$. Since $w_1\cdot w_2$ is null-homotopic and the intersection $\left\langle x_1,x_2\right\rangle \cap \left\langle y_1, y_2 \right\rangle$ is equal to the central subgroup $\left\langle z \right \rangle$, we deduce that there is $q\in \RR$ such that $w_1(x_1,x_2)\equiv z^{q}$ and $w_2(y_1,y_2)\equiv z^{-q}$ in $G_{p,p}$. Recall that the distortion of $\left\langle z \right\rangle$ in $G_{p,p}$ is $\simeq n^{\frac{1}{p-1}}$. Since $\ell(w_1)\leqslant n$ it follows that $|q|\lesssim_{p} n^{p-1}$. Thus, by (\ref{eq:Omega_p-1}), there is $\mmm\in\RR^{p-1}$ with $|\mmm|_p\lesssim n$ such that $z^{q} \equiv  \Omega_{p-1}(\mmm)$. In particular, the words $w_1(x_1,x_2)\cdot \left(\Omega_{p-1}(\mmm)\right)^{-1}$ and $\widetilde{\Omega}_{p-1}(\mmm) w_2(y_1,y_2)$ are null-homotopic in $G_{p,p}$.
  
  On the other hand, we deduce from Lemma  \ref{lem:ChangingFactors} and  Corollary \ref{cor:Usingx3stoReplaceRpp-1s} that the identity $\Omega_{p-1}(\mmm) \equiv  \widetilde{\Omega}_{p-1}(\mmm)$ holds in $G_{p,p}$ with area $\lesssim_{p} n^{p-1}$ and diameter $\lesssim_p n$. We deduce that the identity
  \[
   w_1(x_1,x_2) \cdot w_2(y_1,y_2)\equiv  w_1(x_1,x_2)\cdot \Omega_{p-1}(\mmm)^{-1} \widetilde{\Omega}_{p-1}(\mmm) \cdot w_2(y_1,y_2)
  \]
  holds in $G_{p,p}$ with area $\lesssim_{p}  n^{p-1}$ and diameter $\lesssim_p n$. This completes the proof.
\end{proof}

\section{Second cohomology and centralized Dehn functions}
\label{sec:lower-bound-dehn}
The centralized Dehn function of a discrete torsion-free nilpotent group can be computed by computing the maximal distortion of a central extension. In \S \ref{subsec:lower-bound} -- \ref{subsec:tools-for-computingH2r} we will explain how this characterisation of the centralized Dehn function can be rephrased algebraically in terms of the existence of a second real cohomology class with certain properties. We then apply this algebraic characterisation in \S \ref{subsec:central-extensions-central-products} to prove Theorem \ref{propIntro:centralDehn} and, more generally, to analyse the existence of central extensions of central products of nilpotent groups.

\subsection{An algebraic characterization of centralized Dehn functions of nilpotent groups}
\label{subsec:lower-bound}

\begin{definition}
\label{dfn:i-central-extension}
Let $\mathfrak{g}$ be a nilpotent Lie algebra, and let $r \geqslant 1$. 
Then $0 \to \mathbf R \overset{\iota}{\to} \widetilde{\mathfrak{g}} \overset{\pi}{\to} \mathfrak{g}\to 0$ is called a $r$-central extension if 
$\ker (\pi) \subseteq Z(\widetilde{\mathfrak{g}}) \cap \gamma_{r} \widetilde{\mathfrak{g}}$ and $\ker (\pi) \nsubseteq Z(\widetilde{\mathfrak{g}}) \cap \gamma_{r+1} \widetilde{\mathfrak{g}}$.
One similarly defines $r$-central extensions of nilpotent groups (discrete or Lie).
\end{definition}

Being an $r$-central extension only depends on the equivalence class of the extension and we now explain how it can be read off from $H^2(\mathfrak{g}, \mathbf{R})$. 

Let $\mathfrak{g}$ be a real nilpotent Lie algebra with Lie group $G$. Recall that to any $\omega \in Z^2(\mathfrak{g}, \mathbf R)$ one associates a central extension of $\mathfrak{g}$ defined over the vector space $\mathfrak{g} \times \mathbf R$ by
\begin{equation}
\label{eq:definition-of-central-extension-from-cohomology}
\forall X,Y \in \mathfrak g, \forall s, t \in \mathbf R, [(X,s),(Y,t)]_{\widetilde{\mathfrak{g}}} := \left([X,Y]_{\mathfrak{g}}, \omega(X,Y) \right).
\end{equation}

Denote $H^2(\mathfrak{g}, \mathbf R)^r$, resp. $H^2(\mathfrak{g}, \mathbf R)^{\geqslant r}$ the cohomology classes yielding $r$-central extensions, resp. $r'$-central extension for some $r'\geqslant r$ 

Definition \ref{dfn:i-central-extension} is motivated by the following proposition which relates the centralized Dehn function with the existence of $r$-central extensions.
\begin{proposition}[Compare {\cite[Proposition 4]{YoungAverage}}]
\label{prop:Centralrcentralextensions}
Let $\Gamma$ be a torsion-free finitely generated nilpotent group. Let $G$ be its real Malcev completion and $\mathfrak g$ its Lie algebra. Then $\delta^{\mathrm{cent}}_{\Gamma}(n)\asymp n^a$, where $a$ is the maximum integer $r\geqslant 1$ such that one of the following equivalent statement holds 
\begin{itemize}
\item[(i)] $\Gamma$  admits a $r$-distorted central extension;
\item[(i')] $G$  admits a $r$-distorted central extension that is a simply connected Lie group;
\item[(ii)] $\Gamma$ admits an $r$-central extension;	
\item[(ii')] $G$ admits an $r$-central extension that is a simply connected  Lie group;
\item[(ii'')] $\mathfrak g$ admits an $r$-central extension;
\item[(iii)] $H^2(\mathfrak g, \mathbf R)^{\geqslant r}\neq 0$.
\end{itemize}
\end{proposition}
\begin{proof}
We start by proving the equivalences between these statements. Note that in (i') and (ii') we specify that the central extension is a connected Lie group, as ``wild" extensions that don't correspond to extensions of the Lie algebra could potentially exist.

The equivalences between (i) and (ii), resp. (i') and (ii'), are due to Osin's computation of the distortion of subgroups of nilpotent groups \cite{Osindistort}. The equivalence between (ii) and (ii') follows from Malcev's correspondence. The equivalence between (ii') and (ii'') follows from the correspondence between a simply connected nilpotent Lie group and its Lie algebra. Finally, the equivalence with (iii) follows from the correspondence between central extensions and the second real cohomology group.

For the remaining part of the statement first observe that the equivalence in the case when $a<2$ or $r<2$ is easy to check. Indeed, this can only happen if all central extensions are by taking direct products. Hence, we may assume that $a\geqslant 2$ (or conversely that there is an $r$-central extension with $r\geqslant 2$).

Given a finite presentation $\langle S \mid R\rangle$ of $\Gamma$ let $ \widetilde{\Gamma}=F_S/[F_S ,\mathfrak R]$, where $\mathfrak R$ is the normal subgroup spanned by $R$. Consider the central extension
\[1\to Z\to \widetilde{\Gamma}\to \Gamma\to 1,\]
where $Z=\mathfrak R/[F_S ,\mathfrak R]$. It follows that $\widetilde{\Gamma}$ is a finitely generated nilpotent group, and $Z$ is generated by the finite subset $R$ (modulo $[F_S ,\mathfrak R]$). 
Let $n\in \mathbf N$ and $k=\delta^{\mathrm{cent}}_{\Gamma}(n)$. This means that there exists an element $g\in \widetilde{\Gamma}$ whose word length with respect to $S$ is $n$ and such that $k$ is the minimal integer such that $g$ can be written as a word of length $k$ in the generating set $R$ of $Z$. In other words, $\delta^{\mathrm{cent}}_{\Gamma}(n)$ is the distortion of $Z$ in $\widetilde{\Gamma}$.  It is a classical fact that the central extension $\tilde{\Gamma}$ of $\Gamma$ is universal in the sense that for any other central extension $\overline{\Gamma}$ there exists a morphism $\tilde{\Gamma}\to \overline{\Gamma}$ that extends to a morphism of extensions and induces a surjection between the derived subgroups (see for instance \cite[Lemma 5]{YoungAverage} for more details).  Hence $a$ is indeed characterized by one of the equivalent statements of the proposition.
\end{proof}

\subsection{Carnot gradings}
\label{subsec:Carnot-gradings}

We recall from the Introduction that \emph{(i)} a nilpotent Lie group $G$ is said to be \emph{Carnot gradable} if its Lie algebra $\mathfrak g$ admits a Lie algebra grading $\mathfrak{g} =\bigoplus_{i=1}^s \mathfrak{g}_i$ such that $\operatorname{Liespan}(\mathfrak{g}_1)= \mathfrak{g}$, and that \emph{(ii)} to any simply connected nilpotent Lie group $G$ we can associate a Carnot-graded Lie group $\mathsf{gr}(G)$ with Carnot graded Lie algebra 
\begin{equation*}
\label{eq:carnot-associated}
\mathsf{gr}(\mathfrak{g}) = \bigoplus_{i \geqslant 1} \gamma_{i} \mathfrak{g} / \gamma_{i+1} \mathfrak{g}
\end{equation*}
with brackets induced by those on $\mathfrak{g}$. In particular, $H^1(\mathfrak{g}, \mathbf R)$ and $H^1(\mathsf{gr}(\mathfrak{g}), \mathbf R) = (\mathfrak{g} / [\mathfrak{g},\mathfrak{g}])^\star$ are naturally isomorphic.

$G$ is isomorphic to $\mathsf{gr}(G)$ if and only if $G$ is Carnot gradable, in which case the isomorphism is given by the graded linear isomorphism $\Phi_{(G,\mathfrak{g}_1)}: \bigoplus_i \mathfrak{g}_i \to \bigoplus \gamma_{i} \mathfrak{g} / \gamma_{i+1} \mathfrak{g}$ for any Carnot grading $(\mathfrak g_i)$ on the Lie algebra. 

\begin{remark}
Any pair of Carnot gradings $\lbrace (G,\mathfrak{g}_1), (G, \mathfrak{g}'_1) \rbrace$ on a given group $G$ differs by the automorphism $\Phi_{(G,\mathfrak{g'}_1)}^{-1}\circ \Phi_{(G,\mathfrak{g}_1)}$. It induces the identity on $H^1(\mathfrak{g}, \mathbf R)$.
\end{remark}

We refer to \cite[3.2]{CornulierSystolic1} for more on Carnot gradings.

\begin{example}
Let $G_{p,q}$ be the group defined in the introduction, with $p \geqslant q$. 
Denote by $\mathfrak{g}_{p,q}$ its Lie algebra. Then $\mathfrak{g}_{p,q}$ has a basis $\left\{x_1, \ldots, x_{p-1}, z, y_1, \ldots, y_{q-1}\right\}$ with the following nonzero brackets 
\[ 
[x_1, x_i] = x_{i+1} \text{ for } 2 \leqslant i \leqslant p-2, \; [y_1, y_j] = y_{j+1} \text{ for } 2 \leqslant j \leqslant q-2 \text{ and } [x_1, x_{p-1}] = [y_1, y_{q-1}] = z. 
\]
To simplify notation we use the same letters for the elements of the Lie algebra and the Lie group, even though they don't correspond under the exponential map. We emphasize that in this section we will deviate from the remainder of the paper where we denote the generators of the second factor by $y_{p-q+2},\dots, y_{p-1},z$. This difference in notation is because it proves computationally convenient in the respective parts of the paper.

We observe that with respect to our generators 
\begin{equation*}
 \gamma_i \mathfrak{g}_{p,q} =
  \begin{cases}
   \operatorname{span}_{\RR}\left\{x_{i+1} , \ldots, x_{p-1}, y_{i+1}\ldots, y_{q-1}, z\right\}& \mbox{ for $2\leqslant i \leqslant q-2$}\\
   \operatorname{span}_{\RR}\left\{x_{i+1} , \ldots, x_{p-1}, z\right\}& \mbox{ for $q-1 \leqslant i \leqslant p-2$}\\
   \mathbf R z &\mbox{ for $i = p-2$.}\\
 \end{cases}
\end{equation*}
Identifying $\gamma_{i} \mathfrak{g}_{p,q} / \gamma_{i+1} \mathfrak{g}_{p,q}$ with $\mathbf R x_{1}\oplus \RR x_2 \oplus \mathbf R y_{1}\oplus \RR y_2$, for $i=1$, $\mathbf R x_{i+1} \oplus \mathbf R y_{i+1}$, for $2\leqslant i \leqslant q-2$, $\mathbf R x_{i+1}$, for $q-1 \leqslant i \leqslant p-2$, and $\mathbf R z$, for $i = p-1$, we can define the brackets of $\mathfrak{g}_{p,q}$ and of $\mathsf{gr}(\mathfrak{g}_{p,q})$ on the same vector space. If $p=q$ then $\mathfrak{g}_{p,q}$ is Carnot-graded, otherwise all the brackets are the same in $\mathfrak{g}_{p,q}$ and $\mathsf{gr}(\mathfrak{g}_{p,q})$ except that $[y_1, y_{q-1}] = z$ in $\mathfrak{g}_{p,q}$ while $[y_1, y_{q-1}] = 0$ in $\mathsf{gr}(\mathfrak{g}_{p,q})$.
We deduce that
\begin{equation}
\label{eq:Carnot-ofG-pq}
\mathsf{gr}(G_{p,q}) =
\begin{cases}
G_{p,q} & p = q \\
L_{p} \times L_{q-1} & p \neq q. 
\end{cases}
\end{equation}
\label{exm:Carnot-of-l_p-central-l-q}
\end{example}

\subsection{Tools for computing $H^2(\mathfrak{g}, \mathbf R)^{\geqslant r}$}
\label{subsec:tools-for-computingH2r}
Rephrasing the construction of central extensions from cohomology classes, we state a criterion to decide membership in $H^2(\mathfrak{g}, \mathbf R)^{\geqslant r}$:

\begin{proposition}
\label{lem:first-criterion-cohomology-to-central-extension}
Let $r \geqslant 2$.
The cocycle $\omega \in Z^2(\mathfrak{g}, \mathbf R)$ defines a cohomology class $[\omega] \in H^2(\mathfrak g, \mathbf R)^{\geqslant r}$ if and only if there exist $s \geqslant 1$ and a sequence of pairs $(X_i,Y_i) \in  \mathfrak{g} \times \mathfrak{g}$, $1 \leqslant i \leqslant s$, such that 
\begin{align}
\label{eq:Xi-are-in-gamma-ri} X_i \in \gamma_{r_{1,i}} \mathfrak{g} \text { and } Y_i \in \gamma_{r_{2,i}} \mathfrak{g} & \text{ with } r_{1,i} \leqslant r_{2,i}\text { and } r_{1,i}+r_{2,i}= r, \tag{$\Delta_1$} \\
\label{eq:bracket-is-zero} \sum_{i=1}^s \left[X_i,Y_i\right] & = 0 \tag{$\Delta_2$}, \\
\label{eq:omega-is-nonzero} \sum_{i=1}^s \omega(X_i,Y_i) & = 1 \tag{$\Delta_3$}.
\end{align}
\end{proposition}

\begin{proof}
Assume \eqref{eq:Xi-are-in-gamma-ri}, \eqref{eq:bracket-is-zero} and \eqref{eq:omega-is-nonzero} and let $\pi : \widetilde{\mathfrak{g}} \to \mathfrak{g}$ be the central extension associated to $\omega$; decompose $\widetilde{\mathfrak{g}}$ as a product $\mathfrak{g} \times \mathbf{R}$. In accordance with the definition of $r$-central extension we must prove that $(0,1) \in \gamma_r \widetilde{\mathfrak{g}}$. By  \eqref{eq:bracket-is-zero} and \eqref{eq:omega-is-nonzero} we may represent this element as $\sum_{i=1}^s [\widetilde{X}_i, \widetilde{Y}_i]$ where $\pi(\widetilde{X}_i) = X_i$, resp.\  $\pi(\widetilde{Y}_i) = Y_i$. Note that by \eqref{eq:Xi-are-in-gamma-ri} we may assume that $\widetilde{X}_i \in \gamma_{r_{1,i}} \widetilde{\mathfrak{g}}$ and $\widetilde{Y}_i \in \gamma_{r_{2,i}} \widetilde{\mathfrak{g}}$, and, since $r_{1,i} + r_{2,i} = r$ for all $i$, we deduce that $(0,1)=\sum_{i=1}^s [\widetilde{X}_i, \widetilde{Y}_i] \in \gamma_r \widetilde{\mathfrak{g}}$.

Conversely, assuming that $\widetilde{\mathfrak{g}} \to \mathfrak{g}$ is $r$-central, one can write $(0,1) = \sum_{i=1}^s [(U_{i,1},s_{i,1}),\ldots, (U_{i,r},s_{i,r})]$ with  $s_{i,j} \in \mathbf R$. It is then sufficient to set $X_i = U_{i,1}$ and $Y_i = \pi([(U_{i,2}, s_{i,2}), \ldots , (U_{i,r},s_{i,r})])$.
\end{proof} 

\begin{remark}
Combined with the results of the previous section Proposition \ref{lem:first-criterion-cohomology-to-central-extension} implies Pittet's lower bound on the Dehn function in {\cite[Theorem 3.1]{PittetIsopNil}}. Indeed, Pittet's criterion is equivalent to checking conditions \eqref{eq:Xi-are-in-gamma-ri}, \eqref{eq:bracket-is-zero}, \eqref{eq:omega-is-nonzero} with $s=1$, that is, with only one pair $(X,Y)=(X_1,Y_1)$. To see this note that the elements $X$ and $Y$ then generate an abelian Lie subalgebra $\mathfrak{a}$ of $\mathfrak{g}$, and \cite[Th 3.1]{PittetIsopNil} requires that the map $\iota^\ast: H^2(\mathfrak{g}, \mathbf R) \to H^2(\mathfrak{a}, \mathbf R)$ associated to $\iota: \mathfrak{a} \to \mathfrak{g}$ be nonzero, which amounts to asking for the existence of a cocycle $\omega$ satisfying \eqref{eq:omega-is-nonzero}.
We note that Pittet's exponent $d(\Gamma)$ nevertheless coincides with the growth exponent of the centralized Dehn function up to dimension $6$ included (See \S \ref{sec:low-dimension}).
\end{remark}

\begin{remark}
\label{rem:maximal-degree-central-extensions}
In the special case when $r$ is greater than the nilpotency class $c$ of $\mathfrak{g}$ (i.e. when $r = c+1$) the condition \eqref{eq:bracket-is-zero} is automatic given the assumptions on $X$ and $Y$ (as $r_1 + r_2 >c$) and $(c+1)$-central extensions are the central extensions of step $c+1$. 
For this reason ruling out the existence of $r$-central extensions is a simpler task when $r=c+1$.
\end{remark}

When $\mathfrak{g}$ is Carnot gradable we can go further into the description of cohomology classes yielding $r$-central extensions.
Let $(\mathfrak{g}_i)$ be the Carnot grading on $\mathfrak{g}$ with $\mathfrak{g}_i$ representing $\gamma_i \mathfrak{g}/\gamma_{i+1} \mathfrak{g}$.
Correspondingly, $\bigwedge^1 \mathfrak g^\star = \mathfrak g^\star$ can be graded in the following way: for $i \geqslant 1$ we set $(\bigwedge^1 \mathfrak{g}^\star)_i = \pi_i^\ast \operatorname{Hom}(\mathfrak{g}_i, \mathbf R)$ where $\pi_i$ is the projection to $\mathfrak{g}_i$. 

The exterior square $\bigwedge^2 \mathfrak g^\star$ is then graded by
\[ 
(\bigwedge\nolimits^{\!2} \mathfrak{g}^\star)_{k} = \bigoplus_{i+j = k} (\bigwedge\nolimits^{\!1} \mathfrak{g}^{\star})_i \wedge (\bigwedge\nolimits^{\!1} \mathfrak{g}^\star)_j. 
\]
Since $(\mathfrak{g}_i)$ is a Lie algebra grading on $\mathfrak{g}$, the differential $d: \bigwedge\nolimits^{\!n} \mathfrak g^\star \to \bigwedge\nolimits^{\!n+1} \mathfrak{g}^\star$ has degree $0$ with respect to these gradings. In particular, the cohomology group $H^2(\mathfrak{g}, \mathbf R)$ is also graded and the cohomology classes of weight $r$ under this grading produce $r$-central extensions.

\begin{example}
\label{exm:filiform-extension-of-filiform}
The Dehn function of the model filiform group $L_p$ is at least of order $n^p$.
Indeed, denote by $\mathfrak{l}_p$ the Lie algebra with basis $\left\{x_1, \ldots, x_{p-1}, z\right\}$, where $[x_1, x_i] = x_{i+1}$ for $2\leqslant i\leqslant p-2$ and $[x_1,x_{p-1}] = z$.
For $\xi_1, \ldots , \xi_{p-1}, \zeta$ its dual basis, the cohomology class $[\xi_1 \wedge \zeta]$ corresponding to the tautological extension $\mathfrak{l}_{p+1} \to \mathfrak l_p$ has degree $p$ under the associated grading on $H^2(\mathfrak{l}_p, \mathbf R)$.
\end{example}

We will compute the grading on $H^2(\mathfrak{l}_p, \mathbf R)$ below (see Remark \ref{Rmk:grading-of-lp}).
However, the groups $\mathfrak{g}_{p,q}$ that we are considering are not Carnot gradable for $p\neq q$, $p,q\geqslant 3$. Thus, our main tool in this section will be the criterion provided by Proposition \ref{lem:first-criterion-cohomology-to-central-extension}.

\subsection{Central extensions of central products}
\label{subsec:central-extensions-central-products}
We refer to the introduction for the definition of a central product $\mathfrak{k}\times_{\theta} \mathfrak{l}$ of Lie algebras $\mathfrak{k}$ and $\mathfrak{l}$ (resp. $K\times_{\theta} L$ of groups $K$ and $L$). Here we will be interested in understanding the existence of central extensions of central products in general and, more specifically, in the context of the central products $\mathfrak{g}_{p,q}$. We start with two general results.

\begin{lemma}
\label{lem:direct-product-central-extension}
Let $k, \ell$ be positive integers such that $2 \leqslant k, \ell$.
Let $\mathfrak k$ and $\mathfrak l$ be nilpotent real Lie algebras of step $k$ and $\ell$ respectively, and with one-dimensional center. Then for any isomorphism $\theta: Z(\mathfrak{k}) \to Z(\mathfrak{l})$, the extension $\mathfrak{k} \times \mathfrak{l} \to \mathfrak{k} \times_\theta \mathfrak{l}$ is $\min(k,\ell)$-central.
\end{lemma}

\begin{proof}
Without loss of generality assume that $k \geqslant \ell$.
Since the centers of both factors are one-dimensional, they are contained in the last nonzero term of the central series. Let $z$ generate $Z(\mathfrak{k})$. Then the generator $(z,  \theta(z))$ of $\ker (\mathfrak{k} \times \mathfrak{l} \to \mathfrak{k} \times_\theta \mathfrak{l})$ lies in $\gamma_\ell (\mathfrak{k} \times \mathfrak{l})$, but not in $\gamma_{\ell+1} (\mathfrak{k} \times \mathfrak{l})$. 
\end{proof}

\begin{lemma}
\label{prop:no-central-extension-max-degree}
Let $k, \ell$ be positive integers such that $2 \leqslant k, \ell$.
Let $\mathfrak k$ and $\mathfrak l$ be nilpotent real Lie algebras of step $k$ and $\ell$ respectively, and with one-dimensional center.
Let $\mathfrak g$ be the central product of $\mathfrak{k}$ and $\mathfrak{l}$. Then 
$\mathfrak{g}$ has no $r$-central extension for $r \geqslant \max(k,\ell)+1$. 
\end{lemma}

\begin{proof}
Assume $k \geqslant \ell$. Then $\mathfrak{g}$ is $k$-nilpotent, meaning that $\gamma_{k+1}\mathfrak{g} = 0$. Identify $\mathfrak k$ and $\mathfrak l$ with their images in $\mathfrak g$. 
Let $x_1, \ldots, x_s \in \mathfrak k$ be such that $x_i \notin [\mathfrak k, \mathfrak k]$ and $\operatorname{Liespan} \lbrace x_1,\dots, x_s \rbrace = \mathfrak k$. Let $y_1, \ldots, y_t \in \mathfrak l$ be such that $y_i \notin [\mathfrak l, \mathfrak l]$ and $\operatorname{Liespan} \lbrace y_1,\dots,y_t \rbrace = \mathfrak l$. Let $\widetilde{\mathfrak g}$ sit in the central extension 
\begin{equation}
0 \to \langle z' \rangle \longrightarrow \widetilde{\mathfrak g} \overset{\pi}{\longrightarrow} \mathfrak g \to 0,
\label{eq:central-extension-for-contradiction}
\end{equation}
and let $z \in \widetilde{\mathfrak g}$ be such that $\langle \pi(z) \rangle = Z(\mathfrak g)$.  For $i = 1, \ldots ,s$ and $j = 1, \ldots, t$ let $\widetilde{x}_i$ and $\widetilde{y}_j$ be such that $\pi(\widetilde{x}_i) = x_i$ and $\pi (\widetilde{y}_j) = y_j$. 

Note that $[\widetilde x_i, \widetilde y] \in \langle z' \rangle$ for all $i$ if $\pi(\widetilde{y}) \in \mathfrak l$, and that $[\widetilde y_j, \widetilde x] \in \langle z' \rangle$ for all $j$ if $\pi(\widetilde{x}) \in \mathfrak k$. Since $z'$ is central it follows that, for $m \geqslant 3$, 
$m$-fold commutators of $\widetilde{x}_j$ and $\widetilde{y}_j$ vanish, unless they only contain $\widetilde{x}_j$'s (resp. $\widetilde{y}_j$'s). 
Indeed, the only commutators where this is not trivially true are  the 
$\left[ \widetilde{y}_{i_1},\widetilde{x}_{i_2},\dots , \widetilde{x}_{i_m}\right]$ (resp.
$\left[\widetilde{x}_{i_1},\widetilde{y}_{i_2},\dots , \widetilde{y}_{i_m}\right]$) and they vanish by the Jacobi identity.

If $\widetilde{\mathfrak{g}}$ has step $k+1$ we may thus assume that there are $i_1, \ldots, i_{k+1} \in \lbrace 1, \ldots, s \rbrace$ such that $z' = [\widetilde{x}_{i_1}, \ldots ,\widetilde{x}_{i_{k+1}}]$.
Since 
$[\widetilde{x}_{i_2}, \ldots, \widetilde{x}_{i_{k+1}}] \in \pi^{-1}(\gamma_k \mathfrak{k}) =  \langle z, z' \rangle = \pi^{-1}(\gamma_\ell \mathfrak{l})$ we may rewrite $[\widetilde{x}_{i_2}, \ldots, \widetilde{x}_{i_{k+1}}]$ as $\alpha_1 [\widetilde{y}_{i_1}, \ldots ,\widetilde{y}_{i_\ell}] + \alpha_2 z'$ for $\alpha_1,\alpha_2 \in \RR$. Thus, 
$z' = \alpha_1 [\widetilde{x}_{i_1}, \widetilde{y}_{i_1}, \ldots, \widetilde{y}_{i_\ell}] + \alpha_2 [\widetilde{x}_{i_1}, z'] = 0$,
a contradiction.
\end{proof}

We now turn to the specific case of $\mathfrak{g}_{p,q}$ for $p>q\geqslant 3$.
\begin{proposition}
\label{lem:no-central-extension}
Assume $p > q \geqslant 3$. Then, 
$\mathfrak{g}_{p,q}$ admits a $(p-1)$-central extension if and only if $p$ is even.
\end{proposition}
Since it relies on a cohomology computation for $\mathfrak{g}_{p,q}$, the proof will simultaneously provide the following formulae for the Betti numbers of the lattices $\Gamma_{p,q}\leqslant G_{p,q}$ and $\Lambda_p\times \Lambda_{q-1}\leqslant L_p\times L_{q-1} = \mathsf{gr}(G_{p,q})$. 

\begin{lemma}[Betti numbers]
\label{lem:H^2g_p,q}
Let $p > q \geqslant 3$. Then
\begin{equation}
\label{eq:Betti-numbers-of-gpq}
b_2(\Gamma_{p,q}) = \left\lfloor \frac{p}{2} \right\rfloor + \left\lfloor \frac{q}{2} \right\rfloor  + 3,
\end{equation}
and 
\begin{equation}
\label{eq:Betti-numbers_of-graded}
b_2(\Lambda_p \times \Lambda_{q-1}) = \left\lceil \frac{p}{2} \right\rceil + \left\lfloor \frac{q}{2} \right\rfloor  + 4.
\end{equation}
In particular, the Betti number discrepancy $b_2(\Lambda_p \times \Lambda_{q-1}) - b_2(\Gamma_{p,q})$ is $1$ if $p$ is even and $2$ if $p$ is odd.
\end{lemma}

\begin{remark}
For $(p ,q) = (4, 3)$ and $(5,3)$ the Betti numbers of $\Gamma_{p,q}$ can be extracted from Magnin's comprehensive tables of cohomologies in dimension less or equal 7. Magnin denoted the corresponding Lie algebras $\mathcal G_{6,2}$ and $\mathcal G_{7,3.17}$ respectively \cite{Magnin}. 
For $(p,q)=(4,3)$ these were also computed in \cite[(25)-(26)]{delBarcoSpectralSequence} and \cite[6.19]{cornulier2017sublinear}.
\end{remark}

As before, we will perform our Betti number computations using Lie algebra cohomology. To deduce Lemma \ref{lem:H^2g_p,q} we will thus invoke the following result, that is due to Matsushima for $k=1,2$ and Nomizu for all $k$ \cite[Corollary 7.28]{RagDS}. It shows that the real cohomology of finitely generated torsion-free nilpotent groups only depends on the real Malcev completion, an early manifestation of Shalom's theorem.

\begin{lemma}
\label{lem:lie-algebra-cohomology-and-nilmanifold-cohomology}
Let $\Gamma$ be a lattice in a simply connected nilpotent Lie group $G$ with Lie algebra $\mathfrak{g}$. Then $H^k(G/\Gamma, \mathbf R) = H^k(\mathfrak{g}, \mathbf R)$.
\end{lemma}

Before proving Proposition \ref{lem:no-central-extension} and Lemma \ref{lem:H^2g_p,q}, we observe that they allow us to complete the proof of Theorem \ref{propIntro:centralDehn}, modulo the lower bound from \S \ref{sec:lower-bounds-forms}.
\begin{proof}[{Proof  of Theorem \ref{propIntro:centralDehn}}]
 The first part is a direct consequence of Lemmas \ref{lem:direct-product-central-extension} and \ref{prop:no-central-extension-max-degree} and Propositions \ref{lem:no-central-extension} and \ref{prop:Centralrcentralextensions}. The second part follows from Theorem \ref{thmIntro:Main}, whose proof will be completed in \S \ref{sec:lower-bounds-forms}.
\end{proof}

\begin{proof}[Proof of Proposition \ref{lem:no-central-extension} and Lemma \ref{lem:H^2g_p,q}]
Note that if $\alpha$ is a one-form on $\mathfrak{g}$ then $d\alpha$ is the two-form such that $d \alpha (u,v) = - \alpha \left( [u,v] \right)$ for every $u,v \in \mathfrak{g}$. We will use this below without further mention when computing differentials.

Let $\left\{\xi_1, \ldots, \xi_{p-1}, \zeta, \eta_1, \ldots, \eta_{q-1}\right\}$ be the dual basis of the basis $\left\{x_1, \ldots, x_{p-1}, z , y_1, \ldots, y_{q-1}\right\}$ of $\mathfrak{g}_{p,q}$. The restriction of the subset $\left\{\xi_1,\ldots, \xi_{p-1},\zeta\right\}$ to $\mathfrak{l}_p$ defines the basis of $\mathfrak{l}_p^{\ast}$ induced by the canonical embedding $\mathfrak{l}_p\hookrightarrow \mathfrak{g}_{p,q}$.

We first prove \eqref{eq:Betti-numbers_of-graded}. For this we need to compute $H^2(\mathfrak{l}_p, \mathbf R)$. Since this computation is well-known (it is originally due to Vergne \cite{VergneCohomologieNilpotente}), we only sketch it here and leave the details as an exercise to the reader. We emphasize that this is an exercise well-worth doing to get acquainted with Lie algebra cohomology computations.

We use abbreviations of the form $\xi_{i,j}:=\xi_i \wedge \xi_j$ (and similar for $3$-fold wedge-products). Further we denote $\xi_p:= \zeta$. Note that $d\xi_1 = d\xi_2 =0$, while $d\xi_i = - \xi_{1,i-1}$ for $3 \leqslant i \leqslant p$. We deduce that 
\[
B^2(\mathfrak{l}_p, \mathbf R) = \operatorname{span} \lbrace \xi_{1,2},  \ldots ,  \xi_{1,p-1} \rbrace. 
\]

Let now $ \omega = \sum_{1 \leqslant i < j \leqslant p} a_{i,j} \xi_{i,j} \in \Lambda^2 (\mathfrak{l}_p, \mathbf R)$. We obtain the identities 
\begin{align*}
d \omega ~ & =~  \sum_{i=1}^{p-1} d \left( \xi_i \wedge \sum_{j = i+1}^{p} a_{i,j} \xi_j \right) \\  
& = \sum_{2 \leqslant i < j \leqslant p-1, j \neq i+1} (-a_{i+1,j} - a_{i,j+1}) \xi_{1,i,j} 
+ \sum_{i=2}^{p-2} (- a_{i,i+2}) \xi_{1,i,i+1} + \sum_{i=2}^{p-2} (-a_{i+1,p}) \xi_{1,i,p}.
\end{align*}

Solving the linear system of equations obtained by imposing $d\omega = 0$ yields
\begin{align*}
Z^2(\mathfrak{l}_p, \mathbf R) = \operatorname{span}
 \left\{ \xi_{1, 2}, \ldots, \xi_{1,p-1}, \xi_{1,p}, \nu_4, \nu_{6}, \nu_{8}, \ldots, \nu_{2p'} \right\},
\end{align*}
where $p'=\lceil \frac{p}{2}\rceil$ and $\nu_{2l}:= \xi_{2,2l-1}-\xi_{3,2l-2}+\dots -(-1)^l \xi_{l,l+1}$ for $2\leqslant l \leqslant p'$.

It follows that the cohomology classes  represented by $\left\{\xi_{1,p}, \nu_{2 \cdot 2}, \ldots, \nu_{2p'}\right\}$ form a basis of $H^2(\mathfrak{l}_p,\RR)$ and thus that
\begin{equation}
\label{eq:b_2-l_p}
\operatorname{rank} H^2(\mathfrak{l}_p, \mathbf R) = p'-1 +1 = p'.
\end{equation}

We can now compute the second Betti number of $\mathfrak{l}_{p} \times \mathfrak{l}_{q-1}$, and thus of all lattices in $\mathsf{gr}(G_{p,q})=L_p\times L_{q-1}$ and in particular of $\Lambda_p\times \Lambda_{q-1}$. 
Indeed, using the K\"unneth formula and \eqref{eq:b_2-l_p}, the class of the Poincar\'e polynomial of $\mathfrak l_p \times \mathfrak l_{q-1}$ in $\mathbf Z[t]/(t^3)$ is
\begin{align*}
(1+2t+b_2(\mathfrak l_p)t^2)(1+2t+b_2(\mathfrak l_{q-1}) t^2) & = 
1 + 4t + \left( 4 + \lceil p/2 \rceil + \lceil (q-1)/2 \rceil \right)t^2 \\
& = 
1 + 4t + \left( 4 + \lceil p/2 \rceil + \lfloor q/2 \rfloor \right)t^2.
\end{align*}
and we deduce that $\operatorname{rank}(H^2(\mathfrak{l}_{p} \times \mathfrak{l}_{q-1}, \mathbf R))= \lceil \frac{p}{2}\rceil+\lfloor\frac{q}{2}\rfloor +4$. This completes the proof of \eqref{eq:Betti-numbers_of-graded}.

While we don't use it at this point we record the following observation; it is well-known to experts.

\begin{remark}
\label{Rmk:grading-of-lp}
The degree $2$-cohomology of $\mathfrak{l}_p$ is graded as follows:
$H^2(\mathfrak{l}_p, \mathbf R)^{2k-1} = \operatorname{span}[\nu_{2k}]$ for $2 \leqslant k  < p'$,  
\[ 
H^2(\mathfrak{l}_p, \mathbf R)^p = 
\begin{cases}
\operatorname{span}\lbrace [\nu_{2p'}], [\xi_1 \wedge \xi_p] \rbrace & p \text{ odd} \\
\operatorname{span} \lbrace [\xi_1 \wedge \xi_{p}] \rbrace & p \text{ even}, \\
\end{cases}\]
and all other degrees vanish. In particular, $\nu_{2k}$ represents a $(2k-1)$-central extension.
\end{remark}
 This observation is interesting in itself and also in view of \S \ref{sec:low-dimension}. However, most importantly comparing it to \eqref{eqn:Z2-and-B2-forgpq} provides some intuition for why $\mathfrak{g}_{p,q}$ admits no $(p-1)$-central extension when $p$ is odd. Indeed, we will see that for $p$ odd the analogous cohomology class $\left[\nu_{2p'}\right]$ vanishes in $H^2(\mathfrak{g}_{p,q},\RR)$, while it survives when $p$ is even. In fact, it is precisely the form that induces the $(p-1)$-central extension of $\mathfrak{g}_{p,q}$ when $p$ is even. This is also mirrored by the distinct Betti number discrepancies in Lemma \ref{lem:H^2g_p,q}. Computationally, this difference is reflected in the fact that in $\mathfrak{l}_p$ we have $d\zeta = -\xi_1\wedge\xi_{p-1}$, while in $\mathfrak{g}_{p,q}$ we have $d\zeta =-\xi_1\wedge \xi_{p-1}-\eta_1\wedge\eta_{q-1}$. This ultimately implies that the coefficient $a_{2,p}$ of $\xi_2\wedge \xi_p$ must be zero for every cocycle $\omega$ in $\mathfrak{g}_{p,q}$ with $p$ odd, while it can be non-zero for cocycles in $\mathfrak{l}_p$ or in $\mathfrak{g}_{p,q}$ when $p$ is even. \vspace{.3cm}

We now move on to the computation of $H^2(\mathfrak{g}_{p,q}, \mathbf R)$. We will again use abbreviations of the form $\xi_{i,j}=\xi_i \wedge \xi_j$, $\eta_{i,j}=\eta_i\wedge \eta_j$ etc.

Note that $d\xi_1 = d\xi_2 = d\eta_1 = d\eta_2 = 0$, $d \xi_i = - \xi_{1,i-1}$ and $d\eta_j = - \eta_{1,j-1}$ for $3 \leqslant i \leqslant p-1$ and $3 \leqslant j \leqslant q-1$, and that $d\zeta = - \xi_{1,p-1} - \eta_{1,q-1}$.
We can decompose $\omega \in \Lambda^2(\mathfrak{g}_{p,q}, \mathbf R)$ as
\begin{align}
\omega & = \sum_{i=1}^{p-2} \sum_{j = i+1}^{p-1} a_{i,j} \xi_{i,j} +  
\sum_{i=1}^{q-2} \sum_{j = i+1}^{q-1} b_{i,j} \eta_{i,j}  + \sum_{k=1}^{p-1} c_k \xi_k \wedge \zeta + \sum_{\ell = 1}^{q-1} e_\ell \eta_\ell \wedge \zeta + \sum_{m=1}^{p-1} \sum_{n=1}^{q-1} f_{m,n} \xi_m \wedge \eta_n \notag \\
& = \omega_a + \omega_b + \omega_c + \omega_e + \omega_f. 
\label{eq:decomposition-omega}
\end{align}
We deduce that
\begin{align*}
d \omega &
=
\sum_{2 \leqslant i < j \leqslant p-2, j \neq i+1} (-a_{i+1,j} - a_{i,j+1}) \xi_{1,i,j} 
+ \sum_{i=2}^{p-3} (- a_{i,i+2}) \xi_{1,i,i+1} + \sum_{i=2}^{p-3} (-a_{i+1,p-1} - c_i) \xi_{1,i,p-1} \\
& + 
\sum_{2 \leqslant i < j \leqslant q-2, j \neq i+1} (-b_{i+1,j} - b_{i,j+1}) \eta_{1,i,j} 
+ \sum_{i=2}^{q-3} (- b_{i,i+2}) \eta_{1,i,i+1} + \sum_{i=2}^{q-3} (-b_{i+1,q-1} - e_i) \eta_{1,i,q-1} \\
& + (-c_{p-2}) \xi_{1,p-2,p-1} + \sum_{k=1}^{p-1} c_k \xi_k \wedge \eta_{1,q-1} 
+ \sum_{k=2}^{p-2} c_{k+1} \xi_{1,k} \wedge \zeta 
+ \sum_{m=1}^{p-1} \sum_{n=2}^{q-2} f_{m,n+1} \xi_{m} \wedge \eta_{1,n}
\\
& 
+ (-e_{q-2}) \eta_{1,q-2,q-1} 
+ \sum_{\ell = 1}^{q-1} e_\ell \eta_{\ell} \wedge \xi_{1,p-1}  
+ \sum_{\ell=2}^{q-2} e_{\ell+1} \eta_{1,\ell} \wedge \zeta 
+ \sum_{m=2}^{p-2} \sum_{n=1}^{q-1} (-f_{m+1,n}) \xi_{1,m} \wedge \eta_n. 
\end{align*}
Hence $d \omega  = 0$ if and only if 
\begin{align}
a_{i+1,j} + a_{i,j+1} = 0 & & 2 \leqslant i < j \leqslant p-2, j \neq i+1 \label{eq:homogeneous-cocycle-a} \\
a_{i,i+2} = 0 & & 2 \leqslant i \leqslant p-3 \label{eq:bdry-condition-diag-a} \\
b_{i+1,j} + b_{i,j+1} = 0 & & 2 \leqslant i < j \leqslant q-2, j \neq i+1 
\label{eq:homogeneous-cocycle-b}\\
b_{i,i+2} = 0 & & 2 \leqslant i \leqslant q-3 \label{eq:bdry-condition-diag-b} \\
a_{i+1, p-1} + c_i = 0 & & 2 \leqslant i \leqslant p-3 \label{eq:a_ic_i} \\ 
c_k = 0 & & 1 \leqslant k \leqslant p-1 \label{eq:ck_zero} \\
b_{i+1, q-1} + e_i = 0 & & 2 \leqslant i \leqslant q-3 \label{eq:b_ie_i}  \\ 
e_\ell = 0 & & 1 \leqslant \ell \leqslant q-1 \label{eq:el_zero} \\
f_{m,n} = 0 & & \max(m,n) \geqslant 3. \label{eq:omega_f}
\end{align}
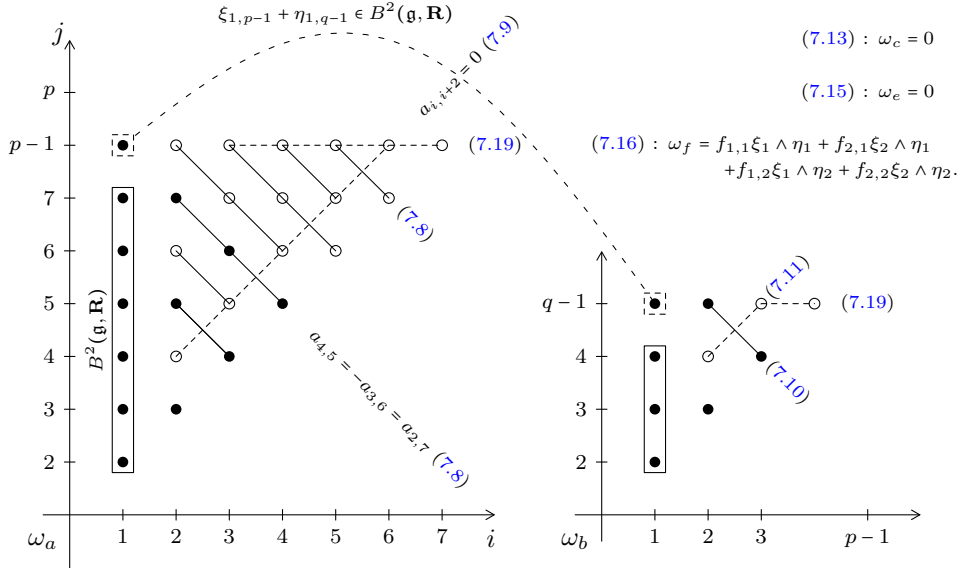
\begin{figure}
\begin{center}
\begin{tikzpicture}[line cap=round,line join=round,>=angle 60,x=0.7cm,y=0.7cm]
%\draw [color=black!30,dash pattern=on 2pt off 2pt, xstep=0.7cm,ystep=0.7cm] 
%(0,2) grid (10,10);
\clip(-1.5,0) rectangle (17,11);
\draw[->,color=black] (-0.5,1) -- (8,1);
\draw (-0.5,0.5) node {$\omega_a$};
\draw (9.5,0.5) node {$\omega_b$};
\foreach \x in {1,2,3,4,5,6,7}
\draw[shift={(\x,1)},color=black] (0pt,2pt) -- (0pt,-2pt) node[below] {\footnotesize $\x$};
\draw[->,color=black] (0,-0.5) -- (0,10);
\foreach \y in {2,3,4,5,6,7}
\draw[shift={(0,\y)},color=black] (2pt,0pt) -- (-2pt,0pt) node[left] {\footnotesize $\y$};
\draw[shift={(0,8)},color=black] (2pt,0pt) -- (-2pt,0pt) node[left] {\footnotesize $p-1$};
\draw[shift={(0,9)},color=black] (2pt,0pt) -- (-2pt,0pt) node[left] {\footnotesize $p$};
%\draw[color=black] (0pt,-10pt) node[right] {\footnotesize $0$};
%\fill[color=black,fill=blue,fill opacity=0.1] (-3,-2) -- (-1,-2) -- (-1,5) -- (-3,5) -- cycle;
\draw [color=black] (2,5)-- (3,4);
\draw [color=black] (2,7)-- (4,5);
\draw [color=black] (2,6)-- (3,5);
\draw [color=black] (2,4) circle (2pt) ;
\draw [color=black] (3,5) circle (2pt) ;
\draw [color=black] (4,6) circle (2pt) ;
\fill [color=black] (2,5) circle (2pt) ;
\fill [color=black] (3,6) circle (2pt) ;
\fill [color=black] (4,5) circle (2pt) ;
\fill [color=black] (3,4) circle (2pt) ;
\fill [color=black] (2,7) circle (2pt) ;
\fill [color=black] (2,3) circle (2pt) ;
\draw (-0.5,10.5) node[anchor=north west] {$ j $};
\draw (7.66,0.9) node[anchor=north west] {$ i $};
\draw [dash pattern=on 2pt off 2pt] 
(2,4) -- (6,8);
\draw [dash pattern=on 2pt off 2pt] 
(3,8) -- (7,8);
\begin{scriptsize}
\draw (7.5,9.5) node [rotate=45] {$a_{i,i+2} = 0$ \eqref{eq:bdry-condition-diag-a}};
\draw (6.5,6.5) node [rotate=-45] {\eqref{eq:homogeneous-cocycle-a}};
\draw (6,3) node [rotate=-45] {$a_{4,5} = -a_{3,6} = a_{2,7}$ \eqref{eq:homogeneous-cocycle-a}};
\draw (8,8) node{ \eqref{eq:bdry-max-a} };
\draw (0.5,4.5) node [rotate=90] {$B^2(\mathfrak{g}, \mathbf R)$};
\draw (5,10.5) node {$\xi_{1,p-1} + \eta_{1,q-1} \in B^2(\mathfrak{g}, \mathbf R)$};
\end{scriptsize}
\draw [color=black] (3,8) circle (2pt) ;
\draw [color=black] (4,8) circle (2pt) ;
\draw [color=black] (5,8) circle (2pt) ;
\draw [color=black] (6,8) circle (2pt) ;
\draw [color=black] (7,8) circle (2pt) ;
\draw [color=black] (5,7) circle (2pt) ;
\draw [color=black] (2,8) circle (2pt) ;
\draw [color=black] (3,7) circle (2pt) ;
\draw [color=black] (4,7) circle (2pt) ;
\draw [color=black] (5,6) circle (2pt) ;
\draw [color=black] (2,6) circle (2pt) ;
\draw [color=black] (6,7) circle (2pt) ;
\draw [color=black] (2,5)-- (3,4);
\draw [color=black] (2,8)-- (4,6);
\draw [color=black] (3,8)-- (5,6);
\draw [color=black] (4,8)-- (5,7);
\draw [color=black] (5,8)-- (6,7);
\foreach \y in {2,3,4,5,6,7,8}
\fill [color=black] (1,\y) circle (2pt) ;
\draw (0.8,1.8) rectangle (1.2,7.2);
\draw (0.8,7.8) [dash pattern = on 2pt off 2pt] rectangle (1.2,8.2);

\draw [domain=1:11, dash pattern=on 2pt off 4pt] plot(\x,{8.3-3*\x/10+3.5*sin((\x-1)*0.314 r)});

\draw[->,color=black] (9.5,1) -- (16,1);
\foreach \x in {1,2,3}
\draw[shift={(10+\x,1)},color=black] (0pt,2pt) -- (0pt,-2pt) node[below] {\footnotesize $\x$};
\draw[->,color=black] (10,0.5) -- (10,6);
\foreach \y in {2,3,4}
\draw[shift={(10,\y)},color=black] (2pt,0pt) -- (-2pt,0pt) node[left] {\footnotesize $\y$};
\draw[shift={(10,5)},color=black] (2pt,0pt) -- (-2pt,0pt) node[left] {\footnotesize $q-1$};
%\draw[shift={(10,6)},color=black] (2pt,0pt) -- (-2pt,0pt) node[left] {\footnotesize $q$};
\draw[shift={(15,1)},color=black] (0pt,2pt) -- (0pt,-2pt) node[below] {\footnotesize $p-1$};
\draw [color=black] (13,5) circle (2pt) ;
\draw [color=black] (14,5) circle (2pt) ;
\fill [color=black] (12,5) circle (2pt) ;
\fill [color=black] (13,4) circle (2pt) ;
\draw [color=black] (12,5)-- (13,4);
\draw [color=black, dash pattern = on 2pt off 2pt] (12,4) -- (13,5) -- (14,5);
\draw [color=black] (12,4) circle (2pt) ;
\fill [color=black] (11,3) circle (2pt) ;
\fill [color=black] (11,2) circle (2pt) ;
\fill [color=black] (11,4) circle (2pt) ;
\fill [color=black] (11,5) circle (2pt) ;
\draw (10.8,4.8) [dash pattern = on 2pt off 2pt] rectangle (11.2,5.2);
\fill [color=black] (12,3) circle (2pt) ;
\draw (10.8,1.8) rectangle (11.2,4.2);
\begin{scriptsize}
\draw (13.5,5.5) node [rotate=45] {\eqref{eq:bdry-condition-diag-b}};
\draw (13.5,3.5) node [rotate=-45] {\eqref{eq:homogeneous-cocycle-b}};
%\draw (13,3) node [rotate=-45] {$a_{4,5} = -a_{3,6} = a_{2,7}$ \eqref{eq:homogeneous-cocycle-b}};
\draw (15,5) node{ \eqref{eq:bdry-max-a} };

\draw (15,10) node{ \eqref{eq:ck_zero} : $\omega_c = 0$};
\draw (15,9) node { \eqref{eq:el_zero} : $\omega_e = 0$};
\draw (13,8) node { \eqref{eq:omega_f} : $\omega_f = f_{1,1} \xi_1 \wedge \eta_1 + f_{2,1} \xi_2 \wedge \eta_1 $};
\draw (14.5,7.5) node {$ + f_{1,2} \xi_1 \wedge \eta_2 + f_{2,2} \xi_2 \wedge \eta_2$.};

\end{scriptsize}

\end{tikzpicture}
\end{center}
\caption{Determination of $Z^2(\mathfrak{g}_{p,q}, \mathbf R)$ and $H^2(\mathfrak{g}_{p,q}, \mathbf R)$ with $(p,q)=(9,6)$. 
The cocycle $\omega$ is decomposed as in \eqref{eq:decomposition-omega}.
On the left, resp.\ on the right, a $\circ$ at $(i,j)$ denotes $a_{i,j}=0$ resp.\ $b_{i,j}=0$; plain edges denote linear dependences and vanishing.}
\label{fig:cocycle-gpq}
\end{figure}

The equations \eqref{eq:homogeneous-cocycle-a} are equivalent to
\begin{equation}
\label{eq:conditions-cocycle-a-reformulated}
a_{i,j}= (-1)^{i-i'}a_{i',j'}, \mbox{ for } i+j= i'+j',~ 2\leqslant i <j\leqslant p-1,~ 2\leqslant i'<j'\leqslant p-1,
\end{equation} 
and the equations \eqref{eq:homogeneous-cocycle-b} are equivalent to
\begin{equation}
\label{eq:conditions-cocycle-b-reformulated}
b_{i,j}= (-1)^{i-i'}b_{i',j'}, \mbox{ for } i+j= i'+j',~ 2\leqslant i <j\leqslant q-1,~ 2\leqslant i'<j'\leqslant q-1.
\end{equation}

Combining \eqref{eq:a_ic_i} and \eqref{eq:ck_zero} (resp. \eqref{eq:b_ie_i} and \eqref{eq:el_zero}) yields
\begin{align}
a_{i,p-1} = 0 & & 3 \leqslant i \leqslant p-2, \label{eq:bdry-max-a} \\
b_{i,p-1} = 0 & & 3 \leqslant i \leqslant q-2. \label{eq:bdry-max-b} 
\end{align}

The $a_{i,j}$ (resp. $b_{i,j}$) with $2\leqslant i<j\leqslant p-1$ (resp. $2\leqslant i<j\leqslant q-1$) are now completely determined by \eqref{eq:bdry-condition-diag-a}, \eqref{eq:conditions-cocycle-a-reformulated} and \eqref{eq:bdry-max-a} (resp. \eqref{eq:bdry-condition-diag-b}, \eqref{eq:conditions-cocycle-b-reformulated} and \eqref{eq:bdry-max-b}).
Indeed, for $2\leqslant i<j \leqslant p-1$ conditions \eqref{eq:conditions-cocycle-a-reformulated} and \eqref{eq:bdry-condition-diag-a} imply that the $a_{i,j}$ with $i+j\geqslant 6$ vanish whenever $i+j$ is even and conditions \eqref{eq:conditions-cocycle-a-reformulated} and  \eqref{eq:bdry-max-a} imply that $a_{i,j}=0$ for $i+j\geqslant p+2$. 
The only constraint on the remaining $a_{i,j}$ with $2\leqslant i<j \leqslant p-1$ is that they satisfy condition \eqref{eq:conditions-cocycle-a-reformulated}. Similar considerations apply for the $b_{i,j}$ with $2\leqslant i<j\leqslant q-1$. 
Since the $a_{1,i}$ for $2\leqslant i \leqslant p-1$ and the $b_{1,i}$ for $2\leqslant i<q-1$ are unconstrained, we conclude from the constraints \eqref{eq:ck_zero}, \eqref{eq:el_zero} and \eqref{eq:omega_f} on the $c_i$, $e_i$ and
$f_{m,n}$ resp., that
\begin{equation}
Z^2(\mathfrak{g}_{p,q}, \mathbf R) = \operatorname{span}
\begin{cases}
\xi_{1,i} & 2 \leqslant i \leqslant p-1, \\
\eta_{1,i} & 2 \leqslant i \leqslant q-1, \\
\nu_{2k} & 2 \leqslant k \leqslant p'', \\
\widetilde{\nu}_{2\ell} & 2 \leqslant \ell \leqslant q'', \\
\xi_{m} \wedge \eta_n & 1 \leqslant m,n \leqslant 2,
\end{cases}\quad
B^2(\mathfrak{g}_{p,q}, \mathbf R) = \operatorname{span}
\begin{cases}
\xi_{1,i} & 2 \leqslant i \leqslant p-2, \\
\eta_{1,i} & 2 \leqslant i \leqslant q-2, \\
\xi_{1,p-1} + \eta_{1,q-1},
\end{cases}
\label{eqn:Z2-and-B2-forgpq}
\end{equation}
where $p''=\left\lfloor \frac{p}{2}\right\rfloor$, $q''=\left\lfloor \frac{q}{2}\right\rfloor$, $\nu_{2\ell}:=\xi_{2, 2 \ell -1} - \xi_{3,2 \ell - 2} + \cdots - (-1)^\ell \xi_{\ell, \ell+1}$ and $\widetilde{\nu}_{2\ell}:=\eta_{2, 2 \ell -1} - \eta_{3,2 \ell - 2} + \cdots - (-1)^\ell \eta_{\ell, \ell+1}$. We refer to Figure \ref{fig:cocycle-gpq} for a visual illustration of our computation for $(p,q)=(9,6)$. 

A basis of $H^2(\mathfrak{g}_ {p,q},\RR)$ is thus given by
\[
 \left\{ \left[\nu_{2\cdot 2}\right],\dots, \left[\nu_{2\cdot p''}\right], \left[\widetilde{\nu}_{2\cdot 2}\right],\dots, \left[\widetilde{\nu}_{2\cdot q''}\right],\left[\eta_{1,q-1}\right]=-\left[\xi_{1,p-1}\right], \left[\xi_i\wedge \eta_j\right]~ 1\leqslant i,j\leqslant 2 ~\right\}.
\]
We deduce that
\begin{align*}
\operatorname{rank} H^2(\mathfrak{g}_{p,q}, \mathbf R)  
& = (p''-1)+(q''-1) + 1 + 4 \\
& = p'' + q'' +3.
\end{align*}
This concludes the proof of \eqref{eq:Betti-numbers-of-gpq}.\vspace{.3cm}

\emph{We can now complete the proof of Proposition \ref{lem:no-central-extension}.} 

\emph{If $p$ is even} then $2p''=p$ and $\nu_{p}(x_2, x_{p-1}) =1$, while $[x_2, x_{p-1}] = 0$. By Proposition \ref{lem:first-criterion-cohomology-to-central-extension}, the cohomology class represented by $\nu_p$ defines the desired $(p-1)$-central extension.

\emph{Thus, assume that $p$ is odd}, and assume for a contradiction that there is a $(p-1)$-central extension defined by a cocycle $\omega$.
By Proposition \ref{lem:first-criterion-cohomology-to-central-extension} there are elements $X_i,Y_i\in \mathfrak{g}_{p,q}$, $1\leqslant i \leqslant s$ which satisfy \eqref{eq:Xi-are-in-gamma-ri}, \eqref{eq:bracket-is-zero}, \eqref{eq:omega-is-nonzero}.
Up to reordering the pairs we can assume that  $r_{1,i}=1$ for $i=1, \dots, t\leqslant s$ and $r_{1,i}>1$ for $i>t$.
Decompose $X_i$ and $Y_i$ into
\[
X_i = \tau_{1,i} x_1 + \cdots + \tau_{p-1,i} x_{p-1} + \sigma_i z + \tau'_{1,i} y_1 + \cdots + \tau'_{q-1,i} y_{q-1} 
\]
\[ 
Y_i = \lambda_{1,i} x_1 + \cdots + \lambda_{p-1,i} x_{p-1} + \mu_i z + \lambda'_{1,i} y_1 + \cdots + \lambda'_{q-1,i} y_{q-1} 
\]
with $\tau_{j,i}, \tau'_{j,i}, \sigma_i, \lambda_{j,i}, \lambda'_{j,i}, \mu_i \in \mathbf R$.

Assume $t <s$. Then, for $i>t$ we have $r_{1,i}>1$ and $r_{2,i}= p-r_{1,i}-1$. Since $\mathfrak{g}_{p,q}$ is metabelian $[X_i,Y_i]=0$.
Using that $X_i,Y_i \in \gamma_2 \mathfrak{g}_{p,q}$ we deduce that $(\xi_{m} \wedge \eta_n)(X_i,Y_i) = 0$ for $1\leqslant m,n \leqslant 2$ and $(\eta_1\wedge \eta_{q-1})(X_i,Y_i)=0$. Moreover, since $p$ is odd we deduce that $2q''\leqslant 2p''\leqslant p-1$.

Observing that $\nu_{2k} \mid_{ \gamma_{r_1} \mathfrak g_{p,q} \times \gamma_{r_2} \mathfrak{g}_{p,q}} = 0$ if $r_1 + r_2 > 2k-1$ (resp. $\widetilde \nu_{2\ell} \mid_{ \gamma_{r_1} \mathfrak g_{p,q} \times \gamma_{r_2} \mathfrak{g}_{p,q}} = 0$ if $r_1 + r_2 > 2\ell-1$), we deduce that for $2\leqslant k \leqslant p''$ (resp. for $2\leqslant \ell \leqslant q''$) we have $\nu_{2k}(X_i,Y_i)=\widetilde{\nu}_{2\ell}(X_i,Y_i)=0$. We conclude that $\omega(X_i,Y_i)=0$.

Hence, we may assume that $t=s$ and therefore $r_{i,1}=1$ and $r_{i,2}=p-2$ for all $i$. In particular $Y_i=\lambda_{p-1,i}x_{p-1}+\mu_i z$ and \eqref{eq:bracket-is-zero} implies that
\[
 0=\sum_{i=1}^s \left[X_i,Y_i\right] = \sum_{i=1}^s \tau_{1,i} \lambda_{p-1,i} z.
\]
On the other hand evaluating the sum of the $\omega(X_i,Y_i)$ yields
\[
1= \sum_{i=1}^s \omega(X_i,Y_i) = \sum_{i=1}^s \omega(X_i,\lambda_{p-1,i} x_{p-1}+\mu_i z) = \sum_{i=1}^s \tau_{1,i} \lambda_{p-1,i}\omega(x_1,x_{p-1}),
\]
where for the last identity we observe that the only pair of basis vectors of the form $(\ast, z)$ and $(\ast, x_{p-1})$ on which our basis of representatives of cohomology classes does not vanish is $(x_1,x_{p-1})$. Comparing the two equalities gives a contradiction. This completes the proof of Proposition \ref{lem:no-central-extension}.

\end{proof}

\begin{remark}
The cocyle $\nu_{2k}$ of the preceding proof was introduced by Vergne in her computation of $H^2(\mathfrak{l}_p,\RR)$ \cite{VergneCohomologieNilpotente}. When $p \geqslant 5$ is odd the central extension associated to the cocycle $\nu_{p+1}$ on $\mathfrak{l}_p$ produces a filiform, but not model filiform, Carnot graded Lie algebra of dimension $p$. Vergne proved its existence and uniqueness (see also Figure \ref{fig:genealogy-filiform-algebras}).
\end{remark}

\begin{figure}
\[
\xymatrix@C=0.8cm{
& & & & \mathfrak{l}'_6 & & \mathfrak{l}'_8 & \\
\mathbf R^2 \ar@{-}[r]^{\xi_1 \wedge \xi_2} & \mathfrak{l}_3 \ar@{-}[r]^{\xi_1 \wedge \xi_3} & \mathfrak{l}_4 \ar@{-}[r]^{\xi_1 \wedge \xi_4} & \mathfrak{l}_5 \ar@{-}[r]_{\xi_1 \wedge \xi_5} \ar@{-}[ru]^{\nu_6}
& \mathfrak{l}_6 \ar@{-}[r]^{\xi_1 \wedge \xi_6} & \mathfrak{l}_7 \ar@{-}[r]_{\xi_1 \wedge \xi_7} \ar@{-}[ru]^{\nu_8} & \mathfrak{l}_8 \ar@{--}[r] & } 
\]
\caption{Carnot graded filiform Lie algebras ($\mathfrak{l}_3$ is the Lie algebra of the Heisenberg group, $\mathfrak{l}_6$ and $\mathfrak{l}'_6$ are $\mathscr{L}_{6,18}$ and $\mathscr{L}_{6,16}$ in de Graaf's list \cite{deGraafclass}). We use the same notation for the cocycles as in the proof of Proposition \ref{lem:no-central-extension}.
}\label{fig:genealogy-filiform-algebras}
\end{figure}

The lower bound of $n^{p-1}$ on the Dehn functions of $G_{p,p-1}$ for even $p\geqslant 4$ (resp. of $G_{p,p}$ for all $p\geqslant 4$) that we obtain from central extensions is sharp by Theorem \ref{thm:Upperbound}. In contrast, and maybe at first rather unexpectedly, for odd $p$ the lower bound of $n^{p-2}$ on the Dehn function of $G_{p,p-1}$ obtained from central extensions is not sharp. In fact  not even its exponent is sharp, providing the first example of a group with this property. We will prove this in the next section. There is a moral reason for this discrepancy, which we will exploit in the next section; for an explanation of this we refer to \S \ref{sec:Intro-moral-proof}.

\section{Lower bounds on the Dehn function from integration of forms}
\label{sec:lower-bounds-forms}
In this section we will explain how to obtain lower bounds on the Dehn functions of the $G_{p,q}$ by integrating bounded forms on Lie groups. In \S \ref{sec:lbf-mainresult} we state the main result of this section and explain how it can be reduced to finding a suitable 1-form that satisfies a certain boundedness condition; this boundedness condition can be thought of as a discretized version of being a primitive of a bounded 2-form. In \S \ref{sec:lbf-filiform-Lie-groups} we will provide a linear representation of the filiform Lie group in all dimensions and construct an exact invariant 2-form from it. In \S \ref{sec:lbf-proof-main-result} we will show how to modify this 2-form to obtain a suitable exact bounded 2-form. Finally, in \S \ref{subsec:unif-bdd-length-form} we will show that this bounded 2-form is the differential of a 1-form that satisfies the boundedness condition from \S \ref{sec:lbf-mainresult} and deduce the desired lower bounds on the Dehn function of $G_{p,q}$.

\subsection{Lower bounds from bounded forms}
\label{sec:lbf-mainresult}
\begin{theorem}
\label{thm:lower-bounds-form}
For $p\geqslant q\geqslant 1$ the Dehn function of $G_{p+1,q+1}$ is $\succcurlyeq n^{p}$.
\end{theorem}

Before going into the proof of Theorem \ref{thm:lower-bounds-form}, we summarise our approach for obtaining the desired lower bound on $\delta_{G_{p+1,q+1}}$. It suffices to find a family of null-homotopic words $w_{p+1,n}$ of length $\simeq n$ and area $\simeq n^p$. A natural candidate for $w_{p+1,n}$ is the word $\Omega_{p+1}(n):=\Omega_{p+1}(n,\dots,n)$ defined via the embedding of $L_{p+1}$ in the first factor of $G_{p+1,q+1}$. The reason for this is that its image with respect to the projection $G_{p+1,q+1}\to L_{p}$ is a product of the words $x_1^{-n} \Omega_p(n)^{-1} x_1^n$ and $\Omega_p(n)$, which both have area $n^p$ in $L_p$. One way of showing that these two words have the asserted area is by integrating them along a primitive of the $2$-form $\xi_1\wedge \xi_p$ from Section \ref{sec:lower-bound-dehn}, defining the $p$-central extension $L_{p+1}\to L_p$. However, a naive attempt to use the same argument to show that $\Omega_{p+1}(n)$ has area $\simeq n^p$ fails, because one can show that $\int_{\Omega_{p+1}(n)} \xi_1\wedge \xi_p = 0$. We overcome this obstacle by replacing $\xi_1\wedge \xi_p$ by a suitable ``bounded'' deformation of itself and then showing that integration over a primitive of this deformation now yields a non-trivial lower area bound which is $\simeq n^{p-1}$. This allows us to  prove Theorem \ref{thm:lower-bounds-form} and confirms our intuition regarding the area of the $\Omega_{p+1}(n)$.\vspace{.2cm}

We now provide the details of our argument. Let us start by introducing some useful notation.
Let $G$ be a connected Lie group equipped with a left-invariant Riemannian metric, and let $S$ be a compact generating subset of $G$.  For a smooth path $\gamma: \left[a,b\right]\to G$ we denote by $L(\gamma)$ its length with respect to the chosen metric on $G$.
We assign to every $s\in S$ a smooth choice of path $\gamma_s$ from $1_G$ to $s$ such that the set $\left\{L(\gamma_s)\mid s\in S\right\}$ is bounded. This allows us to associate to every word $w$ in $S$, a path $\overline{w}$. In what follows, such a path will be called a {\it word-path}. 

 We denote $g\ast \gamma$ the action of $G$ by left translation on the set of paths in $G$. Let us denote $w\cdot w'$ the concatenation of the words $w$ and $w'$.
\begin{proposition}\label{prop:discreteStokes} 
We let $\langle S\mid R\rangle$ be a compact presentation of a connected Lie group $G$ that we also equip with a left-invariant Riemannian metric. 
Assume that there exists a continuous $1$-form $\alpha$, and $C<\infty$ such that for every word-loop $\overline{r}$ associated to a relator $r\in R$ and every $g\in G$, 
\begin{equation}
\label{eqn:bddrelsprop-Stokes}
\left|\int_{g\ast \overline{r}}\alpha\right|\leqslant C.
\end{equation}
Let $w$ be null-homotopic, then
\[\Area(w)\geqslant \frac{1}{C}\left|\int_{\overline{w}}\alpha\right|.\]
\end{proposition}
\begin{proof}
We make the following trivial but crucial observation: given two words $w$ and $w'$ in the alphabet $S$, we have
\begin{equation}\label{eq:concatenation}
\int_{\overline{w\cdot w'}} \alpha= \int_{[w]\ast \overline{w'}}\alpha +\int_{\overline{w}}\alpha.
\end{equation}
In particular, if $w$ and $w'$ are null-homotopic, i.e.\ $[w]=[w']=1_G$, then 
\begin{equation}\label{eq:concatenationLoops}
\int_{\overline{w\cdot w'}} \alpha= \int_{\overline{w'}}\alpha +\int_{\overline{w}}\alpha.
\end{equation}
We also easily deduce from \eqref{eq:concatenation} that if $w$ and $w'$ represent the same element of the free group, then 
\begin{equation}\label{eq:freeequiv}
\int_{\overline{w}}\alpha=\int_{\overline{w}'}\alpha.
\end{equation}
Finally, if $w$ is null-homotopic, i.e.\ $[w]=1_G$, and $u$ is any word, then we get
\begin{equation}\label{eq:conj}
\int_{\overline{w'}}\alpha=\int_{[u]\ast \overline{w}}\alpha,
\end{equation}
where $w'=u\cdot w\cdot u^{-1}$.
Now let $w$ be a word of size $\leqslant n$ in $S$ that freely equals a product of $N$ conjugates of relators.
Then combining \eqref{eq:freeequiv}, \eqref{eq:concatenationLoops}, \eqref{eq:conj} and \eqref{eqn:bddrelsprop-Stokes} in this order, we conclude that $$\left|\int_{\overline{w}}\alpha\right|\leqslant N\cdot C,$$
so we are done.   
 \end{proof}

\subsection{Linear representations of filiform Lie groups}
\label{sec:lbf-filiform-Lie-groups}
It is well-known and easy to check that a linear representation of the Lie algebra of $L_p$ is given by
\[
\mathfrak{l}_p:=\mathrm{Lie}(L_p)=\left\{ \left.
\left(
\begin{array}{cccccc} 
0 		& t_1 	& 0 	 &\cdots 	& 0 	 &t_p\\
\vdots 	&\ddots & \ddots & \ddots 	& \vdots &t_{p-1}\\
		&		& \ddots &			& 0		 & \vdots\\
		&		&		 & 0		& t_1	 & t_3\\
\vdots  &		&		 &			& 0		 &t_2\\
0		&\cdots &		 &			& \cdots &0
\end{array}
\right)\right|
t_1,\dots,t_p\in \RR
\right\}
\]
with the commutator bracket $\left[A,B\right]:= AB-BA$ on matrices. Thus, we can obtain a linear representation of $L_p$ by computing the image $\mathrm{exp}(\mathfrak{l}_p)$. We will now make this explicit. For this we introduce the notation
\[
B_{t_1}:=\left(
\begin{array}{ccccc} 
 0 		& t_1 	& 0 	 &\cdots 	& 0 	 \\
\vdots 	&\ddots & \ddots & \ddots 	& \vdots \\
		&		& \ddots &			& 0		 \\
		&		&		 & 0		& t_1	 \\
0	    &		&		 &\cdots	& 0		 \\
\end{array}
\right)\in \RR^{(p-1)\times (p-1)}
\]
and observe that for $t=(t_1,\dots,t_p)$ and $A_t\in \mathfrak{l}_p$ we obtain
\[
 e^{A_t} = \left(\begin{array}{cc} e^{B_{t_1}} & v_t \\ 0 & 1 \end{array}\right)
\]
for a suitable $v_t\in \RR^{p-1}$. Moreover, it is easy to derive by induction that 
\[
e^{B_{t_1}} =
\left(
\begin{array}{ccccc} 
1 		& t_1 	& \frac{t_1^2}{2!} 	 &\cdots 	& \frac{t_1^{p-2}}{(p-2)!} 	 \\
0	 	&\ddots & \ddots 			 &\ddots 	& \vdots \\
\vdots	&\ddots	& \ddots 			 &			& \frac{t_1^2}{2!}		 \\
		&		&\ddots	 			 & 1		& t_1	 \\
0	    &\cdots	&		 			 &0			& 1		 
\end{array}
\right). 
\]
From this we deduce that
\[
 v_t = \left(
 \begin{array}{c} 
 t_p+ \sum_{n=1}^{p-2} \frac{t_1^n}{n!} t_{p-n}\\
  \vdots \\ 
 t_k + \sum_{n=1}^{k-2} \frac{t_1^n}{n!} t_{k-n}\\
 \vdots\\
 t_3 + \sum_{n=1}^{1} \frac{t_1^n}{n!}t_{3-n}\\
 t_2
 \end{array}
 \right).
\]
Finally the change of coordinates $u(t):=(u_1(t),\dots, u_p(t))$ with $u_1(t)=t_1$ and $u_i(t)= t_i + \sum_{n=1}^{i-2} \frac{t_1^n}{n!} t_{i-n}$ for $2\leqslant i \leqslant p$ provides a diffeomorphism from $\RR^p$ to $L_p$ represented as the linear subgroup
\[
\left\{ 
S_{u}:=
\left( \left.
\begin{array}{ccc|c} &&& u_p\\ 
&e^{B_{u_1}}& &\vdots\\ 
&&& u_2\\ \hline 
0& \cdots &0 &1\\
\end{array} 
\right) 
\right| 
u=(u_1,\cdots,u_p)\in \RR^p \right\} < \mathrm{Gl}_p(\RR).
\]
For $u\in\RR^p$ we will denote by $\del_{u_1,u},\dots,\del_{u_p,u}$ the standard coordinate basis of $T_u\RR^p$. Note that the model filiform group with presentation $\mathcal{P}(\Lambda_p)$ as in \S \ref{subsec:Intro-central-products} embeds as a lattice via the identifications $x_1=\exp(\del_{u_1,0})$, $x_2=\exp(\del_{u_2,0})$ and $x_{i+1}=\left[x_1,x_i\right]$ for $2\leqslant i\leqslant p-1$. 

We will now use the linear representation to compute the left invariant vector fields corresponding to the standard basis $\del_{u_1,0}, \dots, \del_{u_p,0}$ of $T_0\RR^p$ at the identity. We denote by $S_{u,\ast}: TL_p\to TL_p$ the differential of the automorphism of $L_p$ defined by left-multiplication by $S_u$.
\begin{lemma}\label{lem:inv-vec-fields}
 With respect to the coordinates $u$ on $L_p$ a basis of left invariant vector fields is given by
 \[
  S_{u,\ast} \del_{u_1,0} = \del_{u_1,u}
 \]
 and
 \[
 S_{u,\ast} \del_{u_i,0} = \sum_{n=0}^{p-i} \frac{u_1^n}{n!} \del_{u_{i+n},u}.
 \]
\end{lemma}
\begin{proof}
 The first identity is an immediate consequence of the following identities
 \[
  S_{u,\ast}\cdot \frac{d}{du_1}|_{u=0} S_u = 
\left(
\begin{array}{ccc|c} &&& 0\\ 
&e^{B_{u_1}}\cdot \left(
\begin{array}{ccccc}
 0 		& 1 	& 0 	 &\cdots 	& 0 	 \\
\vdots 	&\ddots & \ddots & \ddots 	& \vdots \\
		&		& \ddots &			& 0		 \\
		&		&		 & 0		& 1		 \\
0		&		&		 & \cdots	& 0		 \\
\end{array}\right)
& &\vdots\\ 
&&& 0\\ \hline 
0& \cdots &0 &0\\
\end{array} 
\right)
=
\left(
\begin{array}{c|c}
\frac{d}{du_1} e^{B_{u_1}}&0\\ \hline 0& 0
\end{array}
\right).
 \]
To derive the identities for $2\leqslant i\leqslant p$ denote by $e_i\in \RR^{p-1}$ the $i$-th unit vector and observe that
\[
S_{u,\ast}\cdot \frac{d}{du_i}|_{u=0} 
=
\left( 
\begin{array}{ccc|c} &&& u_p\\ 
&e^{B_{u_1}}& &\vdots\\ 
&&& u_2\\ \hline 
0& \cdots &0 &1\\
\end{array} 
\right) 
\cdot
\left(
\begin{array}{c|c} 0 & e_{p-i+1}\\ \hline 0&0
\end{array} 
\right) 
=
\left(
\begin{array}{c|c} 0 & e^{B_{u_1}}\cdot e_{p-i+1} \\ \hline 0&0
\end{array} 
\right).  
\]
We deduce that in local coordinates we have $S_{u,\ast} \del_{u_i,0}=\sum_{n=0}^{p-i} \frac{u_1^n}{n!} \del_{u_{i+n},u}$. This completes the proof.
\end{proof}

It is now easy to check that the forms $du_1$ and $du_2$ corresponding to the first two coordinates are left $L_p$-invariant. Moreover, we obtain:
\begin{lemma}
The 1-form $\alpha$ defined by 
\[
\alpha_u= \sum_{k=0}^{p-2} (-1)^k \frac{u_1^k}{k!} du_{p-k}
\]
is the unique left $L_p$-invariant form with $\alpha_0=du_p$.
\end{lemma}
\begin{proof}
By definition $\alpha_0=du_p$ and using Lemma \ref{lem:inv-vec-fields} it is easy to check that $\alpha_u(S_{u,\ast}\del_{u_i,0}) =\delta_{p,i}$. This completes the proof.
\end{proof}

Finally we observe that the form $\beta$ defined by
\begin{equation}\label{def:1-form-beta}
 \beta_u=\sum_{k=0}^{p-2} (-1)^k \frac{u_1^{k+1}}{(k+1)!}du_{p-k}
\end{equation}
has left $L_p$-invariant differential
\[
d\beta = du_1 \wedge \alpha.
\]

In fact $d \beta$ is an explicit realisation in the coordinates $u_i$ of the 2-form $\xi_1\wedge \xi_p$ from the proof of Proposition \ref{lem:no-central-extension}. The reason we consider it is that it defines a p-central extension of $L_p$. However, as we have seen we face the problem that this form does not survive in $H^2(\mathfrak{g}_{p,q}, \RR)$ for $q<p$. Thus we can not use it directly to obtain a lower bound on $\delta_{G_{p,q}}(n)$ by defining a p-central extension and, as we have shown, there is actually not even a $(p-1)$-central extension of $G_{p,q}$ for $p$ odd and $q<p$. To overcome this problem and confirm our intuition that $\delta_{G_{p,q}}(n)\succcurlyeq n^{p-1}$, we will now pursue the approach sketched in \S \ref{sec:Intro-sketch-MainThm} of constructing a suitable perturbation $\beta_0$ of $\beta$ with bounded differential, which has integral $\simeq n^{p} $ on certain $(p+1)$-fold iterated commutators in $L_p$; they arise as images of null-homotopic words in $G_{p+1,q+1}$ with respect to the canonical projection. In view of \S \ref{sec:lbf-mainresult} this will allow us to deduce the desired lower bounds on the Dehn function.

\subsection{Construction of a suitable exact bounded 2-form}
\label{sec:lbf-proof-main-result}
To simplify notations, we shall denote for $n\in \RR$, $\Omega_k(n)=\Omega_k(\nnn)$ when $\nnn=(n,\ldots,n)\in \RR^k$, where $\Omega_k(\nnn)$ was defined in  \S \ref{sec:Omega}. Note that $\Omega_k(n)$ can be defined inductively by $\Omega_2(n)=\left[x_1^n,x_2^n\right]$ and $\Omega_{k+1}(n)=\left[x_1^n,\Omega_{k}(n)\right]$ for $k\geqslant 2$. We recall that $\Omega_p(n)$ defines a null-homotopic word in $L_p$.

\begin{remark}\label{rmk:exp-Omega}
One checks by induction on $k$ that the exponent sum of $x_1$ in any prefix word of $\Omega_k(n)^{\pm 1}$ lies in the interval $[-(k-1)n,0]$, for all $n\in \mathbb{N}$.
\end{remark}

We will show that the integral of the form $\beta$ along the loop defined by $\Omega_{p}(n)$ in $L_p$ is $n^{p}$. This is one way to prove that the null-homotopic words $\Omega_{p}(n)$ are area maximising in $L_p$. It makes $\beta$ a good candidate for showing that $\Omega_{p+1}(n)=\left[x_1^n,\Omega_{p}(n)\right]$ also has area $\simeq n^{p}$ in $L_p$. However,  
\[
 \int_{x_1^{-n}\Omega_{p}(n)^{-1}x_1^n} \beta = - \int_{\Omega_{p}(n)} \beta 
\]
and thus the integral of $\beta$ along $\Omega_{p+1}(n)$ vanishes (this is a direct consequence of \eqref{eq:conj} and the left $L_p$-invariance of $d\beta$). This means that the form $\beta$ won't allow us to obtain the desired lower bounds on the Dehn function. 

We will show that we can avoid this problem by replacing $\beta$ by a continuous perturbation $\beta_0$ with the property that the differential $d\beta_0$ exists for $u_1\neq 0$ and coincides with $d\beta$ for $u_1>0$ and with $-d\beta$ for $u_1<0$. Moreover, to simplify our calculations, we will consider the null-homopic word $w_{p+1,n}=x_1^{(p-1)n}\Omega_p(n)x_1^{-(p-1)n}\Omega_p(n)^{-1}$ instead of $\Omega_{p+1}(n)$. Its projection to $L_p$ consists of two disjoint loops $\gamma^+$ and $\gamma^-$, and a line connecting their basepoints. By Remark \ref{rmk:exp-Omega}, the exponent sum of $x_1$ in any prefix word of $\gamma^+$ (resp.\ $\gamma^-$) is positive (resp.\ negative). In particular, the image of $\gamma_+$ is contained in the set, where $d\beta_0=d\beta$, while the image of $\gamma_-$ is contained in the set, where $d\beta_0=-d\beta$. Since $d\beta$ is the 2-form defining the central extension $L_{p+1}\to L_p$, one can deduce from this that $\int_{w_{p+1,n}} \beta_0 \simeq n^p$. Below we provide the details of this argument and calculate the precise value of $\int_{w_{p+1,n}} \beta_0$.

We start by defining $\beta_0$:
\[
\beta_{0,u}:=\sgn(u_1) \sum_{k=0}^{p-2}(-1)^k \frac{u_1^{k+1}}{(k+1)!} du_{p-k}.
\]
A direct calculation shows
\[
d(\beta_{0,u}) = \left\{ 
\begin{array}{ll} 
-d\beta & \mbox{\hspace{.5cm} if } u_1<0\\
 d\beta & \mbox{\hspace{.5cm} if } u_1>0
\end{array}
\right.
\]

To evaluate the integral of $\beta_{0}$ along $w_{p+1,n}$ we need to evaluate it along each part of the loop. For this we will use the following result:
\begin{lemma}
\label{lem:int-x1n-x2n}
For $i=1,2$ and $\epsilon=\pm 1$, let $\gamma_i(t)=S_u\cdot \exp(\epsilon t \del_{u_i,u})=S_u\cdot x_i^{\epsilon\cdot t}$, $t\in \left[0,n\right]$ be a curve in $L_p$ with $\gamma_i(0)=S_u$ and $\gamma_i(n)= S_u \cdot x_i^{\epsilon n}$. Assume further that $u_1=L\cdot n$ for some $L\in \RR$. Then
\begin{enumerate}
 \item $\int_{\gamma_1} \beta_{0} = 0$ and $u'_1=u_1+\epsilon n$ for $u'\in \RR^p$ with $S_{u'}= \gamma_1(n)$;
 \item $\int_{\gamma_2} \beta_{0} = \epsilon n^p \sgn(L) \frac{L^{p-1}}{(p-1)!}$ and $u'_1=u_1$ for $u'\in \RR^p$ with $S_{u'}=\gamma_2(n)$.
\end{enumerate}
\end{lemma}
\begin{proof}
 Assertion (1) follows from Lemma \ref{lem:inv-vec-fields}, the vanishing of $\beta_{0,u}$ on $\del_{u_1,u}$ and
 \[
  \gamma_1(t)= 
\left(\begin{array}{ccc|c} &&& u_p\\ 
&e^{B_{u_1}}& &\vdots\\ 
&&& u_2\\ \hline 
0& \cdots &0 &1\\
\end{array} 
\right)
\cdot
\left(
\begin{array}{c|c} e^{B_{\epsilon t}} & 0\\ \hline 0&1
\end{array} 
\right)
= 
\left(\begin{array}{ccc|c} &&& u_p\\ 
&e^{B_{u_1+\epsilon t}}& &\vdots\\ 
&&& u_2\\ \hline 
0& \cdots &0 &1\\
\end{array} 
\right).
 \]
For Assertion (2) we first observe that
\[
\gamma_2(t)= 
\left(\begin{array}{ccc|c} &&& u_p\\ 
&e^{B_{u_1}}& &\vdots\\ 
&&& u_2\\ \hline 
0& \cdots &0 &1\\
\end{array} 
\right)
\cdot
\left(\begin{array}{ccc|c} &&& 0\\ 
&\mathrm{I}& &\vdots\\ 
&&& 0\\
&&& \epsilon t\\ \hline 
0& \cdots &0 &1
\end{array} 
\right)
=
\left(\begin{array}{ccc|c} &&& u_p+\epsilon \cdot t \frac{u_1^{p-2}}{(p-2)!}\\ 
&e^{B_{u_1}}& &\vdots\\ 
&&& u_3+\epsilon \cdot t u_1\\
&&& u_2+\epsilon \cdot  t\\ \hline 
0& \cdots &0 &1\\
\end{array} 
\right).
\]
We deduce that
\[
\dot{\gamma_2}(t) =  \epsilon \sum_{k=0}^{p-2} \frac{u_1^{p-2-k}}{(p-2-k)!} \del _{u_{p-k},u}.
\]
and that the $u_1$-coordinate is constant along $\gamma_2(t)$. Thus,
\begin{align*}
 \int_{\gamma_2} \beta_{0} &= \int_{0}^n \beta_{0,\gamma_2(t)}(\dot{\gamma_2}(t)) dt\\
 & = \int_0^n \epsilon \cdot \sgn(u_1) \sum_{k=0}^{p-2}(-1)^{k} \frac{u_1^{k+1}}{(k+1)!} \cdot \frac{u_1^{p-2-k}}{(p-2-k)!}\\
 & \stackrel{u_1=L\cdot n}{=} \int_0^n \epsilon\cdot  \sgn(L\cdot n) \sum_{k=0}^{p-2} (-1)^k n^{p-1} \frac{L^{p-1}}{(k+1)!(p-2-k)!}\\
 & = \frac{n^p}{(p-1)!}\epsilon\cdot \sgn(L)\cdot L^{p-1} \sum_{k=1}^{p-1} (-1)^k \binom{p-1}{k}\\
 &\stackrel{(1)}{=} \epsilon \frac{n^p}{(p-1)!} \sgn(L) L^{p-1},
\end{align*}
where in (1) we use the binomial formula $0=(1 + (-1))^{p-1}= \sum_{k=0}^{p-1}(-1)^k \binom{p-1}{k}$. This completes the proof.

\end{proof}

For a word $w(x_1,x_2)$, we introduce the notation $E_{x_1}(w)$ for its $x_1$-exponent sum. Lemma \ref{lem:int-x1n-x2n} shows that 
\begin{enumerate}
 \item if a word $w(x_1,x_2)$ represents the element $S_u$ in $L_p$ for $u\in \RR$ then $u_1$ coincides with $E_{x_1}(w)$;
 \item we can compute $\int_{w_{p+1,n}}\beta_0$ by reading $w_{p+1,n}$ from left to right and adding a contribution for every power of $x_2$ that we encounter. The contribution of such an $x_2$-power will depend solely on the $x_1$-exponent sum of its prefix word and the numerical value of the exponent of this $x_2$-power. In particular, this essentially reduces the computation of $\int_{w_{p+1,n}}\beta_0$ to a combinatorial problem.
\end{enumerate}

\begin{lemma}
\label{lem:count-exp-sums}
For $p\geqslant 2$ the word $\Omega_{p}(n)$ satisfies the following properties:
\begin{enumerate}
 \item $\Omega_{p}(n) = \prod_{j=1}^{N_p} x_1^{\epsilon_{j,1} n} x_2^{- n}x_1^{\epsilon_{j,2} n}x_2^{n}$ in freely reduced form for an integer $N_p$. In particular, the sign of the $x_2$-exponents alternates and the word starts with $x_1^{-n}x_2^{-n}$ and ends with $x_1^nx_2^n$;
  \item  for any decomposition of $\Omega_{p}(n)$ in freely reduced form as $w_1(x_1,x_2) x_2^{\epsilon\cdot n} w_2(x_1,x_2)$ there is $0\leqslant k \leqslant p-1$ with $E_{x_1}(w_1)=-k\cdot n$;
 \item for $0\leqslant k\leqslant p-1$ there are precisely $\binom{p-1}{k}$ ways of decomposing $\Omega_{p}(n)$ in freely reduced form as $w_1(x_1,x_2) x_2^{\epsilon\cdot n} w_2(x_1,x_2)$ with exponent sum $E_{x_1}(w_1)=-k\cdot n$ and $\epsilon =\pm 1$, and, moreover, for all of them $\epsilon = (-1)^k$.
  \item for $0\leqslant k\leqslant p-1$ there are precisely $\binom{p-1}{k}$ ways of decomposing $\Omega_{p}(n)^{-1}$ in freely reduced form as $w_1(x_1,x_2) x_2^{\epsilon\cdot n} w_2(x_1,x_2)$ with exponent sum $E_{x_1}(w_1)=-k\cdot n$ and $\epsilon =\pm 1$, and, moreover, for all of them $\epsilon = (-1)^{k+1}$.
\end{enumerate}
\end{lemma}
\begin{proof}
The proof is by induction on $p$. For $p=2$ we have $\Omega_{2}(n)=\left[x_1^n,x_2^n\right] =x_1^{-n}x_2^{-n}x_1^nx_2^n$ and one checks readily that all assertions hold. Hence, assume that the result holds for some $p\geqslant 2$ and consider $\Omega_{p+1}(n)=\left[x_1^n,\Omega_{p}(n)\right]=x_1^{-n}(\Omega_{p}(n))^{-1} x_1^n \Omega_{p}(n)$. The only new free reduction takes place in the middle of the word, where we reduce $x_2^nx_1^nx_1^nx_1^{-n}x_2^{-n}$ to $x_2^nx_1^nx_2^{-n}$. In particular, it is immediate from the fact that the exponent signs of the $x_2^{\pm n}$ are alternating in $\Omega_{p}(n)$ that the same holds for $\Omega_{p+1}(n)$ and it follows readily that (1) holds for $\Omega_{p+1}(n)$.

Since we have $E_{x_1}\left(x_1^{-n} (\Omega_{p}(n))^{-1}x_1^n\right)=0$ it suffices to count the exponent sums and signs for the $x_1^{-n} (\Omega_{p}(n))^{-1}x_1^n$-part of $\Omega_{p+1}(n)$ with those for the $\Omega_{p}(n)$-part following from the induction hypothesis for $p$.  

To determine the result for the $x_1^{-n} (\Omega_{p}(n))^{-1}x_1^n$-part, let 
\begin{equation}\label{eqn:Decomp-Omega}
\Omega_{p}(n)^{-1}= w_1(x_1,x_2) x_2^{\epsilon n} w_2(x_1,x_2)
\end{equation}
be a decomposition of the freely reduced word represented by $\Omega_p(n)^{-1}$. Its inverse writes 
$\Omega_{p}(n) = w_2^{-1}x_2^{-\epsilon n} w_1^{-1}.$

Observe that $E_{x_1}(w_1)=E_{x_1}(w_2^{-1})$, since $E_{x_1}(\Omega_{p}(n))=0$ and for any word $v(x_1,x_2)$ we have $E_{x_1}(v^{-1})=-E_{x_1}(v)$. It follows that the number of decompositions of $(\Omega_{p}(n))^{-1}$ as in \eqref{eqn:Decomp-Omega} with $E_{x_1}(w_1)=-k\cdot n$ is identical with the number of such decompositions of $\Omega_{p}(n)$. However, the exponent sign of the subsequent $x_2^{\pm n}$ is $(-1)^{k+1}$. This implies (4) for $p$. Moreover, for  $1\leqslant k \leqslant p$, we deduce that the freely reduced form of $x_1^{-n}\Omega_{p}(n)x_1^n$ admits precisely $\binom{p-1}{k-1}$ distinct decompositions as in (2),(3) with $x_1$-exponent sum $-k\cdot n$ and $\epsilon=(-1)^{k+2}=(-1)^k$. Thus, for $0\leqslant k \leqslant p$ the total number of $x_2^{\epsilon n}$ with preceeding $x_1$-exponent sum $-k\cdot n$ is
\[
 \binom{p-1}{k-1} + \binom{p-1}{k}= \binom{p}{k}
\]
and the corresponding $\epsilon$ is always $(-1)^k$. Moreover, there are no decompositions with other $k$-values. This completes the proof of (2) and (3) for $p+1$.
\end{proof}

Using Lemmas \ref{lem:int-x1n-x2n} and \ref{lem:count-exp-sums} we can now compute $\int_{w_{p+1,n}} \beta_{0}$. To do so we first prove the following auxiliary lemma.
\begin{lemma}\label{lem:fC-is-constant}
 For $p\geqslant 2$ we have
 \[
 \sum_{k=0}^{p-1} (-1)^k (-k)^{p-1} \binom{p-1}{k} = (p-1)!
 \]
\end{lemma}
\begin{proof}
Denote $S= \sum_{k=0}^{p-1} (-1)^k (-k)^{p-1} \binom{p-1}{k}$.
We consider the function $h(x)=  \sum_{k=0}^{p-1} e^{ikx} \binom{p-1}{k}$.
Note that $h(x)=\alpha(x)^{p-1}$, where $\alpha(x)=1+e^{ix}$. We observe that 
\begin{equation}\label{eq:sum=h derived p-1 times}
h^{(p-1)}(\pi)= \sum_{k=0}^{p-1} (-1)^k (ik)^{p-1} \binom{p-1}{k} = (-i)^{p-1}S.
\end{equation}
We check by induction on $0\leq m\leq p-1$ that 
\[h^{(m)}(x)-(p-1)(p-2)\cdots (p-m)i^me^{imx}\alpha(x)^{p-m-1}\]
is a multiple of $\alpha(x)^{p-m}$. Since $\alpha(\pi)=0$, we deduce that 
\[h^{(p-1)}(\pi)=(p-1)!i^{p-1}(-1)^{p-1}=(p-1)!(-i)^{p-1},\]
which, combined with (\ref{eq:sum=h derived p-1 times}) implies the lemma.
\end{proof}

\begin{proposition}\label{prop:integral-betaC--1-new}
 The identity
 \[
  \int_{w_{p+1,n}} \beta_{0} =  2n^p
 \]
 holds.
\end{proposition}
\begin{proof}
  As a direct consequence of Lemmas \ref{lem:int-x1n-x2n}, \ref{lem:count-exp-sums}, and the definition of $w_{p+1,n}$, we obtain
  \begin{align*}
  	\int_{w_{p+1,n}}\beta_0 =& \frac{n^p}{(p-1)!} \sum_{k=0}^{p-1} \left((-1)^k (p-1-k)^{p-1} \binom{p-1}{k} - (-1)^{k+1} (-k)^{p-1}\binom{p-1}{k} \right) \\
  	=&\frac{n^p}{(p-1)!} \sum_{k=0}^{p-1} \left((-1)^k (p-1-k)^{p-1} \binom{p-1}{p-1-k} + (-1)^{p-1-k} k^{p-1}\binom{p-1}{k} \right)\\
  	=&2\frac{n^p}{(p-1)!} \sum_{k=0}^{p-1} (-1)^{p-1-k} k^{p-1}\binom{p-1}{k}\\
  	=&2n^p,
  \end{align*}
  where the last equality follows from Lemma \ref{lem:fC-is-constant}.
\end{proof}

\begin{remark}
Note that we can use similar methods to prove that the word $\Omega_{p+1}(n)$ has area bounded below by a function $\geq \frac{n^{p}}{(p-1)!}$. To do so we use that the area of a word is invariant under conjugation and apply the above methods to the conjugate $x_1^n\Omega_{p+1}(n)x_1^{-n}$ of $\Omega_{p+1}(n)$. The reason this works it that the loop described by $x_1^n\Omega_{p+1}(n)x_1^{-n}$  attains values in both of the subsets $\left\{u_1<0\right\}$ and $\left\{u_1>0\right\}$ of $\mathbb{R}^n$.
\end{remark}

\subsection{Integrating along loops of uniformly bounded length}
\label{subsec:unif-bdd-length-form}

We now fix a left-invariant Riemannian metric $g$ on $L_p$, which we choose such that $\del_{u_1,0},\dots,\del_{u_p,0}$ is an orthonormal basis of $T_0 L_p$ under the homeomorphic identification $\RR^p\cong L_p$ with coordinates $u=(u_1,\dots,u_p)$ on $\RR^p$ as before. In this section we will prove the following result, which will allow us to apply Proposition \ref{prop:discreteStokes}. Throughout this section we will assume that all paths are piece-wise smooth.

\begin{proposition}\label{prop:uniformboundonsmallloops}
 For $M>0$ there exists a constant $K=K(M)>0$ such that for every loop $\gamma: \left[0,1\right] \to L_p$ of length $L(\gamma)\leqslant M$ we have
 \[
  \left|\int_{\gamma} \beta_{0}\right| \leqslant K.
 \]
\end{proposition}

We will deduce Proposition \ref{prop:uniformboundonsmallloops} from the fact that we can decompose $\RR^p$ into two sets on which $d\beta_{0}$ is equal to the invariant forms $d\beta$ (resp. $-d\beta$) and the subsequent lemma. 

\begin{lemma}\label{lem:uniformboundonloops}
 Let $M>0$ and let $\alpha$ be a 1-form on $L_p$ with invariant differential $d\alpha$. Then there is a constant $K=K(M,\alpha)$ such that
 \[
  \left|\int_{\gamma} \alpha\right| \leqslant K,
 \]
 for all loops $\gamma: \left[0,1\right]\to L_p$ with $L(\gamma)\leqslant M$.
\end{lemma}
\begin{proof}
Let $f: D\to L_p$  be a filling disc for $\gamma$. By Stokes' Theorem we have
\[
 \left|\int_{\gamma}\alpha\right| = \left|\int_{D} f^{\ast}d\alpha\right| \leqslant K_0\cdot \mathrm{Area}_{f^{\ast}g}(D),
\]
where ${\mathrm{Area}}_{f^{\ast}g}(D)$ denotes the area of $D$ with respect to the pull-back metric $f^{\ast}g$. The last inequality follows by comparing the invariant form $d\alpha$ to the volume form on $D$ induced by the invariant Riemannian metric $g$ on $L_p$. Here $K_0=K_0(\alpha)>0$ is a constant that only depends on $\alpha$.

However, by choosing $D$ to be (arbitrarily close to) a filling disc of minimal area for $\gamma$, we deduce that
\[
\left|\int_{\gamma} \alpha \right| = \left| \int_{D} d\alpha\right| \leqslant K_0 \cdot \mathrm{Area}_{L_p}(\gamma).
\]
Since the area of  loops of length $L(\gamma)\leqslant M$ in $L_p$ is uniformly bounded by a constant, it follows that there is $K=K(M,\alpha)>0$ such that
\[
\left| \int_{\gamma} \alpha\right| \leqslant K_0 \mathrm{Area}_{L_p}(\gamma)\leqslant K
\]
for all such loops.
\end{proof}

\begin{proof}[Proof of Proposition \ref{prop:uniformboundonsmallloops}]
 Observe that with respect to the coordinates $(u_1,\dots,u_p)$ we have
 \[
 \mathrm{dist}_{L_p}\left(\left\{u_1\right\}\times \RR^{p-1},\left\{u'_1\right\}\times \RR^{p-1}\right) >0 \mbox{\hspace{.3cm} for $u_1\neq u_1'$}
 \]
 and 
 \[
 \mathrm{dist}_{L_p}\left(\left\{u_1\right\}\times \RR^{p-1},\left\{u'_1\right\}\times \RR^{p-1}\right) \to \infty \mbox{\hspace{.3cm} for $u_1'\to \pm \infty$}.
 \]
In particular, there is a constant $K_0=K_0(M)>0$ such that the image of any loop $\gamma$ with $L(\gamma)\leqslant M$ which intersects the hypersurface $\left\{ 0\right\} \times \RR^{p-1}$ non-trivially is contained in $\left[-K_0,K_0\right]\times \RR^{p-1}$.

We distinguish the cases $\gamma(\left[0,1\right])\cap \left(\left\{0\right\} \times \RR^{p-1}\right) =\emptyset$ and $\gamma(\left[0,1\right])\cap \left(\left\{0\right\} \times \RR^{p-1}\right) \neq \emptyset$, starting with the former. In this case we observe that $\beta_0$ equals either the form $\sum_{k=0}^{p-2} (-1)^k\frac{u_1^{k+1}}{(k+1)!} du_{p-k}$ in all points of $\gamma(\left[0,1\right])$ or its negative. Both forms extend to global forms on $L_p$ with invariant differential $d\beta$ (respectively $-d\beta$). Thus, Lemma \ref{lem:uniformboundonloops} implies that there is a constant $K_1=K_1(M)>0$ such that $\left| \int_{\gamma} \beta_0\right| \leqslant K_1$ for all loops $\gamma$ satisfying the hypotheses.

Now assume that $\gamma(\left[0,1\right])\cap \left(\left\{0\right\} \times \RR^{p-1}\right) \neq \emptyset$. Then $\gamma(\left[0,1\right])\subset \left[-K_0,K_0\right]\times \RR^{p-1}$ In particular, for $\gamma=\left(\gamma_1,\dots,\gamma_p\right) : \left[0,1\right]\to \RR^p$ we have that $|\gamma_1(t)|$ is uniformly bounded by $K_2:= \mathrm{max}\left\{1, |K_0|\right\}$.

Assume now that $\gamma(t)$ is reparametrized by length, i.e. $\gamma:\left[0,L(\gamma)\right] \to L_p$ with $|| \dot{\gamma} ||_g\equiv 1$. In view of our choice of metric $g$ and Lemma \ref{lem:inv-vec-fields}, this is equivalent to saying that we have functions $\lambda_1,\dots,\lambda_p:\left[0,L(\gamma)\right]\to \RR$ such that $\sum_{i=1}^p \lambda_i^2 \equiv 1$ and 
\[
\dot\gamma(t) = \sum_{i=1}^p \lambda_i(t) \cdot S_{\gamma(t),\ast} \del_{u_i,0} = \lambda_1(t)\del_{u_1,\gamma(t)} + \sum_{i=2}^p\lambda_i(t) \sum_{j=0}^{p-i} \frac{(\gamma_1(t))^j}{j!} \del_{u_{i+j},\gamma(t)}.
\]
In particular, we deduce that
\[
\beta_0(\dot\gamma(t)) = \mathrm{\sgn}\left(\gamma_1(t)\right)\cdot \sum_{k=0}^{p-2}\sum_{i=2}^p\sum_{j=0}^{p-i} (-1)^k\frac{(\gamma_1(t))^{j+k+1}}{j!(k+1)!} \cdot \lambda_i(t) \cdot \delta_{p-k,i+j},
\]
where $\delta_{p-k,i+j}$ denotes the Kronecker function.

Since $|\gamma_1(t)|\leqslant K_2$ and $|\lambda_i(t)|\leqslant 1$, it follows that $|\beta_0(\dot\gamma(t))|\leqslant p^3 \cdot K_2^{2p}$. Hence, we obtain
\begin{align*}
 \left|\int_{\gamma} \beta_0\right| &= \left| \int_0^{L(\gamma)} \beta_0(\dot\gamma(t)) dt\right| \\
 &\leqslant \int_0^{L(\gamma)} |\beta_0(\dot\gamma(t))| dt \leqslant L(\gamma) \cdot p^3 K_2^{2p} \leqslant M \cdot p^3 \cdot K_2^{2p}.
\end{align*}
Choosing $K(M):=\mathrm{max}\left\{K_1,M \cdot p^3 \cdot K_2^{2p}\right\}$ thus completes the proof.

\end{proof}

\begin{proof}[{Proof of Theorem \ref{thm:lower-bounds-form}}]
Consider the null-homotopic word $w_{p+1,n}$ from Section \ref{sec:lbf-proof-main-result} in the first factor $L_{p+1}\leqslant G_{p+1,p}$. Its image in $L_{p}$ under the projection $G_{p+1,p}\to L_{p}\times L_{p-1}\to L_{p}$ is the null-homotopic word $w_{p+1,n}$ in $L_{p}$. Proposition \ref{prop:integral-betaC--1-new}, Proposition \ref{prop:uniformboundonsmallloops} and Proposition \ref{prop:discreteStokes} imply that 
$${\mathrm{Area}}_{G_{p+1,q+1}}(w_n)\geqslant {\mathrm{Area}}_{L_{p}}(w_n) \gtrsim_{p,M} n^{p},$$
 where we choose $M>0$ big enough such that $L(\overline{r})<M$ for all word-loops $\overline{r}$ associated to relations $r\in R$ for the compact presentation $\left\langle S\mid R\right\rangle:=\mathcal{P}(L_{p})$ of $L_{p}$. This completes the proof.
\end{proof}

\begin{remark}
\label{rmk:increasing-Dehn-for-descending-q}
 Theorem \ref{thm:lower-bounds-form} shows that for $2\leqslant q \leqslant p$ we have $n^{p-1}\preccurlyeq \delta_{G_{p,q}}(n) \preccurlyeq n^{p}$. Moreover, following the same arguments as in the first part of the proof of Theorem \ref{thm:Upperbound} in \S \ref{subsec:PfMainThm}, we can actually show that $\delta_{G_{p,q}}(n)\lesssim \delta_{G_{p,q'}}(n)$ for $q'<q$, by reducing to null-homotopic words in $x_1$ and $x_2$.  On the other hand we currently only know the precise Dehn function for $q\in \left\{2,p-1, p\right\}$. Curiously for $q=2$ the Dehn function is $n^p$, since $G_{p,2}=L_p \times \RR$, while for $q=p-1,p$ it is $n^{p-1}$ by our results. This naturally raises the question if the Dehn functions for increasing $q$ interpolate between $n^p$ and $n^{p-1}$ or if the case $q=2$ is just a ``borderline'' phenomenon.
\end{remark}

\section{Application to the large-scale geometry of nilpotent groups}
\label{sec:SBE}

In this section we will study sublinear bilipschitz equivalences (SBE)
in the context of our examples. In particular, we will prove Theorem
\ref{thmIntro:SublinearBIlip} by combining Main Theorem
\ref{thm:Upperbound} from \S \ref{sec:upper-bound-Dehn} with results on
SBEs.

\subsection{Sublinear bilipschitz equivalence between nilpotent groups}

Sublinear bilipschitz equivalences were defined in the introduction.
We refer the reader to Cornulier's paper dedicated to the notion
{\cite{cornulier2017sublinear}} for a more extensive treatment of the
subject.
For our purposes it will be sufficient to consider $O(r^e)$-sublinear
bilipschitz equivalences, that is SBEs for which the function $v$ in
Definition \ref{def:SBE} takes the form $v(t) = t^e$ with $e \in [0,1)$.

We will need the following result from \cite{cornulier2017sublinear},
which generalizes a classical exercise on quasiisometries corresponding
to the special case $e=0$.

\begin{lemma}[Cornulier, {\cite[Proposition
2.4]{cornulier2017sublinear}}]
\label{lem:Cor-SBE-inverses}
Let $Y$ and $Y'$ be pointed metric spaces (e.g. groups with a left-invariant distance, based at the neutral element); denote $\vert \cdot \vert$
the distance to the basepoint in both spaces.
Let $f : Y \to Y'$ be a $O(r^e)$-sublinear bilipschitz equivalence.
Then there exists $g : Y' \to Y$ such that for $y \in Y$ and $y' \in
Y'$,
$d(f \circ g (y'),y') = O(\vert y' \vert^e)$ and $d(g \circ f (y),y) =
O(\vert y \vert^e)$.
\end{lemma}

Lemma \ref{lem:Cor-SBE-inverses} is actually an explicit version of
Cornulier's original statement that $O(r^e)$-SBEs are isomorphisms in
the $O(r^e)$-category, which he defines in the obvious way
\cite{cornulier2017sublinear}. The asymptotic cone functors with fixed
basepoints are well-defined on this category (\cite{CornulierCones11},
\cite{cornulier2017sublinear}) and, in analogy to the case of
quasi-isometries, SBEs induce bilipschitz homeomorphisms between
asymptotic cones.

\begin{proposition}[Cornulier]
\label{prop:asymptotic-cones}

Let $Y$ and $Y'$ be homogeneous metric spaces. If there exists a
$O(r^e)$-SBE $Y \to Y'$, then for any nonprincipal ultrafilter $\omega$
and sequence of scaling factors $(\sigma_j)$ the metric spaces
$\operatorname{Cone}_\omega(Y,\sigma_j)$ and
$\operatorname{Cone}_\omega(Y',\sigma_j)$ are bilipschitz homeomorphic.
\end{proposition}

In particular, if a homogeneous space $Y$ is $O(r^e)$-SBE to a
self-similar homogeneous space $Y'$, then the latter is the asymptotic
cone of $Y$ up to bilipschitz homeomorphism. Not all simply connected
nilpotent Lie groups admit left-invariant self-similar proper geodesic
metrics, only the Carnot gradable ones do.

\begin{theorem}[Cornulier]
\label{th:cornuliers-pansu-theorem}
Let $G$ be a nilpotent simply connected Lie group. Let $\mathfrak{g} =
\operatorname{Lie}(G)$. Let $\mathsf{gr}(G)$ be the associated
Carnot graded Lie group. Equip $G$ and $\mathsf{gr}(G)$ with
geodesically adapted distances.
Then there exists a computable $e_\mathfrak{g} \in [0,1)$ only depending
on $\mathfrak{g}$ such that $G$ and $\mathsf{gr}(G)$ are
$O(r^{e_\mathfrak{g}})$-SBE.
\end{theorem}

\begin{remark}
As explained in \cite[Section 6]{cornulier2017sublinear}, a version of
Theorem \ref{th:cornuliers-pansu-theorem} where $e_\mathfrak{g} = 1 -
1/c$ if $G$ is $c$-step nilpotent can be derived by combining two
results from the 1970s, namely an estimate from Guivarc'h's proof of the
Bass-Guivarc'h dimension formula and Goodman's observation that the laws
of $G$ and $\mathsf{gr}(G)$ differ sublinearly on the large-scale when written as polynomial group laws on
$\mathsf{gr}(\mathfrak{g})$ \cite{Goodman77}.  Cornulier's input in \cite{cornulier2017sublinear} is
in the improvement of $e_\mathfrak{g}$ in terms of finer invariants of
the structure of $\mathfrak{g}$.
We will give low-dimensional examples in Table \ref{table:dim6-step45}.
\end{remark}

\begin{corollary}[Pansu and Cornulier,
\cite{PanCBN,PansuCCqi,CornulierCones11}]
\label{cor:weak-Pansu}
Let $G$ and $G'$ be two simply connected nilpotent Lie groups. The
following are equivalent:
\begin{itemize}
\item[(i)]
There exists a nonprincipal ultrafilter $\omega$ on $\mathbf N$ and a
sequence of normalization factors $(\sigma_j)_{j \in \mathbf N}$ such
that the metric spaces $\operatorname{Cone}_\omega(G,\sigma_j)$ and
$\operatorname{Cone}_\omega(G',\sigma_j)$  are bilipschitz equivalent.
\item[(ii)] The groups $\mathsf{gr}(G)$ and $\mathsf{gr}(G')$ are
isomorphic.
\item[(iii)]There exists $e \in [0,1)$ such that $G$ and $G'$ are
$O(r^e)$-sublinear bilipschitz equivalent.
\end{itemize}
\end{corollary}
\begin{proof}[Proof of Corollary \ref{cor:weak-Pansu}] Assuming (i), we
deduce (ii) from Theorem \ref{thm:Pansu}. (ii) implies (iii) by Theorem
\ref{th:cornuliers-pansu-theorem}. Finally (iii) implies (i) by
Proposition \ref{prop:asymptotic-cones}.
\end{proof}

\begin{remark}
\label{rem:weak-Pansu-to-weak-Breuillard}
Corollary \ref{cor:weak-Pansu} holds for locally compact groups with
polynomial growth $G$, where the construction of $\mathsf{gr}(G)$
requires additional steps. In particular, one first has to pass to a
nilshadow of the Lie shadow of $G$, see Breuillard \cite{Breuillard}.
\end{remark}

Corollary \ref{cor:weak-Pansu} leaves the problem of evaluating the
range of $e$ such that a given pair of groups with identical asymptotic
cones can be $O(r^e)$-equivalent. The question was raised by Cornulier
\cite[Question 6.20]{cornulier2017sublinear}.
For the pair $(L_p\times L_{p-2},G_{p,p-1})$, our Theorem
\ref{thmIntro:SublinearBIlip} states that one must have $e\geq
1/(2p)$, which for the first case of interest $p=4$ implies $e\geq
1/8$. These are the first examples for which a positive lower bound is
known. We will prove Theorem \ref{thmIntro:SublinearBIlip} at the end of
this section.

\subsection{Large-scale fillings and sublinear bilipschitz equivalence}

Our main tool for proving Theorem \ref{thmIntro:SublinearBIlip} is the
following technical lemma.

\begin{lemma}
\label{lem:filling-transfer-SBE}
Let $G$ and $G'$ be two locally compact compactly presented groups
admitting filling pairs $(n^d,n^s)$ and
$(n^{d'},n^{s'})$ respectively.
Let $e\in[0,1)$. If there exists an  $O(r^e)$-SBE between $G$ and $G'$,
then
\[
\left(n^{(1+e)d'+e(1+e)s'd}+n^{(1+e)^2+e(d-1)},n^{(1+e)s'}+n^{e(1+e)s's}\right)
\] 
is a filling pair for $G$.
\end{lemma}

Before starting the proof we fix some conventions and notations.
We will fix Cayley graphs of $G$ and $G'$, and a {\em loop in $G$} will be a loop in the Cayley graph of $G$ (not necessarily based at $1$). When we speak of maps to $G$ (resp. $G'$) we will from now on mean maps to their respective Cayley graphs. 

A combinatorial disk $\Delta:=(X,\phi)$ filling a loop $\gamma$ is defined by the following data: a CW-complex structure $X$ on the closed 2-dimensional unit ball with $N$ 2-cells $\Delta_1,\ldots, \Delta_N$ and injective attaching maps in all dimensions, and a continuous map $\phi:X^{(1)}\to G$ from the $1$-skeleton of $X$ to the Cayley graph of $G$, such that $\gamma$ parametrizes $\phi|_{\partial \Delta}$ and $\phi$ maps vertices to vertices. We will denote $\gamma_i:=\phi|_{\partial \Delta_i}$ the boundary loops of the 2-cells and say that $\Delta$ is a filling of $\gamma$ by loops $\gamma_1,\dots, \gamma_N$. 

Retaining the above notation, one can check that $G$ admits $(n^d,n^s)$ as
a filling pair if and only if there is a constant $M_0>0$ such that
every loop  of length $\leqslant n$ based at the identity in $G$ admits a filling by a combinatorial
disk such that $N\lesssim n^d$, $\phi(X^{(1)})$ is contained in a
ball of diameter $\lesssim n^s$ around the origin and $\gamma_i$ parametrizes a loop of length $\leqslant M_0$. This is straight-forward and well-known for Dehn functions and generalises readily to filling pairs.

\begin{proof}
By Lemma \ref{lem:Cor-SBE-inverses} there is a continuous map $\widehat{F}:G' \to G$ such that $\widehat{F}\circ F$ is $O(r^e)$-close to the identity.
Let $\gamma:S^1\to G$ be any loop of length $n$ in $G$ based at the identity. Then $\gamma':=F \circ \gamma$ defines a loop $\gamma'$ of length $\lesssim n^{1+e}$ in $G'$.
Fill $\gamma'$ with a combinatorial disk $\Delta'=(X,\phi)$ composed of $\lesssim n^{(1+e)d'}$ loops of bounded length and area. Note that $\phi(X^{(1)})$ is contained in a ball of diameter $\lesssim n^{(1+e)s'}$ around the origin.

Composing $\Delta'$ with $\widehat{F}$ yields a combinatorial disk $\Delta'':=(X,\widehat{F}\circ \phi)$ which is composed of $\lesssim n^{(1+e)d'}$ loops of length $\lesssim n^{e(1+e) s'}$. Note also that $\widehat{F}(\phi(X^{(1)}))$ is still contained in a ball of diameter $\lesssim n^{(1+e)s'}$.
% Note that in diameter bounds for filling discs all diameters are at least linear. Thus SBEs do not increase them (they only have an error term for sublinear distances, which e.g. play a role in the length of curves)
The boundary loop $\gamma''$ of $\Delta''$ has length $\lesssim n^{(1+e)^2}$. 
We can thus choose a set of $r \lesssim n^{1+e+e^2}$ points $0=t_1 < t_2< \dots < t_r=1$ on $S^1$ such that $L(\gamma''|_{\left[t_i,t_{i+1}\right]}) \lesssim n^e$. Note that we may further assume that $L(\gamma|_{\left[t_i,t_{i+1}\right]})\leqslant 1$ (after possibly adding $n$ more points).

We define loops $\gamma_i$ of length $\lesssim n^e$ by concatenating $\gamma|_{\left[t_i,t_{i+1}\right]}$, a geodesic segment $[\gamma(t_{i+1}),\gamma''(t_i)]$,  $\gamma''|_{[t_i,t_{i+1}]}$  and a geodesic segment $[\gamma''(t_i),\gamma(t_i)]$; for the bound on the length we use that $\widehat{F}\circ F$ is $O(r^e)$-close to the identity.

Attaching the loops $\gamma_i$ to the combinatorial disk $\Delta''$ defines a combinatorial disk $\Delta'''$ with boundary loop $\gamma$. By construction, $\Delta'''$ is composed of $n^{(1+e)d'}$ loops of length $\lesssim n^{e(1+e)s'}$ at distance $\lesssim n^{(1+e)s'}$ from the origin, as well as $n^{1+e+e^2}$ loops of length $\lesssim n^e$ at distance $\lesssim n^{1+e}$ from the origin.
Using that $(n^{d},n^{s})$ is a filling pair for $G$ to fill these loops
yields the filling pair $$\left(n^{(1+e)d'}\cdot n^{e(1+e)s'd}+n^{1+e+e^2}\cdot n^{ed},n^{1+e}+n^{es}+n^{(1+e)s'}+n^{e(1+e)s's}\right)$$ for $G$. Since $s,s'\geq 1$, we obtain the filling pair
\[
\left(n^{(1+e)d'+e(1+e)s'd}+n^{(1+e)^2+e(d-1)},n^{(1+e)s'}+n^{e(1+e)s's}\right)
\]
for $G$.
\end{proof}

\begin{proof}[Proof of Theorem \ref{thmIntro:SublinearBIlip}]
We apply Lemma \ref{lem:filling-transfer-SBE} to the pair $G=L_p\times
L_{p-2}$ which admits a $(n^{p},n)$ filling pair by \cite[Theorem
2.3]{PittetIsopNilHom}, and $G'=G_{p,p-1}$ which admits a  $(n^{p-1},n)$
filling pair by Theorem \ref{thm:Upperbound}.
We deduce that the Dehn function of $G$ has to satisfy $n^p\lesssim
n^{(1+e)(p-1)+e(1+e)p}+n^{(1+e)^2+e(p-1)}$. This yields the inequality
\[
 p\leq \max\left\{(1+e)(p(1+e)-1), (1+e)^2+e(p-1)\right\}.
\]
A straight-forward calculation shows that for $e= \frac{1}{2p}$ this inequality is not satisfied. Since both of the terms on the right are increasing functions in $e\in \left[0,1\right)$ the inequality cannot be satisfied for any $e\in\left[0, \frac{1}{2p}\right]$, yielding the desired lower bound.
\end{proof}

\section{Overview in low dimensions}
\label{sec:low-dimension}

In this section we provide a complete overview of the real nilpotent Lie algebras of dimension less or equal to $6$ together with the best estimates that we can find on their Dehn functions. By the Dehn function (resp.\ the centralized Dehn function) of a Lie algebra $\mathfrak{g}$, denoted $\delta_{\mathfrak{g}}(n)$ resp.\ $\delta_{\mathfrak{g}}^{\mathrm{cent}}(n)$, we mean the Dehn function (resp.\ the centralized Dehn function) of its associated simply connected nilpotent Lie group $G$ (i.e. $\operatorname{Lie}(G) = \mathfrak{g}$).
A complete classification of real nilpotent Lie algebras of dimension $\leqslant 6$ was given by de Graaf \cite{deGraafclass}.  
We will use his notation\footnote{Note that de Graaf's precise notation is $L_{d,j}$ rather than $\mathscr{L}_{d,j}$.} $\mathscr{L}_{d,j}$, where $d$ is the dimension and $j$ is an integer. Note that in dimension $\leqslant 5$ all Dehn functions were computed by Pittet \cite{PittetIsopNil}. We still list them for the sake of completeness.

We list the nilpotent Lie algebras together with their structure, their homogeneous dimension and the best known estimates on their Dehn functions in Tables \ref{table:dim-less-5}--\ref{table:dim6-step45}. Table \ref{table:dim-less-5} contains all nilpotent Lie algebras of dimension at most $5$ and Tables \ref{table:dim6-step2}--\ref{table:dim6-step45} those of dimension $6$ ordered by their nilpotency classes and homogeneous dimension $\operatorname{hdim}(\mathfrak{g}):= \sum_{s \geqslant 1} s \, \dim \gamma_s \mathfrak{g} / \gamma_{s+1} \mathfrak{g}$. The latter is a quasi-isometry invariant, as it coincides with the exponent of growth of the corresponding group \cite[Thm II.1]{Guivarch73}.

{\em We will now give some explanations regarding the contents of our
tables.} In dimension $6$ we do not list decomposable Lie algebras $\mathfrak g$ (i.e. Lie algebras that split as a direct product of lower-dimensional ones) except if their class of Lie algebras with the same Carnot graded algebra consists of more than one element; this is to keep our tables as compact as possible. 
More generally, we group Lie algebras by their associated Carnot graded algebras, starting with the unique one that is Carnot. 
The nonzero brackets defining the structure of the respective Lie algebras are provided in an abbreviated form: for instance the notation $12=34=5$ means that $[x_1, x_2] = [x_3,x_4]=x_5$ and defines the five-dimensional Heisenberg algebra.

In most cases our estimates on $\delta_{\mathfrak{g}}(n)$ are derived as follows:
\begin{enumerate}
\item
\label{item:upper-bound-overview}
The upper bound is given by the universal upper bound of $n^{c+1}$ on the Dehn function of a nilpotent group of nilpotency class $c$ \cite{GerstenRileyHolt}.
\item
\label{item:lower-bound-overview}
The lower bound is given by the centralised Dehn function $\delta_\mathfrak{g}^{\mathrm{cent}}(n)$. It is obtained by providing a suitable central extension of maximal distortion. 
\end{enumerate}

For \eqref{item:lower-bound-overview} we provide a maximally distorted central extension in abbreviated form in the table. Let us illustrate this via the example of $\mathscr{L}_{5,5}$. In this case we claim that a central extension of maximal distortion is given by $z = 14=35$. This is short-hand for the extension $\mathbf R z\to\widetilde{\mathfrak{g}}\to \mathfrak{g} $, where $z$ satisfies $z =[\sigma(x_1), \sigma(x_4)]=[\sigma(x_3),\sigma(x_5)]$ for any section $\sigma:\mathfrak{g}\to\widetilde{\mathfrak{g}}$. 
Verifying the existence of this extension is easy via the well-known identification of central extensions with second cohomology classes given by Proposition \ref{prop:Centralrcentralextensions}. 
Indeed, in the case of $\mathscr{L}_{5,5}$ the extension $z=14=35$ corresponds to the $2$-form $\omega :=\xi_1\wedge \xi_4 + \xi_3\wedge \xi_5$, where $\xi_1,\dots,\xi_5$ is a dual basis of the basis $x_1,\dots,x_5$. 
We readily deduce from the structure of $\mathscr{L}_{5,5}$ that $d \xi_1=d\xi_2=d\xi_5 =0$, $d\xi_3=-\xi_1\wedge\xi_2$ and $d\xi_4=-\xi_1\wedge\xi_3-\xi_2\wedge\xi_5$. Thus, we obtain that $d\omega = 0$ and that $\omega$ defines a non-trivial cohomology class.

For the cases where there are either better estimates on the Dehn function than one can obtain from the above method or where estimates are well-known we provide a reference to the literature or previous sections. Finally, note that the Dehn functions of the decomposable algebras that we did not list can easily be deduced from Lemma \ref{lem:Dehn-functions-of-direct-products} and the Dehn functions of their factors.

\begin{remark}
We indicate all relations via central extensions between nilpotent Lie algebras $\mathfrak{g}$ of dimension $\leqslant 5$ in Figure \ref{fig:genealogy-nilpotent-algebras}; if $\mathfrak{g}$ is 5-dimensional we also provide at least one $6$-dimensional central extension.
\end{remark}

Note that there are a total of $5$ cases for which we were not able to determine the precise Dehn functions via any method. In particular in these cases the bounds from \eqref{item:upper-bound-overview} and \eqref{item:lower-bound-overview} do not match. We summarize the state of the art for these cases.
\begin{lemma}
\label{lem:six-dimensional-algebras-where-dcent-is-n^c}
Let $ \mathfrak{g} \in \mathcal{L}= \lbrace \mathscr{L}_{6,14}, \mathscr{L}_{6,16}, \mathscr{L}_{6,19}(\pm 1), \mathscr{L}_{6,20} \rbrace$ and let $c$ be its nilpotency class. Then $\mathfrak{g}$ admits a $c$-central extension, but no $(c+1)$-central extension. In particular, the central and regular Dehn functions of $\mathfrak{g}$ satisfy the asymptotic inequalities
\[
 \delta^{\mathrm{cent}}_{\mathfrak g}(n) \asymp n^c \preccurlyeq \delta_{\mathfrak g}(n)\preccurlyeq n^{c+1}.
\]
\end{lemma}

\begin{proof}
For the existence of a  $c$-central extensions we refer to the concrete $c$-central extensions indicated in the tables with the arguments being the same as the ones given above.

The proof of the non-existence of a $(c+1)$-central extension is by performing computations similar to the ones in \S \ref{subsec:central-extensions-central-products}. Note that for the Carnot case the computations are more elegant than for the non-Carnot case, since the differential preserves the grading. The only non-Carnot Lie algebra in $\mathcal{L}$ is $\mathscr{L}_{6,14}$; the corresponding computation is more cumbersome but no harder.

Rather than giving details for all cases, we will restrict to the concrete example of the Carnot Lie algebra $\mathscr{L}_{6,20}$ and leave the remainder of the computations as an exercise to the reader. By definition $\mathscr{L}_{6,20}$ is $3$-step nilpotent.

To show that there is no $4$-central extension it suffices to prove that $H^2(\mathfrak{g}, \mathbf R)^{4} = 0$. 
Recall that $\mathscr{L}_{6,20}$ is defined by the generating set $\left\{x_1,\dots,x_6\right\}$ and the following nonzero relations
\[
[x_1, x_2] = x_4,~ [x_1, x_3] = x_5,~ [x_1, x_5] = [x_2, x_4] = x_6.
\]
We denote its dual basis $\left\{\xi_1, \ldots , \xi_6\right\}$ and, as before, we use the notation $\xi_{i,j}=\xi_i\wedge\xi_j$ etc.

The first quotient of the lower central series of $\mathscr{L}_{6,20}$ is generated by $\left\{ x_1, x_2, x_3 \right\}$. Thus, we have 
\[ 
(\bigwedge\nolimits^{\!2}  \mathscr L_{6,20}^\star)_4 = \operatorname{span}\left\{\xi_{1,6}, \xi_{2,6}, \xi_{3,6}, \xi_{4,5}\right\}
\]
in the associated grading on $\bigwedge\nolimits^{\!2}  \mathscr L_{6,20}^\star$.

It follows that it suffices to check that any cocycle of the form $\omega = a_{1,6} \xi_{1,6}+a_{2,6} \xi_{2,6} + a_{3,6} \xi_{3,6} + a_{4,5} 
\xi_{4,5}$ is trivial. We compute the differential
\begin{align}
d\omega & = - a_{1,6} \xi_1 \wedge d \xi_6 - a_{2,6} \xi_2 \wedge d \xi_6 - a_{3,6} \xi_3 \wedge d \xi_6 + a_{4,5} d \xi_4 \wedge \xi_5 - a_{4,5} \xi_4 \wedge d\xi_5 \notag \\
& = a_{1,6} \xi_{1,2,4} + a_{2,6} \xi_{2,1,5} + a_{3,6} (\xi_{3,1,5} + \xi_{3,2,4}) + a_{4,5} (\xi_{4,1,3} - \xi_{1,2,5}) \notag \\
& = a_{1,6} \xi_{1,2,4} + (a_{2,6} + a_{4,5}) \xi_{2,1,5} + a_{3,6} (\xi_{3,1,5} + \xi_{3,2,4}) + a_{4,5} \xi_{4,1,3}, \label{eq:differential-computation}
\end{align}
which is indeed nonzero unless $a_{1,6} = a_{2,6} = a_{3,6} = a_{4,5} = 0$. This shows that $\mathscr{L}_{6,20}$ does not admit a $4$-central extension.
\end{proof}

Finally, in the last column of Table \ref{table:dim6-step45} we list the best known exponent $e_{\mathfrak{g}}$ such that $G$ and $\operatorname{gr}(G)$ are $O(r^{e_{\mathfrak{g}}})$-SBE (see \S \ref{sec:SBE} for details). We do not list $e_\mathfrak{g}$ in Tables \ref{table:dim-less-5}, \ref{table:dim6-step2} and \ref{table:dim6-step3}, since it is always $0$ if $G$ is Carnot gradable and $1-c^{-1}$ if not, where $c$ is the nilpotency step of $\mathfrak{g}$. For the computation of $e_{\mathfrak{g}}$ when $\mathfrak{g}= \mathscr{L}_{6,d}$, $d \in \lbrace 12,17 \rbrace$ see \cite[6C6]{cornulier2017sublinear}.

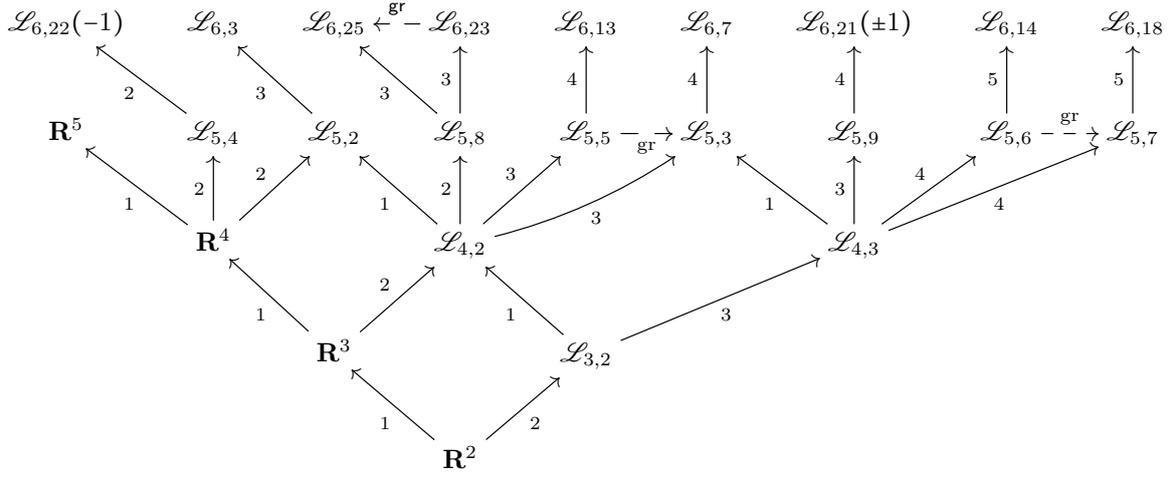
\begin{figure}[t]
\[
\xymatrix@C=0.6cm{
\mathscr{L}_{6,22}(-1)
& \mathscr{L}_{6,3} &  \mathscr{L}_{6,25} & \mathscr{L}_{6,23} \ar@{-->}[l]_{\mathsf{gr}} & \mathscr{L}_{6,13} & \mathscr{L}_{6,7} & \mathscr{L}_{6,21}(\pm 1) & \mathscr{L}_{6,14} & \mathscr{L}_{6,18} \\
\mathbf R^5  & \mathscr{L}_{5,4} \ar[lu]^2 & \mathscr{L}_{5,2} \ar[lu]^3 & \mathscr{L}_{5,8} \ar[lu]^3 \ar[u]^3 & \mathscr{L}_{5,5} \ar@{-->}[r]_{\mathrm{gr}} \ar[u]^4 & \mathscr{L}_{5,3} \ar[u]^4  & \mathscr{L}_{5,9} \ar[u]^4 & \mathscr{L}_{5,6} \ar@{-->}[r]^{\mathrm{gr}} \ar[u]^5 & \mathscr{L}_{5,7} \ar[u]^5  \\
& \mathbf R^4 \ar[ur]^2 \ar[ul]^1 \ar[u]^2 & & \mathscr{L}_{4,2} \ar[ur]^3 \ar@/_/[urr]_3 \ar[ul]^1 \ar[u]^2 & & & \mathscr{L}_{4,3}  \ar[ul]^1 \ar[u]^3 \ar[ur]^4 \ar[rru]_4 & & \\
& & \mathbf R^3 \ar[ul]^1 \ar[ur]^2 & & \mathscr{L}_{3,2} \ar[ul]^1 \ar[urr]_3 & & & \\
& & & \mathbf R^2 \ar[ul]^1 \ar[ur]_2 & & & & }
\]
\caption{
Nilpotent Lie algebras of dimension $\leqslant 5$ and how they are related.\newline The notation $\mathfrak{g} \stackrel{r}{\longrightarrow}\mathfrak{h}$ means that $\mathfrak{h}$ is a $r$-central extension of $\mathfrak{g}$ by $\mathbf R$ (see the tables below for the explicit extensions) and $\mathfrak{g} \stackrel{\mathrm{gr}}{\dashrightarrow} \mathfrak{h}$ means that $\mathfrak{h}=\operatorname{gr}(\mathfrak{g})$.
}
\label{fig:genealogy-nilpotent-algebras}
\end{figure}

\begin{table}[H]
\begin{center}
\begin{tabular}{|c|c|c|c|c|c|c|}
\hline
Algebra & Structure & step & hdim & $\delta(n)$  \\
\hline
$\mathscr{L}_{3,2}= \mathfrak{l}_3 = \mathfrak{heis}_3$  & $12=3$ & 2 & 4 &  $n^3$ \\
\hline
\hline
$\mathscr{L}_{4,2}= \mathscr{L}_{3,2} \times \mathbf R$ & $12=3$ & 2 & 5 &  $n^3 $ \\
\hline
$\mathscr{L}_{4,3} = \mathfrak{l}_4$ & $12=3, 13=4 $ & 3 & 7 &  $n^4$ \\
\hline
\hline
$ \mathscr{L}_{5,2}=\mathscr{L}_{3,2} \times \mathbf R^2$ & $12 = 3$ & 2 & 6 &   $n^3, z = 13$ \\
\hline
$\mathscr{L}_{5,4} = \mathfrak{heis}_5$ & $12=34=5$ & 2 & 6 &   $n^2$ \cite{Allcock,OlsSapCombDehn} \\
\hline
$\mathscr{L}_{5,8}$ & $12 = 3, \, 14 = 5$ & 2 & 7 &  $n^3, z = 15$ \\
\hline
$ \mathscr{L}_{5,3}=\mathscr{L}_{4,3} \times \mathbf R $ & $12 = 3, \, 13 = 4$ & \multirow{2}{*}{3} & \multirow{2}{*}{8}  &   $n^4$, $z=14$ \\
$\mathscr{L}_{5,5}$ & $12 = 3, \, 13 = 25 = 4$ &  &  &   $n^4, z = 14=35$ \\
\hline
$\mathscr{L}_{5,9}$ & $12 = 3, \, 13 = 4, \, 23 = 5$ & 3 & 10 &  $n^4, z = 15=24$  \\
\hline
$\mathscr{L}_{5,7} =\mathfrak{l}_5$ & $12 = 3, \, 13 = 4, \, 14 = 5$ & \multirow{2}{*}{4} & \multirow{2}{*}{11} &  $n^5$, $z=15$  \\
$\mathscr{L}_{5,6}$ & $12 = 3, \, 13 = 4, \, 14 = 23 = 5$ &  & &   $n^5$, $z=25=43$ \\
\hline
\end{tabular}
\end{center}
\caption{Nonabelian nilpotent Lie algebras of dim $\leqslant 5$ and their Dehn functions. }
\label{table:dim-less-5}
\end{table}

\begin{table}[H]
\begin{center}
\begin{tabular}{|c|c|c|c|c|}
\hline
Name & Structure &  hdim & $\delta(n)$ \\
\hline
$\mathscr{L}_{6,22}(-1) = \mathfrak{heis}^{\mathbf C}_{3 \mid \mathbf R}$ & 13 = 24 = 5, 14 = 32 = 6 & 8 &   $n^3, z = 16 = 52$ \\
\hline 
$\mathscr{L}_{6,22}(0) $ & 13 = 24 = 5,\, 14 = 6  & 8 & $n^3, z = 16$ \\
\hline
$\mathscr{L}_{6,26}$ (free rank. $3$)& 12 = 4, \, 23 = 5, \, 31 = 6  & 9  & $n^3$ \cite[Theorem 7]{BaumslagMillerShort}    \\
\hline
\end{tabular}
\end{center}
\vspace{.2cm}
\caption{Indecomposable 2-step nilpotent Lie algebras of dimension $6$ and their Dehn functions.}
\label{table:dim6-step2}
\end{table}

\begin{table}[H]
\begin{center}
\begin{tabular}{|c|c|c|c|c|}
\hline
Name & Structure &  hdim & $\delta(n)$ \\
\hline
$\mathscr{L}_{6,20}$ & $12 = 4, 13 = 5, 15 = 24 = 6$  & 10 & $n^3 \preccurlyeq \delta(n) \preccurlyeq n^4$, $z=14$ \\
\hline
$\mathscr{L}_{6,19}(0)$ & $12 = 4, 13 = 5, 24 = 6$  & 10 & $n^4, z = 26$ \\
\hline
$\mathscr{L}_{6,19}(1)$ & $12 = 4, 13 = 5, 35 = 24 = 6$  & 10 & $n^3 \preccurlyeq \delta(n) \preccurlyeq n^4$, $z=15$ \\
\hline
$\mathscr{L}_{6,19}(-1)$ & $12 = 4, 13 = 5, 53= 24 = 6$  & 10 & $n^3 \preccurlyeq \delta(n) \preccurlyeq n^4$, $z=15$   \\
\hline
$\mathscr{L}_{6,3} = \mathscr{L}_{4,3} \times \mathbf R^2 $ & $12=3, 13=4$ & \multirow{3}{*}{9} & $n^4$ (product) \\
$\mathscr{L}_{6,5}=\mathscr{L}_{5,5} \times \mathbf R$ & $12=3,\, 13=25=4$  & & $n^4$ (product) \\
$\mathscr{L}_{6,10} = \mathfrak{g}_{4,3}$ & $12=3,\, 13=56=4$  & &  $n^3$, Theorem \ref{thmIntro:Main} \\
\hline 
$\mathscr{L}_{6,25}$ & $12 = 3, 13 = 5, 14 = 6$ & \multirow{2}{*}{10} & $n^4, z = 15$  \\
$\mathscr{L}_{6,23}$ & $12=3,\, 13=24=5, \, 14 = 6$  & & $n^4, z = 15 = 34$  \\
\hline
$\mathscr{L}_{6,9}= \mathscr{L}_{5,9} \times \mathbf R$ & $12 = 3, 13 = 4, 23 = 5$  & \multirow{4}{*}{11} & $n^4$ (product)\\

$\mathscr{L}_{6,24}(1)$ & $12=3,\, 13=26=4, \, 16=23=5$ & & $n^4, z = 15 = 24$  \\

$\mathscr{L}_{6,24}(-1)$ & $12=3, \, 13=26=4, \, 61=23=5$  & & $n^4, z = 15 = 24$  \\

$\mathscr{L}_{6,24}(0)$ & $12=3,\, 13=26=4, \, 23=5$ & & $n^4, z = 15 = 24$  \\
\hline
\end{tabular}
\end{center}
\vspace{.2cm}
\caption{3-step nilpotent Lie algebras of dimension $6$ and their Dehn functions.}
\label{table:dim6-step3}
\end{table}

\begin{table}[H]
\begin{center}
\begin{tabular}{|c|c|c|c|c|c|}
\hline
$\mathfrak{g}$ & Structure &  hdim & $\delta_{\mathfrak{g}}(n)$ & $e_{\mathfrak{g}}$ \\
\hline
$\mathscr{L}_{6,7} = \mathscr{L}_{5,7} \times \mathbf R$ & $12 = 3, 13=4, 14=5$ &  \multirow{5}{*}{12} & $n^5$ (product) & $0$ \\
$\mathscr{L}_{6,6} = \mathscr{L}_{5,6} \times \mathbf R$ & $12 = 3, 13=4, 14=23=5$ &   & $n^5$ (product)& $3/4$ \\
$\mathscr{L}_{6,12} $ & $12 = 3, 13=4, 14=26=5$ &  & $n^5, z = 15 = 36$  & $1/2$  \\
$\mathscr{L}_{6,11} $ & $12 = 3, 13=4, 14=23=26=5$ &  & $n^5, z = 15 = 24 = 36$  & $3/4$ \\
$\mathscr{L}_{6,13} $ & $12 = 3, 13=26=4, 14=36=5$ &  & $n^5, z = 15 = 46$  & $3/4$  \\
\hline
$\mathscr{L}_{6,21}(1)$ & $12 = 3, 13 = 4, 23 = 5, 14 = 6, 25 = 6$ & $14$ & $n^5, z = 16=35$  & $0$  \\
\hline
$\mathscr{L}_{6,21}(-1)$ & $12 = 3, 13 = 4, 23 = 5, 14 = 6, 52 = 6$ &  $14$ & $n^5, z = 16 =53$  & $0$ \\
\hline
$\mathscr{L}_{6,21}(0)$ & $12 = 3, 13 = 4, 23 = 5, 14 = 6$ &  $14$ & $n^5, z = 16$ & $0$ \\
\hline 
\hline
$\mathscr{L}_{6,18}$ & $12=3, 13=4, 14=5, 15=6$ & \multirow{3}{*}{16} & $n^6, z = 16$  & $0$ \\
$\mathscr{L}_{6,17} $ & $12 = 3, 13=4, 14=5, 15=23=6$ &  & $n^6, z = 16 = 24$  & $3/5$ \\
$\mathscr{L}_{6,15} $ & $12 = 3, 13=4, 14=23=5, 15=24=6$ &  & $n^6, z = 16 = 25$ & $4/5$ \\
\hline
$\mathscr{L}_{6,16}$ & $12=3, 13=4, 14=5, 25=43=6$ &  \multirow{2}{*}{16} & $n^5 \preccurlyeq \delta(n) \preccurlyeq n^6$, $z=15$ & $0$ \\
$\mathscr{L}_{6,14} $ & $12 = 3, 13=4, 14=23=5, 25=43=6$ &  & $n^5 \preccurlyeq \delta(n) \preccurlyeq n^6$, $z=15=24$ & $4/5$ \\
\hline
\end{tabular}
\caption{Nilpotent Lie algebras of dimension $6$ and step $\geqslant 4$, and their Dehn functions.}
\label{table:dim6-step45}
\end{center}
\end{table}

\section{Questions and speculations}\label{sec:questions}

We start with a question whose answer would complete the computation of the Dehn functions of all simply connected nilpotent Lie groups of dimension less or equal $6$.

\begin{question}
What are the Dehn functions of the $5$ simply connected nilpotent Lie groups associated to the nilpotent Lie algebras in $\mathcal{L}$ from Lemma \ref{lem:six-dimensional-algebras-where-dcent-is-n^c}?
\end{question}

With the exception of $\mathscr{L}_{6,14}$ all groups corresponding to the Lie algebras in $\mathcal{L}$ are possible candidates for a positive answer to the following question.

\begin{question}
Does there exist a {\it Carnot gradable} simply connected nilpotent Lie group such that its Dehn function and its centralized Dehn function both grow like $n^a$, but with different exponents $a$? 
\end{question}

More generally, we might expect a general picture for  Dehn functions of central products. Let $\mathfrak{k}$ and $\mathfrak{l}$ be nilpotent Lie algebras of step $k$, resp. $\ell$, with $k\geqslant \ell\geqslant 2$, and 1-dimensional centers $\mathfrak{z}$ and $\mathfrak{z}'$. Let $\theta: \mathfrak{z}\to \mathfrak{z}'$ be an isomorphism between their centers and let $\mathfrak{k}\times_{\theta} \mathfrak{l}$ be their central product. We denote by $K$, $L$ and $G:=K\times_{\theta} L$ the associated simply connected Lie groups.

\begin{conjecture}
 The Dehn function of $G$ satisfies $n^{k}\preccurlyeq \delta_{G}(n)\prec n^{k+1}$.
 \label{conj:Dehncentral}
\end{conjecture}

We explain the intuition behind this conjecture. First we observe that the fact that the centers in consideration are 1-dimensional implies that there is still a cocycle $\omega$, defining the $k$-central extension $\mathfrak{k}\to \mathfrak{k}/\mathfrak{z}$. As for our examples $\mathfrak{g}_{p,q}$, this cocycle represents the ``trivial'' central extension $\mathfrak{k}\times \mathfrak{l} \to \mathfrak{k}\times_{\theta} \mathfrak{l}$. It is thus $\ell$-central and, in particular, it will only be $k$-central if $k=\ell$. Moreover, there is no $r$-central extension for $r\geqslant k+1$ (see Lemmas \ref{lem:direct-product-central-extension} and \ref{prop:no-central-extension-max-degree}). Hence, we can at best hope for a lower bound of $n^k$ on the Dehn function of $G$ by using central extensions. On the other hand we can in general not even hope for this, as for $k>\ell$ the form $\omega$ does not provide such an extension and our examples show that no other $k$-central extension might exist. However, it seems reasonable to believe that perturbation arguments similar to the ones developed in \S \ref{sec:lower-bounds-forms} can be used to show that the Dehn function of $G$ is $\succcurlyeq n^k$. This explains our guess for the lower bound. 

For the upper bound the key intuition is that it should still be possible to commute central words $w(X)$ in the generators $X$ of $K$ with arbitrary words $v(X)$ at a lower cost than $n^{k+1}$ by using what we will now call the ``central word trick'': one replaces $w(X)$ by a suitable word $w'(X')$ in the generators $X'$ of $L$ at cost $\prec n^{k+1}$ and then exploits that $\left[X,X'\right]=1$ to commute it with $v$. For the overall approach one should mimic the boot-strapping trick of using an inductive argument on the nilpotency class $k$ that we applied in \S \ref{sec:upper-bound-Dehn} (also see its sketch in the second half of \S \ref{sec:Intro-sketch-MainThm}). 

The basic idea would be to first reduce the word $w(X)$ to a word $u(Y)$, where the letters $Y$ live in a subgroup $H<K$ of nilpotency class strictly lower than $k$ (in our case, $K=L_p$, while $H=L_{p-1}$). Such a $u$ will presumably have length $n^2$. We then assume that the conjecture holds by induction for $k-1$ and apply it in the central product $H\times_{\theta} L$ to commute $w(X)$ with other words in $X$ at cost $\prec n\cdot n^{k}=n^{k+1}$. As we saw in \S \ref{sec:upper-bound-Dehn} this simple trick, used in the right way, is the fundamental reason why our argument works. 

Once we inductively reduced to a 2-step nilpotent central product, we can invoke Olshanskii and Sapir's result that the Dehn function of such a group is bounded above by $n^2 \log(n)$  \cite{OlsSapCombDehn}. This would allow us to conclude. We remark that while they don't say this explicitly, the reason why Olshanskii and Sapir's argument for 2-step nilpotent groups works ultimately also boils down to the central word trick (and we are convinced that the authors were aware of this). However, as we have seen in \S \ref{sec:upper-bound-Dehn} it is far from obvious how to make such an argument work in higher step. There are various reasons for this, for instance, to mention just one of them, making it work requires the reduction step that turns words of length $n$ in $X$ into words of length $n^2$ in a suitable alphabet $Y$ at sufficiently low cost, a step that was not needed for 2-step nilpotent groups.

The fact that already for the specific class of groups $G_{p,p-1}$, whose structure is as simple as one may hope for, the argument turns out to be highly technical, suggests that actually proving Conjecture \ref{conj:Dehncentral} in general will at the very least require the development of a refined version of our methods and potentially even a completely different approach.

Finally it is worth noting that it would even be interesting to prove Conjecture \ref{conj:Dehncentral} for other specific classes of examples. Indeed, well-chosen classes of examples might well produce new groups that satisfy all the main conclusions of our results in the introduction. A first such class to consider would be the general class of groups $G_{p,q}$ for which so far we were only able to determine the precise Dehn function for $q\in\left\{2,p-1,p\right\}$ (see also Remark \ref{rmk:increasing-Dehn-for-descending-q}). 

\begin{question}
What is the Dehn function of $G_{p,q}$ for $3\leqslant q\leqslant p-2$? 
\end{question}
Considering specific classes of examples seems particularly tempting, because, with some real speculation involved, a well-chosen class of examples could potentially produce nilpotent groups with Dehn functions strictly between $n^q$ and $n^{q+1}$ for all integers $q\geqslant 3$, generalising Wenger's examples \cite{Wenger}, or, on an even more speculative note, even nilpotent groups whose Dehn functions do not have integer exponents.

\renewcommand*{\bibliofont}{\small}

\bibliographystyle{alpha}

\bibliography{biblioNilpDehn}

\end{document}